\documentclass{amsart}
\usepackage[utf8]{inputenc}

\usepackage{amssymb}
\usepackage{amsmath}
\usepackage{amsfonts}
\usepackage{tikz-cd}
\usetikzlibrary{arrows}
\usepackage{relsize}

\tikzset{
curvarr/.style={
  to path={ -- ([xshift=2ex]\tikztostart.east)
    |- (#1) [near end]\tikztonodes
    -| ([xshift=-2ex]\tikztotarget.west)
    -- (\tikztotarget)}
  }
}

\tikzset{
  curvedlink/.style={
    to path={
      let \p1=(\tikztostart.east), \p2=(\tikztotarget.west),
      \n1= {abs(\y2-\y1)/4} in
      (\p1) arc(90:-90:\n1) -- ([yshift=2*\n1]\p2) arc (90:270:\n1)
    },
  }
}

\usepackage{amsthm}
\usepackage{enumitem}
\usepackage{mathrsfs}
\usetikzlibrary{decorations.pathmorphing}
\usepackage[utf8]{inputenc}
\usepackage[T1]{fontenc}
\usepackage{indentfirst}
\usepackage{xcolor}
\usepackage{mathtools}
\usepackage{bm}
\usepackage{stmaryrd}
\usepackage{lineno}
\usepackage{hyperref}
\hypersetup{
    colorlinks,
    citecolor=black,
    filecolor=black,
    linkcolor=black,
    urlcolor=black
}

\newtheorem{theorem}{Theorem}[section]
\newtheorem{lemma}[theorem]{Lemma}

\newtheorem{proposition}[theorem]{Proposition}
\newtheorem{corollary}[theorem]{Corollary}

\theoremstyle{definition}
\newtheorem{define}[theorem]{Definition}
\newtheorem{example}[theorem]{Example}

\newtheorem{convention}[theorem]{Convention}

\theoremstyle{remark}
\newtheorem{remark}[theorem]{Remark}

\numberwithin{equation}{section}

\DeclareMathOperator{\Hom}{Hom}
\DeclareMathOperator{\Ext}{Ext}

\newcommand{\D}{\mathscr{D}}

\DeclareMathOperator{\Der}{Der}

\DeclareMathOperator{\ann}{ann}

\DeclareMathOperator{\gr}{gr}

\DeclareMathOperator{\Var}{V}

\DeclareMathOperator{\pdim}{pdim}

\DeclareMathOperator{\Spec}{Spec}

\DeclareMathOperator{\rel}{rel}
\DeclareMathOperator{\Ch}{Ch}
\DeclareMathOperator{\im}{im}
\DeclareMathOperator{\reg}{reg}

\DeclareMathOperator{\qi}{q.i.}
\DeclareMathOperator{\deRham}{DR}

\DeclareMathOperator{\maxdeg}{maxdeg}

\DeclareMathOperator{\red}{red}

\DeclareMathOperator{\id}{id}

\DeclareMathOperator{\UrComplex}{Ur}
\DeclareMathOperator{\Weyl}{\mathbb{A}}
\DeclareMathOperator{\Tot}{Tot}
\DeclareMathOperator{\an}{an}
\DeclareMathOperator{\wdeg}{deg_{\textbf{w}}}

\DeclareMathOperator{\Sing}{Sing}
\DeclareMathOperator{\ord}{ord}
\DeclareMathOperator{\CombRoots}{CombR}

\usepackage [english]{babel}
\usepackage [autostyle, english = american]{csquotes}
\MakeOuterQuote{"}

\author{Daniel Bath}
\title[Bernstein--Sato polynomials: locally quasi-homogeneous divisors in $\mathbb{C}^3$]{Bernstein--Sato polynomials of locally quasi-homogeneous divisors in $\mathbb{C}^{3}$}
\address{Departement Wiskunde, KU Leuven, Celestijnenlaan 200B, 3001 Leuven, Belgium}
\email{dan.bath@kuleuven.be}
\thanks{The author is supported by FWO grant \#12E9623N}
\begin{document}
\sloppy
\maketitle
\begin{abstract}
We consider the Bernstein--Sato polynomial of a locally quasi-homogeneous polynomial $f \in R = \mathbb{C}[x_{1}, x_{2}, x_{3}]$. We construct, in the analytic category, a complex of $\mathscr{D}_{X}[s]$-modules that can be used to compute the $\mathscr{D}_{X}[s]$-dual of $\mathscr{D}_{X}[s] f^{s-1}$ as the middle term of a short exact sequence where the outer terms are well understood. This extends a result by Narv\'{a}ez Macarro where a freeness assumption was required.

We derive many results about the zeroes of the Bernstein--Sato polynomial. First, we prove each nonvanishing degree of the zeroeth local cohomology of the Milnor algebra $H_{\mathfrak{m}}^{0} (R / (\partial f))$ contributes a root to the Bernstein--Sato polynomial, generalizing a result of M. Saito's (where the argument cannot weaken homogeneity to quasi-homogeneity). Second, we prove the zeroes of the Bernstein--Sato polynomial admit a partial symmetry about $-1$, extending a result of Narv\'{a}ez Macarro that again required freeness. We give applications to very small roots, the twisted Logarithmic Comparison Theorem, and more precise statements when $f$ is additionally assumed to be homogeneous.

Finally, when $f$ defines a hyperplane arrangement in $\mathbb{C}^{3}$ we give a complete formula for the zeroes of the Bernstein--Sato polynomial of $f$. We show all zeroes except the candidate root $-2 + (2 / \deg(f))$ are (easily) combinatorially given; we give many equivalent characterizations of when the only non-combinatorial candidate root $-2 + (2/ \deg(f))$ is in fact a zero of the Bernstein--Sato polynomial. One equivalent condition is the nonvanishing of $H_{\mathfrak{m}}^{0}( R / (\partial f))_{\deg(f) - 1}$. 
\end{abstract}

\tableofcontents

\section{Introduction}

Let $X = \mathbb{C}^n$ with its analytic structure sheaf $\mathscr{O}_X$, given by analytic functions on $U \subseteq X$. Also let $R = \mathbb{C}[x_1, \dots, x_n]$ be the polynomial ring of global algebraic functions on $X$. Consider a reduced $f \in R$ with its attached analytic divisor $D$. Our goal is to understand the roots of the \emph{Bernstein--Sato polynomial} $b_{f}(s)$ of $f$ when: $R = \mathbb{C}[x_{1}, x_{2}, x_{3}]$ and $f$ satisfies ``strong'' homogeneity assumptions. That is, when $D$ is analytically defined by a quasi-homogeneous polynomial\footnote{A polynomial $f \in R$ is quasi-homogeneous when there exist positive weights $w_i$ for each $x_i$ such that $f$ is homogeneous with respect to the grading of $R$ induced by these weights.} locally everywhere; in parlance, when $f$ and $D$ are \emph{locally quasi-homogeneous}. (When saying $f$/$D$ are locally quasi-homogeneous, we implicitly assume $f$ itself is a quasi-homogeneous polynomial.)

Our simplification bounds the singular structure of $f$ between two better understood classes of singularities: isolated quasi-homogeneous singularities; \emph{free} locally quasi-homogeneous singularities. Freeness means the module of logarithmic derivations is locally free; geometrically this equates to the singular subscheme of $f$ being Cohen--Macaulay of codimension two (\cite[Proposition 2.1]{MondNotesLogarithmic}. So when $n = 3$, $X = \mathbb{C}^{3}$, our polynomial $f$ is ``close'' to either class of singularities. In these extremal cases the Bernstein--Sato polynomial inherits very nice properties which we will extend to our polynomials. 

On one hand, if $f$ has an isolated singularity, the quasi-homogeneous assumption grants an explicit formula for the zeroes $Z(b_{f}(s))$ of the Bernstein--Sato polynomial. Grade $R$ using the quasi-homogeneity assumption so that $\wdeg(x_{i}) = w_{i}$ and let $\wdeg(f)$ be its weighted degree. Then by \cite{BrianconGrangerMaisonobe}, \cite{BrianconGrangerMaisonobeMiniconi}, \cite{MalgrongeIsolee}, \cite{YanoTheoryBFunctions}, \cite{KochmanBernsteinPolys}:
\begin{equation} \label{eqn - intro - BS poly, isolated singularity}
Z(b_{f}(s)) = \big\{ \frac{- (t + \sum w_{i})}{\wdeg(f)} \mid [R / (\partial f)]_{t} \neq 0 \big\} \cup \{-1\}
\end{equation}
where $(\partial f) = ( \partial_{1} \bullet f, \dots, \partial_{n} \bullet f)$ is the \emph{Jacobian ideal} and $[R / (\partial f)]_{t}$ are the homogeneous weighted degree $t$ elements of the \emph{Milnor algebra} $R / (\partial f)$. To try and realize \eqref{eqn - intro - BS poly, isolated singularity} outside the isolated singularity case, we can and will replace the Milnor algebra with $H_{\mathfrak{m}}^{0}(R / (\partial f))$, its $0^{\text{th}}$-local cohomology with support in the irrelevant ideal $\mathfrak{m}$. 

On the other hand, when $f$ is free and locally quasi-homogeneous, Narv\'{a}ez Macarro \cite[Theorem 4.1]{DualityApproachSymmetryBSpolys} proved the Bernstein--Sato polynomial is symmetric about $-1$. Hence,
\begin{equation} \label{eqn - intro - BS poly, free symmetry}
    \alpha \in Z(b_{f}(s)) \iff -2 - \alpha \in Z(b_{f}(s)).
\end{equation}
To explain his method, recall that the Bernstein--Sato polynomial $b_{f}(s) \in \mathbb{C}[s]$ of $f$ is the minimal, monic polynomial satisfying
\[
b_{f}(s) f^{s} \in \mathscr{D}_{X}[s] f^{s+1},
\]
where $\mathscr{D}_{X}$ is the sheaf of analytic differential operators, $\mathscr{D}_{X}[s]$ is a polynomial ring extension in a new variable $s$, and $\mathscr{D}_{X}[s] f^{s}$ is the cyclic $\mathscr{D}_{X}[s]$-module generated by $f^{s}$ where operators act on $f^{s}$ (or $f^{s+k}$ with $k \in \mathbb{Z}$) by formal application of the chain rule. Alternatively, $b_{f}(s)$ is the monic generator of the $\mathbb{C}[s]$-ideal
\[
\ann_{\mathbb{C}[s]} \frac{ \mathscr{D}_{X}[s] f^{s}}{ \mathscr{D}_{X}[s] f^{s+1}} \subseteq \mathbb{C}[s].
\]

Under the free and locally quasi-homogeneous assumption, Narv\'{a}ez Macarro \cite[Corollary 3.6]{DualityApproachSymmetryBSpolys} computed the $\mathscr{D}_{X}[s]$-dual of $\mathscr{D}_{X}[s] f^{s}$:
\begin{equation} \label{eqn - intro - free, duality fml}
\Ext_{\mathscr{D}_{X}[s]}^{n} (\mathscr{D}_{X}[s] f^{s}, \mathscr{D}_{X}[s])^{\ell} \simeq \mathscr{D}_{X}[s] f^{-s-1}.
\end{equation}
(Here $^{\ell}$ is the side-changing operation.) This lets him compute the $\mathscr{D}_{X}[s]$-dual of $\mathscr{D}_{X}[s] f^{s} / \mathscr{D}_{X}[s] f^{s+1}$ as $\mathscr{D}_{X}[s] f^{-s-2} / \mathscr{D}_{X}[s] f^{-s - 1}$ from whence symmetry of the Bernstein--Sato polynomial falls out. Freeness cannot be weakened in his proof: with it a sort of Spencer complex turns out to be a free $\mathscr{D}_{X}[s]$-resolution of $\mathscr{D}_{X}[s] f^{s}$; without freeness no candidate resolution previously existed.

We show that $f \in \mathbb{C}[x_{1}, x_{2}, x_{3}]$ locally quasi-homogeneous exhibit the above phenomena: as in the isolated case, the non-vanishing degrees of $H_{\mathfrak{m}}^{0}(R / (\partial f))$ \emph{always give} roots of the Bernstein--Sato polynomial; as in the free case, there is a \emph{partial symmetry} of the zeroes of the Bernstein--Sato polynomial. 

The strategy is to compute $\Ext_{\mathscr{D}_{X}[s]}^{3}(\mathscr{D}_{X}[s] f^{s}, \mathscr{D}_{X}[s])^{\ell}$. (Actually we first focus on $\mathscr{D}_{X}[s] f^{s-1}$ for natural reasons, but changing $s-1$ to $s$ is formal, cf. Remark \ref{rmk - formal substitution is ok}). This requires concocting a projective resolution of $\mathscr{D}_{X}[s] f^{s}$ without any freeness hypothesis. Our construction uses \emph{logarithmic data}. More specifically, it is inspired by the twisted logarithmic de Rham complex
\[
0 \to \Omega_{X}^{0}(\log f) \xrightarrow[]{d(-) + \lambda df/f \wedge } \Omega_{X}^{1}(\log f) \to \cdots \to \Omega_{X}^{n}(\log f) \to 0
\]
where the differential is the exterior derivative $d(-)$ twisted by $ \lambda df / f \wedge $ for $\lambda \in \mathbb{C}$. The objects are the \emph{logarithmic differential forms} of K. Saito \cite{SaitoLogarithmicForms}, i.e. forms in $\Omega_{X}^{p}(\star f)$ that both: have pole order at most one along $D$; have this pole order property preserved under exterior differentiation. By M. Saito's theory of differential morphisms \cite{SaitoModulesPolarisables}, \cite{InducedDModsDifferentialComplexes}, there is a $\mathscr{D}_{X}$-module interpretation of this complex as studied in \cite{BathSaitoTLCT}. By replacing the weight $\lambda$ with the indeterminate $s$ we obtain (Definition \ref{def - the Ur complex}) a new complex of \emph{right} $\mathscr{D}_{X}[s]$-modules
\begin{equation}
    U_{f}^{\bullet} = 0 \to \Omega_{X}^{0}(\log f) \otimes_{\mathscr{O}_{X}} \mathscr{D}_{X}[s] \to  \cdots \to \Omega_{X}^{n}(\log f) \otimes_{\mathscr{O}_{X}} \mathscr{D}_{X}[s].
\end{equation} 

Our first main result is Proposition \ref{prop - U complex resolves}:

\begin{proposition} \label{prop - intro - U complex resolves}
    Suppose that $f \in R = \mathbb{C}[x_{1}, \dots, x_{n}]$ defines a tame, Saito-holonomic divisor. Then $U_{f}^{\bullet}$ resolves $(\mathscr{D}_{X}[s] / \ann_{\mathscr{D}_{X}[s]}^{(1)} f^{s-1})^{r}$, that is, the augmented complex
    \[
    U_{f}^{\bullet} \to \left( \frac{\mathscr{D}_{X}[s]}{\ann_{\mathscr{D}_{X}[s]}^{(1)} f^{s-1}} \right)^{r}
    \]
    is acyclic. If in addition $f$ is locally quasi-homogeneous, then $U_{f}^{\bullet}$ resolves $\mathscr{D}_{X}[s]f^{s-1}$. In particular, $U_{f}^{\bullet}$ resolves $\mathscr{D}_{X}[s]f^{s-1}$ when $f \in \mathbb{C}[x_{1}, x_{2}, x_{3}]$ is locally quasi-homogeneous.
\end{proposition} 
Here: $\ann_{\mathscr{D}_{X}[s]}^{(1)} f^{s-1}$ is the $\mathscr{D}_{X}[s]$-ideal generated by the total order $\leq 1$ differential operators killing $f^{s-1}$; tameness is a weak assumption on the projective dimension of logarithmic forms that is automatic in $\mathbb{C}^{3}$, cf. Remark \ref{rmk - tame in three dim}; Saito-holonomic is a weak inductive hypothesis automatically satisfied by locally quasi-homogeneous divisors, cf. Remark \ref{rmk - log stratification}; and $(-)^{r}$ is the side-changing functor. We are crucially using Walther's \cite[Theorem 3.26]{uli} guaranteeing that, in our setting, $\ann_{\mathscr{D}_{X}[s]}^{(1)} f^{s-1} = \ann_{\mathscr{D}_{X}[s]} f^{s-1}$. 

Of course, the objects of $U_{f}^{\bullet}$ need not be projective right $\mathscr{D}_{X}[s]$-modules. But when $U_{f}^{\bullet}$ is a resolution, we can replace each object by its own projective resolution, take the consequence double complex, and acquire a projective right $\mathscr{D}_{X}[s]$-resolution of $(\mathscr{D}_{X}[s] f^{s-1})^{r}$. By dualizing this double complex and studying a resultant spectral sequence, we can generalize Narv\'{a}ez Macarro's duality formula \cite[Corollary 3.6]{DualityApproachSymmetryBSpolys}, cf. \eqref{eqn - intro - free, duality fml}. Here we see both symmetry properties and Milnor algebra local cohomology data appearing simultaneously:

\begin{theorem} \label{thm - intro - spectral sequence dual}
    Let $f \in R = \mathbb{C}[x_{1}, x_{2}, x_{3}]$ be reduced and locally quasi-homogeneous. Then we have a short exact sequence of left $\mathscr{D}_{X}[s]$-modules
    \begin{equation} \label{eqn - intro - thm, spectral sequence dual, statement, ses}
        0 \to \mathscr{D}_{X}[s]f^{-s} \to \Ext_{\mathscr{D}_{X}[s]}^{3} \left( \mathscr{D}_{X}[s] f^{s-1}, \mathscr{D}_{X}[s] \right)^{\ell} \to L_{f} \to 0,
    \end{equation}
    where $L_{f}$ is the left $\mathscr{D}_{X}[s]$-module supported at $0$, or is itself $0$, cf. Definition \ref{def - the cokernel, ext dual object}. 
    
    Moreover, $L_{f}$ has a $b$-function, namely
    \begin{equation} \label{eqn - intro - thm, spectral sequence dual, statement, funny b-function}
        B(L_{f}) = \prod_{t \in \wdeg(H_{\mathfrak{m}}^{0}(R / (\partial f)))} \left( s - \frac{ - t + 2 \wdeg(f) - \sum w_{i}}{\wdeg(f)} \right) 
    \end{equation}
    (If $\wdeg(H_{\mathfrak{m}}^{3}(R / (\partial f))) = \emptyset$ we take $B(L_{f}) = 1$ as this equates to $L_{f} = 0.)$ 

\end{theorem}

\noindent Here: $b$-function means (monic generator of the) $\mathbb{C}[s]$-annihilator (Definition \ref{def - b-function module}); the left $\mathscr{D}_{X}[s]$-module $L_{f}$ is explicitly constructed in Definition \ref{def - the cokernel, ext dual object}. Alternatively, $L_{f}$ is the cokernel of a $\mathscr{D}_{X}[s]$-endomorphism of $\mathscr{D}_{X}[s] \otimes_{\mathscr{O}_{X}} \Ext_{\mathscr{O}_{X}}^{3}(\mathscr{O}_{X} / (\partial f), \mathscr{O}_{X})$, cf. Remark \ref{rmk - cokernel object, endomorphism description}. The main difficulty of the paper is in proving Theorem \ref{thm - intro - spectral sequence dual}: it requires constructing nice lifts of maps along resolutions and a direct computation of the $\mathbb{C}[s]$-annihilator of $L_{f}$.

To realize local cohomology data of the Milnor algebra within the analytic duality computation of Theorem \ref{thm - intro - spectral sequence dual}, we must transfer algebraic data into the analytic context. Consequently, the manuscript sometimes switches between the algebraic and analytic, passing information from one to the other. Beginnings of relevant sections and subsections explain what settings we are working in and why, see also Remark \ref{rmk - why analytic, free outside 0}, Remark \ref{rmk - why analytic, graded Betti}, Remark \ref{rmk - why analtyic, U complex}, Remark \ref{rmk - why analytic, total complex}, and Remark \ref{rmk- why alg then anal lifts}. 

By Lemma \ref{lemma - b-function and b-function of dual}, the zeroes of the Bernstein--Sato polynomial are those of the $b$-function $(\mathbb{C}[s])$-annihilator) of the $\mathscr{D}_{X}[s]$-dual of $\mathscr{D}_{X}[s] f^{s} / \mathscr{D}_{X}[s] f^{s+1}$. Here we find degree data of $H_{\mathfrak{m}}^{0} ( R / (\partial f))$ contributing zeroes of the Bernstein--Sato polynomial, exactly as in the isolated singularity case. Concretely, in Theorem \ref{thm - new roots of BS-poly} we prove:

\begin{theorem} \label{thm - intro - new roots of BS-poly}
    Let $f \in R = \mathbb{C}[x_{1}, x_{2}, x_{3}]$ be reduced and locally quasi-homogeneous. Then we have a surjection of left $\mathscr{D}_{X}[s]$-modules
    \begin{equation} \label{eqn - intro - new roots of BS-poly, statement 1}
        \Ext_{\mathscr{D}_{X}[s]}^{4} \left( \frac{\mathscr{D}_{X}[s] f^{s}}{\mathscr{D}_{X}[s] f^{s+1}}, \mathscr{D}_{X}[s] \right) ^{\ell} \twoheadrightarrow L_{f, s+2}
    \end{equation}
    where $L_{f, s+2}$ is explicitly defined in Definition \ref{def - twisted cokernel object}/Definition \ref{def - the cokernel, ext dual object}. In particular, the zeroes of the $b$-function of $L_{f, s+2}$ are contained in the zeroes of the Bernstein--Sato polynomial of $f$:
    \begin{equation} \label{eqn - intro - new roots of BS-poly, statement 2}
        Z(b_{f}(s)) \supseteq Z(B(L_{f, s+2})) = \left\{ \frac{ - (t + \sum w_{i}) }{\wdeg(f)} \mid t \in \wdeg(H_{\mathfrak{m}}^{0} (R / (\partial f))) \right\}.
    \end{equation}

\end{theorem}

Experts will note that if $\Var(f)$ is the affine cone of $Z \subseteq \mathbb{P}^{2}$ where $Z$ only has isolated quasi-homogeneous singularities, then \eqref{eqn - intro - new roots of BS-poly, statement 2} recovers M. Saito's \cite[Corollary 3]{SaitoBSProjectiveWeightedHomogeneous}.  However, his methods depend on blow-up techniques and do not seem to speak to the generalization to locally quasi-homogeneous divisors, nor do they speak to our other results.

In the free case, we have symmetry of the Bernstein--Sato polynomial. Again we are able to derive a partial symmetry in our case. Interestingly, the obstruction to symmetry is determined exactly by the new roots of $b_{f}(s)$ found in Theorem \ref{thm - intro - new roots of BS-poly}. Theorem \ref{thm - localized symmetry of BS poly} gives our version of \eqref{eqn - intro - BS poly, free symmetry}:

\begin{theorem} \label{thm - intro - localized symmetry of BS poly}
    Let $f \in R = \mathbb{C}[x_{1}, x_{2}, x_{3}]$ be reduced and locally quasi-homogeneous. Define
    \begin{equation} \label{eqn - thm - intro - localized symmetry of BS poly, statement 1}
        \Xi_{f} = \bigcup_{t \in \wdeg(H_{\mathfrak{m}}^{0} (R / (\partial f)))} \bigg\{ \frac{-(t + \sum w_{i})}{\wdeg(f)} \bigg\} \cup \bigg\{ \frac{-(t + \sum w_{i}) + \wdeg(f)}{\wdeg(f)}  \bigg\}.
    \end{equation}
    Then we have partial symmetry of the Bernstein--Sato polynomial's roots about $-1$:
    \begin{equation} \label{eqn - thm - intro - localized symmetry of BS poly, statement 2}
    \alpha \in \bigg( Z(b_{f}(s)) \setminus \Xi_{f} \bigg) \iff - \alpha - 2 \in \bigg( Z(b_{f}(s)) \setminus \Xi_{f} \bigg).
    \end{equation}
\end{theorem}
Using Theorem \ref{thm - intro - new roots of BS-poly} and Theorem \ref{thm - intro - localized symmetry of BS poly} we compute $Z(b_{f}(s)) \cap (-3,-2]$ explicitly (Corollary \ref{cor - small roots BS poly}) and characterize exactly the weights $\lambda \in (-\infty, 0]$ for which a twisted Logarithmic Comparison Theorem holds (Corollary \ref{cor - twisted LCT}). Both Corollaries state their results in terms of Milnor algebra data.

If we restrict to $f$ homogeneous and locally quasi-homogeneous, then $H_{\mathfrak{m}}^{0}(R / (\partial f))$ enjoys nice degree properties: symmetry of its non-vanishing degrees about $(3 \deg(f) - 6)/2$ \cite[Theorem 4.7]{StratenDucoWarmtGorensteinDuality}, \cite[Theorem 3.4]{LocalCohomologyJacobianRing}; unimodality of its degree sequence \cite[Theorem 4.1]{LefschetzAlmostCompleteIntersections}, \cite[Section 2]{BrennerKaidSyzygyBundles}. Consequently we can prove the zeroes of the Bernstein--Sato polynomial can be reconstructed entirely from the minimal degree for which $H_{\mathfrak{m}}^{0}(R / (\partial f))$ does not vanish, $\deg(f)$, and the zeroes in $(-1,0)$. Theorem \ref{thm - homogeneous, BS poly description} says

\begin{theorem} \label{thm - intro - homogeneous, BS poly description}
    Let $f \in R = \mathbb{C}[x_{1}, x_{2}, x_{3}]$ be reduced, homogeneous, and locally quasi-homogeneous. With $\tau = \min \{t \mid t \in \deg H_{\mathfrak{m}}^{0} (R / (\partial f))\}$, set 
    \begin{equation*}
        \Upsilon_{f} = \frac{1}{\deg(f)} \cdot \left( \mathbb{Z} \cap [- 3 \deg(f) + \tau + 3, - (\tau + 3)] \right).
    \end{equation*}
    Moreover, let $\sigma : \mathbb{R} \to \mathbb{R}$ the involution $\alpha \mapsto - 2 - \alpha$ about $-1$. Then
    \begin{enumerate}[label=(\alph*)]
        \item $ \Upsilon_{f} \subseteq Z(b_{f}(s))$
        \item $\beta \in Z(b_{f}(s)) \setminus (\Upsilon_{f} \cup \sigma(\Upsilon_{f})) \iff \sigma (\beta) \in Z(b_{f}(s)) \setminus (\Upsilon_{f} \cup \sigma(\Upsilon_{f})) $
    \end{enumerate}
   And we have the taxonomy:
    \begin{align} \label{eqn - thm - intro - homogeneous, BS poly description, statement 2}
        \beta \in Z(b_{f}(s)) \cap (-3, -2] \iff &\beta \in \Upsilon_{f} \cap (-3,-2]; \\
        \beta \in Z(b_{f}(s)) \cap (-2, -1) \iff &\beta \in \Upsilon_{f} \cap (-2, -1) \nonumber \\
        \enspace &\text{ and/or } \sigma(\beta) \in Z(b_{f}(s)) \cap (-1, 0); \nonumber \\
        \beta \in Z(b_{f}(s)) \cap [-1, 0) \implies &\sigma(\beta) \in Z(b_{f}(s)) \cap (-2,-1]. \nonumber 
    \end{align}
     Hence $Z(b_{f}(s))$ is determined by
     \[
     \tau, \deg(f), \text{ and } \big[ Z(b_{f}(s)) \cap [-1,0) \big] \setminus \Upsilon_{f} \quad  (\Upsilon_{f} \text{ is determined by $\tau$ and $\deg(f)$}.)
     \]
\end{theorem}

We finish with a \emph{complete} formula for the zeroes of the Bernstein--Sato polynomial of hyperplane arrangements in $\mathbb{C}^{3}$. These zeroes are not combinatorial, i.e. they cannot be determined from the arrangement's intersection lattice alone. The only/first example comes from a well studied pathology: Ziegler's pair of arrangements \cite{ZieglerCombinatorialConstructionLogForms} are two arrangements $h$ and $g$ in $\mathbb{C}^{3}$ with the same combinatorics but where $H_{\mathfrak{m}}^{0}(R / \partial h)_{\deg(h) - 1} \neq 0$ and $H_{\mathfrak{m}}^{0}(R / (\partial g))_{\deg(g) - 1} = 0$. They arise from whether or not the triple points (viewed in $\mathbb{P}^{2}$) do or do not lie on a quadric. Walther \cite[Example 5.10]{uli} showed $(-2 \deg(h) + 2)/\deg(h)$ is in $Z(b_{h}(s))$ but not $Z(b_{g}(s))$.

Our formula shows this is the only type of non-combinatorial behavior possible. We prove every root but $(-2 \deg(f) + 2)/\deg(f)$ is (easily) combinatorially given; $(-2 \deg(f) + 2)/\deg(f)$ is not combinatorial, and whether or not $(-2 \deg(f) + 2)/\deg(f) \in Z(b_{f}(s))$ is determined by a list of equivalent properties including the non-vanishing of $H_{\mathfrak{m}}^{0}(R / (\partial f))_{\deg(f) - 1}$. Theorem \ref{thm - BS poly arrangement} is:

\begin{theorem} \label{thm - intro - BS poly arrangement}
    Let $f \in R = \mathbb{C}[x_{1}, x_{2}, x_{3}]$ define a reduced, central, essential, and indecomposable hyperplane arrangement; let $Z \subseteq \mathbb{P}^{2}$ be its projectivization. For every point $z \in \Sing(Z)$, set $m_{z}$ to be the number of lines in $Z$ containing $z$. Define
    \[
    \CombRoots = \left[ \bigcup_{3 \leq k \leq 2 \deg(f) - 3}  \frac{-k}{\deg(f)}  \right] \cup \left[ \bigcup_{z \in \Sing(Z)} \quad \bigcup_{2 \leq i \leq 2m_{z} - 2}  \frac{-i}{m_{z}}  \right].
    \]
    Then
    \begin{equation} \label{eqn - thm - intro - BS poly arrangement, statement 1}
        Z(b_{f}(s)) = \CombRoots \quad \text{OR} \quad Z(b_{f}(s)) = \CombRoots \enspace \cup \enspace \frac{-2\deg(f) + 2}{\deg(f)}.
    \end{equation}
    That is, $Z(b_{f}(s))$ always contains the combinatorially determined roots $\CombRoots$ and at most one non-combinatorial root $\frac{-2\deg(f) + 2}{\deg(f)}$. 
    
    The presence of the non-combinatorial root is characterized by the following equivalent properties, where $\widetilde{\frac{R}{(\partial f)}}$ is the graded sheafification of $R / (\partial f)$):
    \begin{enumerate}[label=(\alph*)]
        \item $Z(b_{f}(s)) \ni \frac{-2 \deg(f) + 2}{\deg(f)}$;
        \item $[H_{\mathfrak{m}}^{0}(R / (\partial f)]_{\deg(f) - 1} \neq 0$;
        \item $[H_{\mathfrak{m}}^{0}(R / (\partial f))]_{2 \deg(f) - 5} \neq 0$;
        \item $\reg R / (\partial f) = 2 \deg(f) - 5$;
        \item there is the bound on the number of global sections twisted by $2 \deg(f) - 5$:
        \[
        \dim_{\mathbb{C}} \bigg( \Gamma \big( \mathbb{P}^{2}, \widetilde{ \frac{R}{ (\partial f)}} (2 \deg(f) - 5) \big) \bigg) < \dim_{\mathbb{C}} \bigg( [ R / (\partial f)]_{2d - 5} \bigg);
        \]
        \item there is the bound on the number of global sections twisted by $\deg(f) - 1$:
        \[
        \dim_{\mathbb{C}} \bigg( \Gamma \big( \mathbb{P}^{2}, \widetilde{ \frac{R}{ (\partial f)}} (\deg(f) - 1) \big) \bigg) < \dim_{\mathbb{C}} \bigg( [\Der_{R}(-\log_{0} f)]_{\deg(f) -2} \bigg) + \bigg(\binom{\deg(f) + 1}{2} - 3 \bigg) 
        \]
        \item the hyperplane arrangement $\Var(f)$ is not formal. 
    \end{enumerate}
\end{theorem}

\noindent Here: (d) refers to Castelnuovo--Mumford regularity; (f) refers to the logarithmic derivations (Definition \ref{def - log derivations}); the formality property (g) (Definition \ref{def - formality}) is well-studied in arrangement theory, cf. \cite{FalkRandellHomotopyTheoryArrangements}, \cite{BrandtTeraoFreeArrangementsRelationSpaces}, \cite{TohvaneanuTopologicalFormalArrangements}, \cite{DipasqualeSidmanTravesGeometricAspectsJacobian}. Note that Theorem \ref{thm - intro - BS poly arrangement} answers \cite[Question 8.8]{DipasqualeSidmanTravesGeometricAspectsJacobian}. 

\vspace{5mm}

We equip the reader with a roadmap. Section 2 is mostly review of logarithmic data in the case of $\mathbb{C}^{3}$-divisors as well as translating this data into Milnor Algebra data. Section 3 introduces the $\mathscr{D}_{X}[s]$-module constructions we need, discussing duality, side-changing, as well as the Bernstein--Sato polynomial. Preliminaries done, we begin our quest to compute the $\mathscr{D}_{X}[s]$-dual of $\mathscr{D}_{X}[s] f^{s}$ for our class of divisors. In Section 4 we define (Definition \ref{def - the Ur complex}) the complex $U_{f}^{\bullet}$, show (Proposition \ref{prop - U complex resolves}) it is a (not necessarily projective) resolution of $\mathscr{D}_{X}[s] f^{s}$, construct a free resolution (Proposition \ref{prop - total complex free resolution}), and begin the spectral sequence study of the resultant dual complex (Lemma \ref{lem - spectral sequence dual, page 2}). This free resolution involves choosing lifts of certain maps; Section 5 is devoted to constructing nice choices of lifts. In Section 6 we take one of these nice lifts, dualize, and compute a certain cokernel as well as its $b$-function. This is the aforementioned $L_{f}$ module (Definition \ref{def - the cokernel, ext dual object}, Proposition \ref{prop - funny b-function}.) In Section 7 we use all this set-up to compute the $\mathscr{D}_{X}[s]$-dual of $\mathscr{D}_{X}[s] f^{s}$ by our spectral sequence approach. Then we derive the Theorems advertised in the Introduction.

\vspace{5mm}

We would like to thank Guillem Blanco, Nero Budur, Yairon Cid Ruiz, Andreas Hohl, Luis Narv{\'a}ez Macarro, and Uli Walther for all the helpful conversations and insight they provided. We also thank the referee for their thorough comments, which greatly improved the quality of the text.

\section{Homogeneity Assumptions and Logarithmic Data}

In this preliminary section we introduce basic assumptions on our polynomials as well as remind the reader about logarithmic derivations and logarithmic forms. Special attention is paid to the case of divisors in $\mathbb{C}^{3}$: here logarithmic data is particularly simple. We eventually restrict to $\mathbb{C}^{3}$. In general, $f \in R = \mathbb{C}[x_{1}, \dots, x_{n}]$ is nonconstant, $X = \mathbb{C}^{n}$, and $\mathscr{O}_{X}$ is the analytic structure sheaf. When speaking of the divisor $D$ of $f$, we mean the analytic divisor $D \subseteq X$. We differentiate between $f$ or $D$ and the reduced version $f_{\red}$ or $D_{\red}$.

At times we switch between the analytic and algebraic setting, setting up tools for each side-by-side. The eventual philosophy is to translate data of minimal graded free resolutions into the analytic realm. Because locally quasi-homogeneous polynomials (Definition \ref{def - positively weighted homogeneous loc everywhere}) are really an analytic notion, by the end of the manuscript (for example Sections 4, 6, 7) we must work entirely analytically. In the beginning of subsection 2.2, 2.3, and 2.4 we explain whether we are passing algebraic data to analytic data, or the reverse. See also Remark \ref{rmk - why analytic, free outside 0} and Remark \ref{rmk - why analytic, graded Betti}.

\subsection{Homogeneity Assumptions}
We are interested in studying the singular structure of polynomials $f \in R = \mathbb{C}[x_{1}, \dots, x_{n}]$ with strong homogeneity assumptions. First, a generalization of homogeneity:

\begin{define} \label{def - positively weighted homogeneous}
    We say $f \in R = \mathbb{C}[x_{1}, \dots, x_{n}]$ is \emph{quasi-homogeneous}, with weights $\mathbf{w} = \{w_{i}\}_{1 \leq i \leq n}$ when $f$ is non-constant and
    \begin{enumerate}[label=(\roman*)]
        \item $\mathbf{w} \in \mathbb{R}_{>0}^{n}$, i.e. each weight satisfies $w_{i} \in \mathbb{R}_{>0}$;
        \item if $f = \sum_{\mathbf{u}} \alpha_{\textbf{u}} x^{\mathbf{u}}$, where we use multi-index notation, $\textbf{u}$ runs over the monomial support of $f$, and $\alpha_{\mathbf{u}} \in \mathbb{C}^{\star}$, then $\sum_{u_{i} \in \mathbf{u}} w_{i} u_{i} = \sum_{v_{i} \in \mathbf{v}} w_{i} v_{i}$ for vectors $\mathbf{u}, \mathbf{v}$ in the monomial support of $f$.
    \end{enumerate}
    In this case, we say $E = \sum w_{i} x_{i} \partial_{i}$ is the \emph{weighted-homogeneity} of $f$ with respect to the weights $\{w_{i}\}$; we also define the \emph{weighted degree} or $\mathbf{w}$-degree $f$ to be
    \[
    \wdeg(f) = \sum_{u_{k} \in \mathbf{u}} w_{k} u_{k} \text{, for any $\mathbf{u} \in \mathbb{R}_{\geq 0}^{n}$ in the monomial support of } f.
    \]
    An analytic divisor $D \subseteq X$ is quasi-homogeneous at $\mathfrak{x} \in X$ when we may pick a defining equation $f$ of $D$ at $\mathfrak{x}$ such that $f$ is quasi-homogeneous.
\end{define}

\begin{remark} \label{rmk - positively weighted homogeneous} \emph{(Basics about quasi-homogeneous polynomials)}

    \noindent \begin{enumerate}[label=(\alph*)]
    \item A polynomial being homogeneous is the same as it being quasi-homogeneous with respect to the weights $\textbf{w} = \textbf{1}$. In this case we call $E = \sum x_{i} \partial_{i}$ the \emph{Euler derivation} of $f$ and we use $\deg(f)$ instead of $\wdeg(f)$.
    \item If $f$ is quasi-homogeneous with weighted-homogeneity $E$, then $E \bullet f = \wdeg(f) f$.
    \item Quasi-homogeneity is the same as saying: the Newton polytope of $f$ is degenerate and is contained in a hyperplane whose normal vector can be chosen to have strictly positive entries.
    \end{enumerate}
\end{remark}

Because we will be working with quasi-homogeneous divisors/polynomials, we will also be working with non-standard graded polynomial rings $R$. We briefly touch on the key features of such a grading, following the treatment in \cite[Section 8]{CombinatorialCommutativeAlgebra}. 

\begin{define} \label{def - grading}
    Suppose that $f \in R = \mathbb{C}[x_{1}, \dots, x_{n}]$ is quasi-homogeneous with respect to the weights $\mathbf{w} = \{w_{i}\}_{1 \leq i \leq n} \in \mathbb{R}_{>0}^{n}$. We endow $R$ with a \emph{multigrading} by defining a monomial $\alpha_{\textbf{u}} x^{\textbf{u}} \neq 0$ to have \emph{weighted degree} or \emph{$\mathbf{w}$-degree}
\begin{equation} \label{eqn - multigrading, weighted degree}
\wdeg(\alpha_{\textbf{u}} x^{\textbf{u}}) = \sum_{u_{i} \in \textbf{u}} w_{i} u_{i} \in \mathbb{R}_{> 0}.
\end{equation}
    We give elements of $\mathbb{C}$ weighted degree $0$. Such a grading of $R$ is a \emph{weighted grading} or $\mathbf{w}$-grading of $R$, with respect to the weights $\mathbf{w} = \{w_{i}\}$. \emph{Whenever} working with a quasi-homogeneous polynomial/divisor, the graded structure on $R$ is the one induced by $\mathbf{w}$; similarly for any graded $R$-module $M$.

    For a $M$ a $\mathbf{w}$-graded $R$-module: $M_{a} = \{m \in M \mid m \text{ homogeneous,} \wdeg(m) = a \}$; $M(b)$ is the $\mathbf{w}$-graded $R$-module whose underlying module is $M$ and whose grading is defined by $M(b)_{a} = M_{b + a}$. Additionally, we define
\begin{equation} \label{eqn - weighted nonzero degrees of module}
    \wdeg(M) = \{a \in \mathbb{R} \mid M_{a} \setminus 0 \neq \emptyset \}.
\end{equation}
\end{define}

By the positivity assumption on the weights, the only elements of $R$ of weighted degree $0$ are constants. So by \cite[Theorem 8.6, Definition 8.7]{CombinatorialCommutativeAlgebra}, this weighted grading is a \emph{positive multigrading} of $R$. Because graded Nakayama's Lemma still holds in the positive multigraded case (\cite[page 155]{CombinatorialCommutativeAlgebra}), if $M$ is a finitely generated graded $R$ module (with respect to our weighted grading) then $M$ has a minimal graded free resolution with respect to this weighted grading (\cite[Proposition 8.18, Section 8.3]{CombinatorialCommutativeAlgebra}). Later we will make use of these minimal graded free resolutions (e.g. Proposition \ref{prop - graded free resolution of log 1, 2 forms}, Definition \ref{def - fixing minimal graded free res, log 1, 2 forms}) and the consequent fact that the multigraded Betti-numbers are well-defined \cite[Definition 8.22]{CombinatorialCommutativeAlgebra} (e.g. Proposition \ref{prop - computing analytic lifts, dualizing}). We also will use graded local duality in this setting, cf. subsection 2.4.

Finally, we consider an additional homogeneity condition on polynomials stronger than quasi-homogeneity.

\begin{define} \label{def - positively weighted homogeneous loc everywhere}
    Let $f = R = \mathbb{C}[x_{1}, \dots, x_{n}]$ and $D \subseteq X$ the attached analytic divisor. We say $f$ or $D$ is \emph{quasi-homogeneous at} $\mathfrak{x} \in D$ if there exist local coordinates $\{x_{i}\}$ on some analytic open $\mathfrak{x} \in U \subseteq X$ and a quasi-homogeneous polynomial $g$ with respect to these coordinates, such that the germ of $\Var(f)$ at $\mathfrak{x}$ is the same as the germ of $\Var(g)$ at $\mathfrak{x}$. We say $f$ or $D$ is \emph{locally quasi-homogeneous} if it is quasi-homogeneous at \emph{all} $\mathfrak{x} \in \Var(f)$. 

    \emph{Whenever} we say $f \in R = \mathbb{C}[x_{1}, \dots, x_{n}]$ is locally quasi-homogeneous we mean $f$ is quasi-homogeneous \emph{and} $f$/$D$ are locally quasi-homogeneous, i.e. we are implicitly assuming $f \in R$ is itself a quasi-homogeneous representative of the divisor germ at $0$.
\end{define}

\begin{remark} \label{rmk - basics pos weighted hom, loc everywhere} \emph{(Basics on locally quasi-homogeneous)}

    \noindent \begin{enumerate}[label=(\alph*)]
        \item Any divisor $D$ is quasi-homogeneous at any $\mathfrak{x} \in D_{\reg}$.
        \item Hyperplane arrangements are locally quasi-homogeneous.
        \item Suppose $Z \subseteq \mathbb{P}^{n-1}$ has isolated, quasi-homogeneous singularities. Then the affine cone $C(Z) \subseteq \mathbb{C}^{n}$ is locally quasi-homogeneous. Away from $0 \in \mathbb{C}^{n}$, de-homogeneization demonstrates this.
        \item Locally quasi-homogeneous divisors are Saito-holonomic, e.g. \cite[Lemma 1.6]{BathSaitoTLCT}.
        \item A quasi-homogeneous polynomial need not be locally quasi-homogeneous. For example, the homogeneous $f = x^5 z + x^3 y^3 + y^5 z \in \mathbb{C}[x,y,z]$ is not quasi-homogeneous on the punctured $z$-axis $\{(0, 0, z) \in \mathbb{C}^3 \mid z \neq 0\}$, as showing immediately following \cite[Definition 1.3]{CohomologyComplementFreeDivisor}. Here is an alternative argument. Along the punctured $z$-axis, and after a change of coordinates (Remark \ref{rmk - basics pos weighted hom, loc everywhere}.(c)), the germ admits a defining equation $a^5 + a^3 b^3 + b^5 \in \mathbb{C}[a, b, c]$. The hypersurface $\Var(a^5 + a^3 b^3 + b^5) \subseteq \mathbb{C}^2$ is not quasi-homogeneous at $0$, since its Milnor and Tjuriana numbers differ, cf. \cite{KSaitoQuasihomogene}. So \cite[Proposition 2.4]{CohomologyComplementFreeDivisor} implies that $f$ itself is not quasi-homogeneous along the punctured $z$-axis.
        \item The literature sometimes refers to ``locally quasi-homogeneous'' by the names ``strongly quasi-homogeneous'' (\cite{CohomologyComplementFreeDivisor}) or ``positively weighted homogeneous locally everywhere'' (\cite{BathSaitoTLCT}). Definition \ref{def - positively weighted homogeneous loc everywhere} is a stronger form of strongly Euler-homogeneous (locally everywhere), cf. \cite{DualityApproachSymmetryBSpolys}, \cite{uli}.
    \end{enumerate}
\end{remark}

\begin{convention} \label{convention - main hypotheses}
    Whenever we refer to our ``main hypotheses'' on a polynomial $f$, we are assuming $f \in R = \mathbb{C}[x_{1}, x_{2}, x_{3}]$ is reduced and locally quasi-homogeneous. Also, whenever $f$ is honestly homogeneous ($\mathbf{w} = \mathbf{1}$) we substitute $\deg$ for $\wdeg$.
\end{convention}

\subsection{Logarithmic Derivations}

Our analysis of the singularities of polynomials depends on: logarithmic data; transferring between the algebraic and analytic setting. We first introduce \emph{logarithmic derivations}. These are the sygyzies of $\{\partial_{1} \bullet f, \dots, \partial_{n} \bullet f, f \}$. They induce a \emph{logarithmic stratification} we need later, cf. Remark \ref{rmk - log stratification}. We end with Proposition \ref{prop - weighted homogeneous, pdim at most 1, new}: this describes, under the used assumptions, the highest nonvanishing $\Ext$ modules of the logarithmic derivations in the algebraic category.

Note that locally quasi-homogeneous divisors are defined by the existence of certain ``nice'' analytic local coordinates locally everywhere. Hence, there is subtlety in translating ``obvious'' analytic results about locally quasi-homogeneous polynomials into the algebraic category. Proposition \ref{prop - weighted homogeneous, pdim at most 1, new}, one of the main goals of this subsection, exemplifies this by establishing an algebraic result that while immediate in the analytic set-up, demands a strategically different proof.

\begin{define} \label{def - log derivations}
    Fix $f \in R = \mathbb{C}[x_{1}, \dots, x_{n}]$. The $R$-module of \emph{(global algebraic) logarithmic derivations} along $f$ is 
    \[
    \Der_{R}(-\log f) = \{ \delta \in \Der_{R} \mid \delta \bullet f \in R \cdot (f) \}.
    \]
    We say $f$ is \emph{free} when $\Der_{R}(-\log f)$ is a free $R$-module. Also define
    \[
    \Der_{R}(-\log_{0} f) = \{ \delta \in \Der_{R}(-\log f) \mid \delta \bullet f = 0\}.
    \]
    The $\mathscr{O}_{X}$-module of \emph{(analytic) logarithmic derivations} along $f$ is
    \[
    \Der_{\mathscr{O}_{X}}(-\log f) = \{ \delta \in \Der_{X} \mid \delta \bullet f \in \mathscr{O}_{X} \cdot (f) \}.
    \]
    In the analytic setting we say $f$ is \emph{free} when $\Der_{\mathscr{O}_{X, \mathfrak{x}}}(-\log f)$ is a free $\mathscr{O}_{X,\mathfrak{x}}$-module for all $\mathfrak{x} \in X$. We define the analytic $\Der_{\mathscr{O}_{X,\mathfrak{x}}}(\log_{0} f)$ symmetrically.
\end{define}

\begin{remark} \label{rmk - grading log der} \emph{(Grading $\Der_{R}(\log f)$)} If $f \in R$ is quasi-homogeneous with weights $\mathbf{w}$, then $\Der_{R}(-\log f)$ and $\Der_{R}(-\log_{0} f)$ are also $\mathbf{w}$-graded modules. Indeed, $\partial_{i} \bullet f$ is zero or homogeneous of $\textbf{w}$-degree $\wdeg(f) - w_{i}$. Our convention is to grade $\Der_{R}(-\log f) \subseteq \Der_{R}$ so that $\partial_{x_{i}}$ has weight $-w_{i}$, i.e. $\Der_{R} = R(w_{1}) \oplus \cdots \oplus R(w_{n})$. For example, the weighted-homogeneity $E = \sum_{i} w_{i} x_{i} \partial_{i}$ is homogeneous of degree $0$. Both $\Der_{R}(-\log f)$ and $\Der_{R}(-\log_{0} f)$ inherit this convention.
\end{remark}

\begin{remark} \label{rmk - basics algebraic log derivations} \emph{(Basics on Logarithmic Derivations)}
\noindent
    \begin{enumerate}[label=(\alph*)]
        \item As $f$ and all its partial derivatives are polynomials, one obtains $\Der_{\mathscr{O}_{X}}(-\log f)$ by: starting with $\Der_{R}(-\log f)$; algebraically sheafifying to $\widetilde{\Der_{R}(-\log f)}$; applying the analytification functor.
        \item By the product rule, both $\Der_{R}(-\log f)$ and $\Der_{\mathscr{O}_{X}}(-\log f)$ are independent of choice of defining equation of the divisor of $f$.
        \item Suppose that $f$ is quasi-homogeneous with $E$ its weighted-homogeneity. Then $\Der_{R}(-\log_{0} f) \oplus R \cdot E.$ This direct sum decomposition and $\Der_{R}(-\log_{0} f)$ may depend on choice of defining equation of the divisor $D$ of $f$, cf. \cite[Remark 2.10]{uli}.
        \item Suppose that $f \in R = \mathbb{C}[x_{1}, x_{2}, x_{3}]$. We know that $\Der_{R}(-\log f)$ is a reflexive $R$-module, cf. \cite[pg 268]{SaitoLogarithmicForms}. This implies its $R$-depth is $\geq 2$ which by Auslander-Buchsbaum means $\pdim \Der_{R}(-\log f) \leq 1$. The same is true for $\Der_{\mathscr{O}_{X}}(-\log f)$: by loc. cit. for any $\mathfrak{x} \in X$, the stalk is a reflexive $\mathscr{O}_{X, \mathfrak{x}}$-module and hence $\pdim \Der_{\mathscr{O}_{X,\mathfrak{x}}}(-\log f) \leq 1$. 
        \item The preceding item implies that if $g \in A = \mathbb{C}[x_{1}, x_{2}]$, then $\Der_{A}(-\log g)$ is a free module and $g$ is free. The same is true when considering analytic logarithmic derivations along $g$ where $g$ a global function on a two-dimensional complex manifold. 
        \item By the product rule, logarithmic derivations in both the analytic and algebraic setting, depend only on the reduced $f_{\red}$ or $D_{\red}$, e.g. $\Der_{R}(-\log f) = \Der_{R}(-\log f_{\red})$.
    \end{enumerate}
\end{remark}

Below we introduce the \emph{logarithmic stratification} of $\Var(f)$. When this stratification is \emph{locally finite} (we call such divisors \emph{Saito-holonomic}) induction is viable. We use this in Proposition \ref{prop - U complex resolves}. This often requires the analytic setting, though note Remark \ref{rmk - basics pos weighted hom, loc everywhere}.(c).

\begin{remark} \label{rmk - log stratification} \emph{(Basics on Logarithmic Stratifications)}
\begin{enumerate}[label=(\alph*)]
    \item The \emph{logarithmic stratification of $X$ along $D$ (or $f$)} is induced by the equivalence relation: $\mathfrak{x} \sim \mathfrak{y}$ if there is an integral curve of a logarithmic derivation passing through $\mathfrak{x}$ and $\mathfrak{y}$. The resulting equivalence classes are \emph{logarithmic strata}; we say $D$ is \emph{Saito-holonomic} if the logarithmic stratification is locally finite. See \cite[Section 3]{SaitoLogarithmicForms} and \cite[Remark 2.6]{uli} for details. By Remark \ref{rmk - basics algebraic log derivations}.(g), the logarithmic stratification of $D$ is the same as of the reduced $D_{\red}$.
    \item $D_{\red, \reg}$ comprises the $n-1$ dimensional logarithmic stratum; $X \setminus D$ the $n$-dimensional logarithmic stratum. Consequently, if $X^{\prime}$ is a two dimensional complex manifold any divisor $D^{\prime} \subseteq X^{\prime}$ is Saito-holonomic.
    \item Suppose that $f$ is a quasi-homogeneous polynomial and $X = \mathbb{C}^{3}$. Let us show $f$ is Saito-holonomic. Let $E = \sum w_{i} x_{i} \partial_{i}$ be the weighted-homogeneity. As $E$ vanishes only at the origin, the only zero dimensional logarithmic stratum will be the origin. It follows that for all $\mathfrak{x} \neq 0$, there is a coordinate system inducing a local (at $\mathfrak{x})$) analytic isomorphism 
    \begin{equation} \label{eqn - pos weighted homogeneous, away from origin, local product}
        (X, D, \mathfrak{x}) \simeq (\mathbb{C}, \mathbb{C}, 0) \times (X^{\prime}, D^{\prime}, \mathfrak{x}^{\prime}).
    \end{equation}
    Here: $X^{\prime}$ is a two dimensional complex manifold and $D^{\prime} \subseteq X^{\prime}$ is a divisor germ at $\mathfrak{x}^{\prime}$. As Saito-holonomicity is preserved along such products, it follows that $f$ is Saito-holonomic. (See \cite[(3.4)-(3.6)]{SaitoLogarithmicForms} or \cite[Remark 1.6(a)]{BathSaitoTLCT} for details, especially about the analytic isomorphism.)
    \item Stay in the setting of the preceding item. The analytic isomorphism \eqref{eqn - pos weighted homogeneous, away from origin, local product} shows that at $\mathfrak{x} \neq 0$ we may realize $\Der_{\mathscr{O}_{X,\mathfrak{x}}}(\log f)$ as a tensor product of $\Der_{\mathscr{O}_{\mathbb{C}, 0}}$ and $\Der_{\mathscr{O}_{X^{\prime}, x^{\prime}}}(- \log f^{\prime})$, where $f^{\prime}$ is a defining equation for the germ $(X^{\prime}, D^{\prime}, \mathfrak{x}^{\prime})$. We deduce: if $f$ is a quasi-homogeneous polynomial and $X = \mathbb{C}^{3}$, then $\Der_{\mathscr{O}_{X}}(-\log f)$ is free away from the origin. 
\end{enumerate} 
\end{remark}

We conclude with a simple observation about $\Ext_{R}^{1}(\Der_{R}(-\log f), R)$ provided that $f \in R = \mathbb{C}[x_{1}, x_{2}, x_{3}]$ is quasi-homogeneous. Note that the $R$-modules $\Ext_R^{\geq 2}(\Der_R(-\log f), R)$ vanish by Remark \ref{rmk - basics algebraic log derivations}.(d).

\begin{proposition} \label{prop - weighted homogeneous, pdim at most 1, new}
    Suppose that $f \in R = \mathbb{C}[x_{1}, x_{2}, x_{3}]$ is quasi-homogeneous. Then one of the following occurs:
    \begin{enumerate}[label=(\alph*)]
        \item $f$ is free, that is, $\Der_{R}(-\log f)$ is a free $R$-module;
        \item $\Ext_{R}^{1}(\Der_{R}(-\log f), R)$ is nonzero and supported at $0$.
    \end{enumerate}
\end{proposition}

\begin{proof}
   Consider the short exact sequence of $R$-modules 
   \begin{equation*}
       0 \to \Der_R(-\log f) \to \Der_R \oplus R \to (\partial_1 \bullet f, \partial_2 \bullet f, \partial_3 \bullet f, f) \to 0
   \end{equation*}
   whose last nontrivial map is given by $(\partial_i \oplus 0) \mapsto \partial_i \bullet f$ and $(0 \oplus 1) \mapsto f$. Thus,
   \begin{equation} \label{eqn - ext identification, for referee}
       \Ext_R^1(\Der_R(-\log f), R) \simeq \Ext_R^3(R / (\partial_1 \bullet f, \partial_2 \bullet f, \partial_3 \bullet f, f), R)
   \end{equation}
   Since $\dim R = 3$, the rightmost $\Ext$-module of \eqref{eqn - ext identification, for referee} is either zero or has zero dimensional support. By Remark \ref{rmk - basics algebraic log derivations}.(d), the vanishing of the modules in \eqref{eqn - ext identification, for referee} is equivalent to condition (a). 

   So if (a) does not hold, $\Ext_R^1(\Der_R(-\log f), R)$ must be nonzero with zero dimensional support. Since this module is positively graded and the only positively graded maximal ideal is the irrelevant ideal, \cite[Lemma 1.5.6]{BrunsHerzogCohenMacaulayRings} implies (b). 
\end{proof}

\begin{remark} \label{rmk - why analytic, free outside 0} (Algebraic vs analytic category)
The analytic version of Proposition \ref{prop - weighted homogeneous, pdim at most 1, new} is immediate by Remark \ref{rmk - basics algebraic log derivations}.(d). This utilizes the local analytic product structure \eqref{eqn - pos weighted homogeneous, away from origin, local product}, which need not have an algebraic counterpart. The homological argument of Proposition \ref{prop - weighted homogeneous, pdim at most 1, new} avoids invoking the analytic.
\end{remark}


\subsection{Logarithmic Forms and Free Resolutions}

Now we turn to logarithmic differential forms, a generalization of logarithmic derivations introduced in \cite{SaitoLogarithmicForms}. Familiarity in these objects is central to our approach: they are the building box of the linchpin complex $U_{f}^{\bullet}$ of Definition \ref{def - the Ur complex}. Note that $U_f^\bullet$ is defined analytically.

We end with a detailed discussion of case $f \in R = \mathbb{C}[x_{1}, x_{2}, x_{3}]$ quasi-homogeneous. Here log forms are naturally $\mathbf{w}$-graded and we construct explicit free resolutions of logarithmic $1$,$2$-forms in both the algebraic (where its graded) and the analytic setting (Proposition \ref{prop - graded free resolution of log 1, 2 forms}, Definition \ref{def - fixing minimal graded free res, log 1, 2 forms} and Proposition \ref{prop - analytic free res, log 1, 2 forms}, Definition \ref{def - fixing analytic free res, log 1, 2 forms}). These are used throughout Section 5. In Proposition \ref{prop - finite hom gen set of Ext modules}, we use the algebraic graded data to describe a finite generating set of $\Ext_{\mathscr{O}_{X,0}}^{1}(\Omega_{X,0}^{j}(\log f), \mathscr{O}_{X,0})$, for $j = 1,2$, as a $\mathbb{C}$-vector space and attach $\textbf{w}$-degree data to each generator. These numerics will also be used in Section 5 and are key to our understanding of the module $L_{f}$ introduced in Definition \ref{def - the cokernel, ext dual object}.

In the previous subsection, we compared analytic facts to algebraic ones. We concluded with Proposition \ref{prop - weighted homogeneous, pdim at most 1, new}, establishing a fact in the algebraic category whose analytic variant is ``obvious''. Here we reverse course: the main prize is Proposition \ref{prop - finite hom gen set of Ext modules} which describes certain analytic $\mathscr{O}_X$-modules entirely in terms of certain $R$-modules, the latter's structure being relatively straightforward.

First, the objects at hand:

\begin{define} \label{def - log forms}
    For $f \in R = \mathbb{C}[x_{1}, \dots, x_{n}]$ we define the \emph{(global algebraic) logarithmic $p$-forms along $f$} to be
    \[
    \Omega_{R}^{p}(\log f) = \{ \eta \in \frac{1}{f} \Omega_{R}^{p} \mid d(\eta) \in \frac{1}{f} \Omega_{R}^{p+1} \}.
    \]
    Additionally we define
    \[
    \Omega_{R}^{p}(\log_{0} f) = \{ \eta \in \Omega_{R}^{p}(\log f) \mid df/f \wedge \eta = 0\}. 
    \]
    We have similar constructions in the analytic setting. We define the \emph{(analytic) logarithmic $p$-forms along $f$} to be 
    \[
    \Omega_{X}^{p}(\log f) = \{ \eta \in \Omega_{R}^{p}(f) \mid d(\eta) \in \Omega_{R}^{p+1}(f) \}
    \]
    where $\Omega_{R}^{p}(f) = \Omega_{R}^{p}(D) = \Omega_{R}^{p} \otimes_{\mathscr{O}_{X}} \mathscr{O}_{X}(D)$ where $D$ is the divisor defined by $f$. We also define
    \[
    \Omega_{X,\mathfrak{x}}^{p}(\log_{0} f) = \{\eta \in \Omega_{X,\mathfrak{x}}^{p}(\log f) \mid df/f \wedge \eta = 0\}.
    \]
\end{define}

\begin{remark} \label{rmk - basics alg log forms} (Basics on Algebraic Log Forms) Throughout $f \in R$ is quasi-homogeneous with weighted-homogeneity $E$.
    \begin{enumerate}[label=(\alph*)]
        \item For any logarithmic derivation $\delta$ and any log $p$-form $\eta$ the contraction $\iota_{\delta} (\eta)$ of $\eta$ along $\delta$ is a log $p-1$ form, cf. \cite[pg 268]{SaitoLogarithmicForms}. (This does not require $f$ to be quasi-homogeneous.) In particular
        \[
        \iota_{E} \Omega_{R}^{p+1}(\log_{0} f) \subseteq \Omega_{R}^{p}(\log f).
        \]
        \item By arguing as in \cite[Lemma 3.11]{uli}
        \[
        \iota_{E} \Omega_{R}^{p+1}(\log_{0} f) \simeq \Omega_{R}^{p+1}(\log_{0} f),
        \]
        that is, $\iota_{E}$ restricted to $\Omega_{R}^{p+1}(\log_{0} f)$ is an $R$-module isomorphism onto its image. Moreover
        \[
        \Omega_{R}^{p}(\log f) = \Omega_{R}^{p}(\log_{0} f) \oplus \iota_{E} \Omega_{R}^{p+1}(\log_{0} f).
        \]
        \item It is easy to check that $\Omega_{R}^{0}(\log f) = \Omega_{R}^{0}$ and $\Omega_{R}^{n}(\log f) = (1/f) \Omega_{R}^{n}$.
        \item We consider log $1$-forms. By a preceding item
        \[
        \Omega_{R}^{0}(\log f) = \Omega_{R}^{0}(\log_{0} f) \oplus \iota_{E} \Omega_{R}^{1}(\log_{0} f) \simeq \Omega_{R}^{1}(\log_{0} f).
        \]
        This implies two things: there exists a log $1$-form $\eta$ such that $\Omega_{R}^{1}(\log_{0} f) = R \cdot \eta$; contracting the aforementioned $\eta$ along $E$ produces a unit in $R$. As $df/f \wedge df/f = 0$, we may find $g \in R$ so that $df/f = g \eta$. Then $1 = \iota_{E} (df/f) = \iota_{E} (g \eta) = g \iota_{E} \eta$ which forces $g$ to be a unit. That is, $\eta$ equals $df/f$ up to constant and 
        \[
        \Omega_{R}^{1}(\log_{0} f) = R \cdot df/f.
        \]
        \item Here we consider log $n-1$ forms. First of all, there is an isomorphism $\Omega_{R}^{n-1}(\log f) \simeq \Der_{R}(-\log f)$ given by
        \[
        \sum_{i} \frac{a_{i} \widehat{dx_{i}}}{f} \mapsto \sum_{i} (-1)^{i-1} a_{i} \partial_{i}. 
        \]
        We also have the following direct sums and isomorphisms:
        \[
        \Omega_{R}^{n-1}(\log f) = \Omega_{R}^{n-1}(\log f_{0}) \oplus R \cdot \iota_{E} (dx/f) \simeq \Der_{R}(-\log_{0} f) \oplus R \cdot E. 
        \]
        \item If $f$ is reduced, then contraction gives a perfect pairing \cite{SaitoLogarithmicForms}:
        \[
        \Der_{R}(-\log f) \times \Omega_{R}^{1}(\log f) \to \Omega_{R}^{0}(\log f) \simeq R.
        \]
        This holds both algebraically and analytically.
        \item We have not assumed reducedness. See \cite[Appendix A]{BathANote} for non-reduced references as needed. The above perfect pairing exists in modified form without reducedness.
    \end{enumerate}
\end{remark}

\begin{remark} \label{rmk - grading log forms} \emph{(Grading $\Omega_{R}^{p}(\log f)$)} Suppose $f \in R$ is quasi-homogeneous with weights $\mathbf{w}$ and weighted-homogeneity $E = \sum w_{i} x_{i} \partial_{i}$. We may grade $(1/f)\Omega_{R}^{p}$ by declaring $dx_{I} / f$ to have $\textbf{w}$-degree $(- \wdeg(f) + \sum w_{i})$. Since the submodule $\Omega_{R}^{p}(\log f) \subseteq (1/f) \Omega_{R}^{p}$ is characterized by syzygy relations amongst $\{f , \partial_{1} \bullet f, \dots, \partial_{n} \bullet f\}$, all of which are $\textbf{w}$-weighted homogeneous, $\Omega_{R}^{p}(\log f)$ is a graded submodule of $(1/f)\Omega_{R}^{p}.$ Similarly for $\Omega_{R}^{p}(\log_{0} f)$. The grading is compatible with contraction along $E$: if $\eta \in \Omega_{R}^{p}(\log f)$ is $\textbf{w}$-homogeneous, then $\iota_{E}(\eta)$ will be $\textbf{w}$-homogeneous and $\wdeg (\iota_{E}(\eta)) = \wdeg( \eta )$ provided that $\iota_{E}(\eta) \neq 0$.
    
\end{remark}

Now let us start tracking degree data attached to logarithmic forms of quasi-homogeneous divisors in $R = \mathbb{C}[x_{1}, x_{2}, x_{3}]$. This class of polynomials is our focus for the rest of this subsection and the next. The following is a straightforward construction of graded free resolutions; recall they exist for our $\mathbf{w}$-grading as discussed in the first subsection.

\begin{proposition} \label{prop - graded free resolution of log 1, 2 forms}
    Let $f \in R \in \mathbb{C}[x_{1}, x_{2}, x_{3}]$ be quasi-homogeneous with $E$ its weighted-homogeneity. Let 
    \[
    0 \to F_{t} \xrightarrow[]{\beta_{t}} F_{t-1} \to \cdots \to F_{1} \xrightarrow[]{\beta_{1}} F_{0} \xrightarrow[]{\beta_{0}} \Omega_{R}^{2}(\log_{0} f) \to 0
    \]
    be a minimal graded free resolution of $\Omega_{R}^{2}(\log_{0} f)$. Then $t \leq 1$. Moreover:
    \begin{align} \label{eqn - free res log 2-form, homogeneous, dim 3}
    0 \to F_{t} \xrightarrow[]{\beta_{t}} F_{t-1} \to \cdots \to F_{1} &\xrightarrow[]{(\beta_{1}, 0)} F_{0} \oplus R(n-d) \\
    &\xrightarrow[]{(\beta_{0}, \cdot \iota_{E} dx/f)} \Omega_{R}^{2}(\log_{0} f) \oplus R \cdot \iota_{E} dx/f \to 0 \nonumber
    \end{align}
    is a graded free resolution of $\Omega_{R}^{2}(\log f)$;
    \begin{equation} \label{eqn - free res log 1-form, homogeneous dim 3}
    0 \to F_{t} \xrightarrow[]{\beta_{t}} F_{t-1} \cdots \to F_{1} \xrightarrow[]{(\beta_{1}, 0)} F_{0} \oplus R \xrightarrow[]{(\alpha_{0}, \cdot df/f)} \iota_{E}\Omega_{R}^{2}(\log_{0} f) \oplus R \cdot df/f \to 0
    \end{equation}
    is a graded free resolution of $\Omega_{R}^{1}(\log f)$ where $\alpha_{0}(m) = \iota_{E} \beta_{0}(m)$. 
\end{proposition}

\begin{proof}
    By the hypothesis, \eqref{eqn - free res log 2-form, homogeneous, dim 3} is acyclic; by the fact contraction along $E$ presrves degree, it is graded acyclic. Remark \ref{rmk - basics alg log forms}.(b) shows \eqref{eqn - free res log 2-form, homogeneous, dim 3} resolves $\Omega_{R}^{2}(\log f)$. As for \eqref{eqn - free res log 1-form, homogeneous dim 3}, the fact that $\iota_{E}: \Omega_{R}^{k+1}(\log_{0} f) \to \iota_{E}\Omega_{R}^{k+1}(\log_{0} f)$ is a degree preserving $R$-isomorphism (Remark \ref{rmk - grading log forms}) along with the construction shows \eqref{eqn - free res log 1-form, homogeneous dim 3} is graded acyclic; Remark \ref{rmk - basics alg log forms}.(b),(d) shows it resolves $\Omega_{R}^{1}(\log f)$. The resolutions are minimal by nature of the maps. As for the claim $t \leq 1$, this is just Remark \ref{rmk - basics algebraic log derivations}.(d) combined with the isomorphism $\Omega_{R}^{2}(\log_{0} f) \simeq \Der_{R}(\log_{0} f)$ of Remark \ref{rmk - basics alg log forms}.(e).
\end{proof}

We name the resolutions just constructed:

\begin{define} \label{def - fixing minimal graded free res, log 1, 2 forms}
    Let $f \in R = \mathbb{C}[x_{1}, x_{2}, x_{3}]$ be quasi-homogeneous with $E$ its weighted-homogeneity. We set a minimal graded free resolution of $\Omega_{R}^{2}(\log_{0} f)$ to be
    \[
    F_{\bullet} = 0 \to F_{1} = \bigoplus_{i} R(b_{i}) \xrightarrow[]{\beta_{1}} F_{0} = \bigoplus_{j} R(a_{j}) \xrightarrow[]{\beta_{0}} \Omega_{R}^{2}(\log_{0} f) \to 0
    \]
    (When $f$ is free we allow $F_{1} = 0$, otherwise we do not.) We set the minimal graded free resolutions of Proposition \ref{prop - graded free resolution of log 1, 2 forms} to be:
    \begin{equation} \label{eqn - fixed min graded free res, log 2 forms}
        P_{\bullet}^{2} = 0 \to P_{1}^{2} = F_{1} \xrightarrow[]{(\beta_{1}, 0)} \to P_{0}^{2} = F_{0} \oplus R( - d + \sum_{i} w_{i}) \xrightarrow[]{(\beta_{0}, \iota_{E} dx/f)} \Omega_{R}^{2}(\log f) \to 0;
    \end{equation}
    \begin{equation}  \label{eqn - fixed min graded free res, log 1 forms}
        P_{\bullet}^{1} = 0 \to P_{1}^{1} = F_{1} \xrightarrow[]{(\beta_{1}, 0)} P_{0}^{0} = F_{0} \oplus R \xrightarrow[]{(\iota_{E} \beta_{0}, \cdot df/f)} \Omega_{R}^{1}(\log f) \to 0.
    \end{equation}
\end{define}

We convert Definition \ref{def - fixing minimal graded free res, log 1, 2 forms} into the analytic setting:

\begin{proposition} \label{prop - analytic free res, log 1, 2 forms}
    Let $f \in R = \mathbb{C}[x_{1}, x_{2}, x_{3}]$ be quasi-homogeneous with $E$ its weighted-homogeneity. Consider the graded free $R$-resolutions of \eqref{eqn - fixed min graded free res, log 2 forms} and \eqref{eqn - fixed min graded free res, log 1 forms}. Sheafifying (in the algebraic category) and then analytifying produces locally free $\mathscr{O}_{X}$-module resolutions of $\Omega_{X}^{2}(\log f)$ and $\Omega_{X}^{1}(\log f)$ respectively. 
\end{proposition}

\begin{proof}
    Since $\Spec R$ is affine, sheafifying (in the algebraic category) is exact. The analytification functor is also exact.
\end{proof}

We name these analytic free resolutions for later use, e.g. Definition \ref{def - the double complex}.

\begin{define} \label{def - fixing analytic free res, log 1, 2 forms}
    Let $f \in R = \mathbb{C}[x_{1}, x_{2}, x_{3}]$ be quasi-homogeneous with $E$ its weighted-homogeneity. We set the free $\mathscr{O}_{X}$-module resolutions produced by Proposition \ref{prop - analytic free res, log 1, 2 forms} to be
    \begin{equation} \label{eqn - fixed analytic free res, log 2 forms}
        Q_{\bullet}^{2} = 0 \to Q_{1}^{2} \to Q_{0}^{2} \to \Omega_{X}^{2}(\log f) \to 0,
    \end{equation}
    \begin{equation} \label{eqn - fixed analytic free res, log 1 forms}
        Q_{\bullet}^{1} = 0 \to Q_{1}^{1} \to Q_{0}^{1} \to \Omega_{X}^{1}(\log f) \to 0.
    \end{equation}
    In particular, $Q_{i}^{j} = (\widetilde{F_{i}^{j}})^{\an}$ and the maps, which are as described in Definition \ref{def - fixing minimal graded free res, log 1, 2 forms}, are given by matrices with $\textbf{w}$-homogeneous polynomial entries. Additionally, $Q_{1}^{2} = Q_{1}^{1}.$
\end{define}

Finally we give out promised analysis of $\Ext_{\mathscr{O}_{X}}^{1}(\Omega_{X}^{j}(\log f), \mathscr{O}_{X})$ for $j = 1,2$. The takeaways here are: the module with $j=1$ is naturally identified with the module $j=2$; both arise from the algebraic analogues in natural ways; the stalks of each at $0$ have naturally induced finite $\mathbb{C}$-generating sets of ``homogeneous'' elements, whose degree data is entirely determined by the algebraic analogue. 

\begin{proposition} \label{prop - finite hom gen set of Ext modules}
    Let $f \in R = \mathbb{C}[x_{1}, x_{2}, x_{3}]$ be quasi-homogeneous. Then $\Hom_{\mathscr{O}_{X}}(Q_{\bullet}^{j}, \mathscr{O}_{X})$, for $j = 1, 2$ is obtained by sheafifying (in the algebraic category) and then analytifying the complex $\Hom_{R}(P_{\bullet}^{j}, R)$. We also have isomorphisms:
    \begin{equation} \label{eqn - diagram, four ext modules}
    \begin{tikzcd} 
        \Ext_{\mathscr{O}_{X}}^{1}(\Omega_{X}^{2}(\log f), \mathscr{O}_{X})  \dar{\simeq}
            & (\widetilde{\Ext_{R}^{1}(\Omega_{R}^{2}(\log f), R)})^{\an}  \dar{\simeq} \lar{\simeq} \\
        \Ext_{\mathscr{O}_{X}}^{1}(\Omega_{X}^{1}(\log f), \mathscr{O}_{X}) 
            & (\widetilde{\Ext_{R}^{1}(\Omega_{R}^{1}(\log f), R)})^{\an} \lar{\simeq}.
    \end{tikzcd}
    \end{equation}
    Moreover, $\Ext_{R}^{1}(\Omega_{R}^{j}(\log f), R)$ is supported at $\mathfrak{m}$ and is a finite $\mathbb{C}$-vector space with basis a finite set of homogeneous elements $\Delta_{f}$. Finally the only potentially nonzero stalk of $\Ext_{\mathscr{O}_{X}}^{1}(\Omega_{\mathscr{O}_{X}}^{j}(\log f), \mathscr{O}_{X})$ is at the origin, and there it is a finite $\mathbb{C}$-vector space generated by $\Delta_{f}$. $\Delta_{f}$ can be chosen the same for each choice of $j$. 
\end{proposition}

\begin{proof}
    The statement $\Hom_{\mathscr{O}_{X}}(Q_{\bullet}^{j}, \mathscr{O}_{X}) = (\widetilde{\Hom_{R}(P_{\bullet}^{j}, R)})^{\an}$ follows by construction. Indeed we have equality of objects since $\mathscr{O}_{X}$ and $R$ are both self-dual. We also have equality of maps: in all of these $\Hom$ complexes the maps are induced by the transposes of the matrices giving the maps of $P_{\bullet}^{j}$ of Definition \ref{def - fixing minimal graded free res, log 1, 2 forms}, which have polynomial entries.  

    To justify the horizontal isomorphism of \eqref{eqn - diagram, four ext modules}, we observe:
    \begin{align*}
    \Ext_{\mathscr{O}_{X}}^{1}(\Omega_{X}^{j}(\log f), \mathscr{O}_{X})
    & \simeq H^{1} \bigg( \Hom_{\mathscr{O}_{X}}(Q_{\bullet}^{j}, \mathscr{O}_{X}) \bigg) \\
    & \simeq H^{1} \bigg( \big( \widetilde{\Hom_{R}(P_{\bullet}^{j}, R)} \big)^{\an} \bigg) \\
    &\simeq \bigg( \widetilde{ H^{1} \big( \Hom_{R}(P_{\bullet}^{j}, R) \big) } \bigg)^{\an} \\
    &\simeq \big( \widetilde{\Ext_{R}^{1}(\Omega_{R}^{j}(\log f), R)} \big)^{\an}.
    \end{align*}
    Here: the first ``$\simeq$'' is Proposition \ref{prop - analytic free res, log 1, 2 forms}; the second ``$\simeq$'' is what we showed in the first paragraph; the third ``$\simeq$'' is the fact (algebraic) sheafification and analytification are exact, and hence commute with taking cohomology; the fourth ``$\simeq$'' is Proposition \ref{prop - graded free resolution of log 1, 2 forms}. As for the vertical isomorphisms of \eqref{eqn - diagram, four ext modules}, it suffices to check the rightmost isomorphism. By the definition of $P_{\bullet}^{j}$ and Proposition \ref{prop - graded free resolution of log 1, 2 forms},
    \begin{align} \label{eqn - algebraic ext 1, 2 forms graded isomorphic}
    \Ext_{R}^{1}(\Omega_{R}^{1}(\log f), R) 
    &\simeq \frac{\Hom_{R}(F_{1}, R)}{ \im [\beta_{1}^{T}:  \Hom_{R}(F_{0}, R) \to \Hom_{R}(F_{1}, R)]} \\
    &\simeq \Ext_{R}^{1}(\Omega_{R}^{2}(\log f), R). \nonumber
    \end{align}

    As for the final statements, \eqref{eqn - algebraic ext 1, 2 forms graded isomorphic} shows $\Ext_{R}^{1}(\Omega_{R}^{1}(\log f), R)$ and $\Ext_{R}^{1}(\Omega_{R}^{2}(\log f), R)$ are graded isomorphic. Recall that $\Omega_{R}^{2}(\log f) \simeq \Der_{R}(\log f)$ by Remark \ref{rmk - basics alg log forms}.(e). Using Proposition \ref{prop - weighted homogeneous, pdim at most 1, new}, there is a single finite set of weighted homogeneous elements $\Delta_f$ generating both $\Ext_R^1(\Omega_R^1(\log f), R)$ and $\Ext_R^1(\Omega_R^2(\log f), R)$ as finite dimensional $\mathbb{C}$-vector spaces. By Remark \ref{rmk - basics algebraic log derivations}.(d) and an analytic version of Remark \ref{rmk - basics alg log forms}.(e), the only potentially nonzero stalk of $\Ext_{\mathscr{O}_{X}}^{1}(\Omega_{\mathscr{O}_{X}}^{j}(\log f), \mathscr{O}_{X})$ is at the origin; by \eqref{eqn - diagram, four ext modules} this stalk is generated as a finite dimensional $\mathbb{C}$-vector space by $\Delta_{f}$ as well.
\end{proof}

\begin{remark} \label{rmk - why analytic, graded Betti} (Algebraic versus analytic category)
    Eventually (Section 4, 6, 7) we will have to deal with certain $\Ext$ module calculations in the analytic setting. Proposition \ref{prop - finite hom gen set of Ext modules} will simplify this: it will help characterize these analytic $\Ext$-modules in terms of algebraic graded Betti data.
\end{remark}

\subsection{Degree Data: Milnor Algebra and Local Cohomology}

Under the assumption $f \in R = \mathbb{C}[x_{1}, x_{2}, x_{3}]$ quasi-homogeneous, we have emphasized the degree data of $\Ext_{R}^{1}(\Omega_{R}^{1}(\log f), R) \simeq \Ext_{R}^{1}(\Omega_{R}^{2}(\log f), R)$. Here we translate this degree data into a more familiar form for the non logarithmically-inclined reader. Throughout the subsection we stay in the algebraic setting, though later we can and do use Proposition \ref{prop - finite hom gen set of Ext modules} to convert this degree data to the analytic world.

Let $R = \mathbb{C}[x_{1}, \dots, x_{n}]$ with irrelevant ideal $\mathfrak{m}$ and let $M$ be a finite $R$-module. Define
\[
\Gamma_{\mathfrak{m}}^{0} M = \{ m \in M \mid \mathfrak{m}^{\ell} m = 0 \text{ for some } \ell \in \mathbb{Z}_{\geq 0} \}. 
\]
Then $H_{\mathfrak{m}}^{t}$ is the \emph{local cohomology of $M$ with respect to $\mathfrak{m}$}; it is the $t^{\text{th}}$ derived functor of $\Gamma_{\mathfrak{m}}^{0}$. If $R$ is positively graded and $M$ is also $\textbf{w}$-graded, then $H_{\mathfrak{m}}^{t} M$ is also $\textbf{w}$-graded thanks to the Cech complex interpretation of local cohomology. 

Take $f \in R$. The \emph{Jacobian ideal} $(\partial f) \subseteq R$ of $f$ is $(\partial f) = (\partial_{1} \bullet f, \dots, \partial_{n} \bullet f)$, i.e. the ideal generated by the partial derivatives. When $f$ is quasi-homogeneous with weighted-homogeneity $E = \sum w_{i} x_{i} \partial_{i}$, then $f \in (\partial f)$ since $E \bullet f = \wdeg(f) f$. We call $R / (\partial f)$ the \emph{Milnor algebra} of $f$. 

When $R$ is given a positive $\textbf{w}$-grading, $R$ is a $^{\star}$complete $^{\star}$local ring in the sense of Bruns and Herzog \cite[Section 1.5,  pg 142]{BrunsHerzogCohenMacaulayRings} and so it admits a version of graded local duality \cite[Theorem 3.6.19]{BrunsHerzogCohenMacaulayRings}. In particular, for any finite graded module $M$ over $R$, we have, for all $t$, isomorphism of $\mathbb{C}$-vector spaces
\begin{equation} \label{eqn - positive graded local duality}
    [\Ext_{R}^{j}(M, R(- \sum w_{i}))]_{t} \simeq [\big(H_{\mathfrak{m}}^{n-j} M \big)^{\star}]_{t}
\end{equation}
where $w_{i} = \deg(x_{i})$ and $(-)^{\star}$ denotes the graded $\mathbb{C}$-dual. 

There is a dictionary between degree data of $\Ext$ modules of logarithmic forms and local cohomology of the Milnor algebra:

\begin{proposition} \label{prop - weighted data dictionary, ext log 1 forms, local cohomology milnor algebra}
    Let $f \in R = \mathbb{C}[x_{1}, x_{2}, x_{3}]$ be quasi-homogeneous with weighted-homogeneity $E = \sum w_{i} x_{i} \partial_{i}$. Then
    \begin{align*}
    \wdeg( \Ext_{R}^{1}(\Omega_{R}^{1}(\log f), R) 
    &= \wdeg ( \Ext_{R}^{1}(\Omega_{R}^{2}(\log f), R) ) \\
    &= \wdeg ( \Ext_{R}^{1}(\Der_{R}(-\log f) (\wdeg(f) - \sum w_{i}), R) ) \\
    &= \big\{-t + 2 \wdeg(f) - 2 \sum w_{i} \mid t \in \wdeg (H_{\mathfrak{m}}^{3} (R/(\partial f)) ) \big\}
    \end{align*}
\end{proposition}

\begin{proof}
    First note that $\wdeg( \Ext_{R}^{1}(\Omega_{R}^{j}(\log f), R))$ is finite, for either $j = 1, 2$, by Proposition \ref{prop - finite hom gen set of Ext modules}. By Proposition \ref{prop - graded free resolution of log 1, 2 forms}, there is a graded $R$-isomorphism
    \begin{equation} \label{eqn - ext 1 graded iso, log 1 log 2 forms}
        \Ext_{R}^{1}(\Omega_{R}^{1}(\log f), R) \simeq \Ext_{R}^{1}(\Omega_{R}^{2}(\log f), R)
   \end{equation}
   which gives the first equality promised.
    
    Recall our grading convention $\Der_{R} = \oplus R \cdot \partial_{i} = \oplus R(w_{i}).$ We have a s.e.s. of graded $R$-modules
    \begin{equation} \label{eqn - s.e.s. syzygies of jacobian}
        0 \to \Der_{R}(-\log_{0} f)(- \wdeg(f)) \to \Der_{R}(- \wdeg(f)) \to (\{ \partial_{i} \bullet f \}_{i} ) \to 0.
    \end{equation}
    Here: the last nontrivial map is given by $\partial_{i} \mapsto \partial_{i} \bullet f$; the first nontrivial map is induced by the graded inclusion $\Der_{R}(-\log_{0} f) \hookrightarrow \Der_{R}$. We also have a graded $R$-isomorphism
    \begin{align} \label{eqn - graded iso log 2 forms, log der}
    \Omega_{R}^{2}(\log f)(- \wdeg(f) + \sum w_{i}) \simeq \Der_{R}(-\log f)
    \end{align}
    as described in Remark \ref{rmk - basics algebraic log derivations}.(e). 

    Using all this we find graded isomorphisms:
    \begin{align} \label{eqn - graded iso, ext 1 log 2 forms, ext 2 jacobian}
        \Ext_{R}^{1}(\Omega_{R}^{2}(\log f), R) 
        &\simeq \Ext_{R}^{1} \bigg( \Der_{R}(-\log f)(\deg_{w}(f) - \sum w_{i}), R \bigg) \\
        &\simeq \Ext_{R}^{2} \bigg( (\partial f)(2 \wdeg(f) - \sum w_{i}, R) \bigg) \nonumber \\
        &\simeq \Ext_{R}^{2} ((\partial f), R) (- 2 \wdeg(f) + \sum w_{i}). \nonumber \\
        &\simeq \Ext_{R}^{3}(R / (\partial f), R)(-2 \wdeg(f) + \sum w_{i}). \nonumber \\
        &\simeq \Ext_{R}^{3}(R / (\partial f), R(-\sum w_{i}))(-2\wdeg(f) + 2 \sum w_{i}) \nonumber
    \end{align}
    Here: the first ``$\simeq$'' is \eqref{eqn - graded iso log 2 forms, log der}; the second ``$\simeq$'' is \eqref{eqn - s.e.s. syzygies of jacobian} along with the graded isomorphism $\Der_{R}(-\log f) \simeq \Der_{R}(-\log_{0} f) \oplus R \cdot E$; the penultimate ``$\simeq$'' is the graded s.e.s. $0 \to (\partial f) \to R \to R/(\partial f) \to 0$. The proposition follows by \eqref{eqn - ext 1 graded iso, log 1 log 2 forms} and applying graded local duality, cf. \eqref{eqn - positive graded local duality}, to \eqref{eqn - graded iso, ext 1 log 2 forms, ext 2 jacobian}.
\end{proof}

When $f$ is homogeneous, Proposition \ref{prop - weighted data dictionary, ext log 1 forms, local cohomology milnor algebra} can be reformulated slightly. Indeed, here $R$ has the standard $\textbf{1}$-grading and \cite[Theorem 4.7]{StratenDucoWarmtGorensteinDuality}, \cite[Theorem 3.4]{LocalCohomologyJacobianRing} there is symmetry in the degree data of $H_{\mathfrak{m}}^{0}(R / (\partial f))$ about $(3 \deg(f) - 6)/2$:
\begin{equation} \label{eqn - local cohomology degree duality}
H_{\mathfrak{m}}^{0} (R / (\partial f))_{k} \simeq H_{\mathfrak{m}}^{0} (R / (\partial f))_{3 \deg(f) - 6 - k} \quad (\text{as $\mathbb{C}$-vector spaces}).
\end{equation}

\begin{proposition} \label{prop - homogeneous data dictionary, milnor alg, local cohomology}
    Let $f = R = \mathbb{C}[x_{1}, x_{2}, x_{3}]$ be homogeneous. Then
    \begin{align*}
    \deg( \Ext_{R}^{1}(\Omega_{R}^{1}(\log f), R) 
    &= \deg ( \Ext_{R}^{1}(\Omega_{R}^{2}(\log f), R) \\
    &= \deg \big( H_{\mathfrak{m}}^{0} (R / (\partial f)) \big) - \deg(f).
    \end{align*}
\end{proposition}

\begin{proof}
    Proposition \ref{prop - weighted data dictionary, ext log 1 forms, local cohomology milnor algebra} and \eqref{eqn - local cohomology degree duality}.
\end{proof}

\section{$\mathscr{D}_{X}[s]$-modules and Bernstein--Sato Polynomials}

Recall $\mathscr{O}_{X}$ is the analytic structure sheaf of $X = \mathbb{C}^{n}$. We denote the sheaf of differential operators on $X$ by $\mathscr{D}_{X}$. We denote $\mathscr{D}_{X}[s] = \mathscr{D}_{X} \otimes_{\mathbb{C}} \mathbb{C}[s]$ the polynomial ring extension of $\mathscr{D}_{X}$ by a new variable $s$. In this tiny section we recall some basic notions about $\mathscr{D}_{X}[s]$-modules, define Bernstein--Sato polynomials and related $\mathscr{D}_{X}[s]$-modules, and recall a ``localization'' technique we use sparingly.

In this section we work analytically, but all results hold in the algebraic category with appropriate modifications.

\subsection{Bernstein--Sato Polynomials}

Let us construct the \emph{Bernstein--Sato polynomial}. For $g \in \mathscr{O}_{X}$ global we can define a left $\mathscr{D}_{X}[s]$-module structure on $\mathscr{O}_{X}(\star g)[s] g^{s}$, where $g^{s}$ is a formal symbol. First, extend the $\mathscr{D}_{X}$-action on $\mathscr{O}_{X}(\star g)$ to a $\mathscr{D}_{X}[s]$-action on $\mathscr{O}_{X}(\star g)[s]$ by letting $s \in \mathscr{D}_{X}[s]$ act as multiplication and commute with other operators. Then extend these actions to a $\mathscr{D}_{X}[s]$-module structure on $\mathscr{O}_{X}(\star g)[s] g^{s}$ by the Liebnitz rule and letting a derivation act on $g^{s}$ via a naive application of the chain rule. For example, if $h(s) \in \mathscr{O}_{X}[s]$, 
\[
\partial_{i} \bullet h(s)g^{s} = \bigg( (\partial_{i} \bullet h(s) ) g^{s} \bigg) + \bigg( h(s)s(\partial_{i} \bullet g) g^{s-1} \bigg) = \left( (\partial_{i} \bullet h(s)) + \frac{ s( \partial_{i} \bullet g)}{g} \right) g^{s}.
\]

For $k \in \mathbb{Z}$, we identify $g^{s + k} \in \mathscr{O}_{X}(\star f)[s] g^{s}$ with $g^{k} g^{s} \in \mathscr{O}_{X}(\star f)[s] g^{s}$. 

\begin{define} \label{def - bs poly} With $g \in \mathscr{O}_{X}$ global, define
    \begin{equation} \label{def - submod generated by g to s+k}
    \mathscr{D}_{X}[s] g^{s+k} = \text{left } \mathscr{D}_{X}[s] \text{ submodule of } \mathscr{O}_{X}(\star g)[s] g^{s} \text{ generated by } g^{s+k}.
    \end{equation}
Now consider the left $\mathscr{D}_{X}[s]$-module 
    \begin{equation} \label{eqn - bs poly, defining quotient}
    \frac{\mathscr{D}_{X}[s] g^{s}}{\mathscr{D}_{X}[s] g^{s+1}}.
    \end{equation}
The $\mathbb{C}[s]$-annihilator of \eqref{eqn - bs poly, defining quotient} is an ideal in $\mathbb{C}[s]$; we call its monic generator $b_{g}(s) \in \mathbb{C}[s]$ the \emph{Bernstein--Sato polynomial of $g$.} We denote the zereos of the Bernstein--Sato polynomial by $Z(b_{g}(s)) \subseteq \mathbb{C}$. If we work with $g \in \mathscr{O}_{X,\mathfrak{x}}$, we can repeat the construction, obtaining $b_{g, \mathfrak{x}}(s)$, the \emph{local Bernstein--Sato polynomial of $g$ at $\mathfrak{x}$}.
\end{define}

\begin{remark} \label{rmk - BS poly classical facts}
    That the Bernstein--Sato polynomial is not $0$ is nontrivial and was first established (in the algebraic setting) by Bernstein \cite{BernsteinBSExistence}. Many classic singularity invariants lurk within the Bernstein--Sato polynomial: eigenvalues of algebraic monodromy (Kashiwara \cite{KashiwaraVanishingCycleSheaves}, Malgrange \cite{MalgrongeIsolee}); jumping numbers; log canonical thresholds; minimal exponents; etc. We recommend the survey \cite{WaltherSurvey}.
\end{remark}

More generally,

\begin{define} \label{def - b-function module}
    Let $M$ be a coherent (left or right) $\mathscr{D}_{X}[s]$-module. We define the $\mathbb{C}[s]$-ideal
    \begin{equation}
        B(M) = \ann_{\mathbb{C}[s]} M \subseteq \mathbb{C}[s].
    \end{equation}
    If $B(M) \neq 0$, then its monic generator is the \emph{b-function of $M$}. We denote the zeroes of $B(M)$ by $Z(B(M)) \subseteq \mathbb{C}$. We have a similar notion locally at $\mathfrak{x}$ using $M_{\mathfrak{x}}$. 
\end{define}

\subsection{Annihilators}

Since $\mathscr{D}_{X}[s] g^{s+k}$ is a cyclic module, we have the canonical isomorphism
\[
\frac{\mathscr{D}_{X}[s]}{\ann_{\mathscr{D}_{X}[s]}g^{s+k}} \xrightarrow[]{\simeq} \mathscr{D}_{X}[s] g^{s+k}
\]
where, of course, $\ann_{\mathscr{D}_{X}[s]} g^{s+k} = \{P(s) \in \mathscr{D}_{X}[s] \mid P(s) \bullet g^{s+k} = 0\}$. 

In certain situations, include those of Convention \ref{convention - main hypotheses}, $\ann_{\mathscr{D}_{X}[s]}g^{s+k}$ has as nice as possible description. To set this up, note that if $\delta \in \Der_{\mathscr{O}_{X}}(-\log g)$ is a logarithmic derivation, then $\delta - (s+k)(\delta \bullet g)/g  \in \mathscr{D}_{X}[s]$ and in fact kills $g^{s+k}$. We define

\begin{define}
    We define the left $\mathscr{D}_{X}[s]$-ideal 
    \begin{equation} \label{eqn - ann g s order at most one}
        \ann_{\mathscr{D}_{X}[s]}^{(1)} g^{s+k} = \mathscr{D}_{X}[s] \cdot \{ \delta - s\frac{ \delta \bullet g}{g} \mid \delta \in \Der_{\mathscr{O}_{X}}(\log g) \} \subseteq \ann_{\mathscr{D}_{X}[s]} g^{s+k}.
    \end{equation}
\end{define}

There has been a lot of work studying when the inclusion of \eqref{eqn - ann g s order at most one} is actually an equality, cf. \cite{YanoTheoryBFunctions}, \cite{ModulefsLocallyQuasiHomogeneousFree}, \cite{LCTforIntegrableLogarithmicConnections}, \cite{DualityApproachSymmetryBSpolys}, \cite{uli}. The most general assumptions are provided in \cite[Theorem 3.26]{uli}: $g$ is Saito-holonomic (Remark \ref{rmk - log stratification}); $g$ is \emph{tame} (Definition \ref{def - tame}); $g$ is \emph{strongly Euler-homogeneous locally everywhere} (this is a weaker form of locally quasi-homogenous, cf. \cite[Defintion 2.7]{uli}). For our purposes, if $f \in R = \mathbb{C}[x_{1}, x_{2}, x_{3}]$ is locally quasi-homogeneous, then $f$ satisfies these assumptions (cf. Remark \ref{rmk - log stratification}.(d), Remark \ref{rmk - tame in three dim}).

\subsection{Filtrations and Relative Holonomicity}

There are three filtrations that we sometimes consider: the canonical \emph{order filtration} $F_{\bullet}^{\ord} \mathscr{D}_{X}$ on $\mathscr{D}_{X}$; the \emph{relative order filtration} $F_{\bullet}^{\rel} \mathscr{D}_{X}[s] = F_{\bullet}^{\ord} \mathscr{D}_{X} \otimes_{\mathbb{C}} \mathbb{C}[s]$ on $\mathscr{D}_{X}[s]$; the \emph{total order filtration} $F_{\bullet}^{\sharp} \mathscr{D}_{X}[s] = \oplus_{0 \leq k \leq \bullet} F_{\bullet}^{\ord} \mathscr{D}_{X} \otimes \mathbb{C}[s]_{\bullet - k}$ on $\mathscr{D}_{X}[s]$. One can define characteristic varieties with respect to all of these filtrations. For example, we denote $\Ch^{\rel}(-)$ the characteristic variety of the input with respect to the relative order filtration. 

\begin{remark} \label{rmk - log annihilator fs, total order def}
    One can check that $\ann_{\mathscr{D}_{X}[s]}^{(1)} g^{s} = \mathscr{D}_{X}[s] \cdot (F_{1}^{\sharp} \mathscr{D}_{X}[s] \cap \ann_{\mathscr{D}_{X}[s]} g^{s})$.
\end{remark}

Maisonobe introduced \cite{MaisonobeFiltrationRelative} a useful generalization of holonomic $\mathscr{D}_{X}$-modules to the case of $\mathscr{D}_{X}[s]$-modules. While he uses the phrase ``marjor{\'e} par une lagrangienne'' we use the verbiage from the extension of his theory developed in \cite{ZeroLociBSIdeals}, \cite{ZeroLociBSIdealsII}. Some, but not all, of the cited results in \cite{ZeroLociBSIdeals} can be found in \cite{MaisonobeFiltrationRelative}.

\begin{define} \label{def - relative holonomic}
    We say a coherent (left or right) $\mathscr{D}_{X}[s]$-module $M$ is \emph{relative holonomic} at $\mathfrak{x} \in X$ when there exists a small ball $B_{\mathfrak{x}} \subseteq X$ such that the relative characteristic variety $\Ch^{\rel}(M)$ satisfies
    \[
    \Ch^{\rel}(M) \cap \pi^{-1}(B_{\mathfrak{x}} \times \Spec \mathbb{C}[s])= \cup_{\alpha} \Lambda_{\alpha} \times S_{\alpha} \subseteq T^{\star} X \times \Spec \mathbb{C}[s]
    \]
    where: $\Lambda_{\alpha} \subseteq T^{\star}X$ is a Lagrangian; $S_{\alpha} \subseteq \Spec \mathbb{C}[s]$ is an algebraic variety; $\pi : T^{\star} X \times \Spec \mathbb{C}[s] \to X \times \Spec \mathbb{C}[s]$ is the projection. $M$ is \emph{relative holonomic} if it is so at all $\mathfrak{x} \in X$.
\end{define}

We conclude with an elementary Lemma which is contained in both \cite{MaisonobeFiltrationRelative} and \cite{ZeroLociBSIdeals}. 

\begin{lemma} \label{lemma - rel hol basics}
    Throughout, $M$ is a coherent (left or right) $\mathscr{D}_{X}[s]$-module. 
    \begin{enumerate}[label=(\alph*)]
        \item If $M$ is supported at $0$, then $M$ is relative holonomic.
        \item If $M$ is relative holonomic, then any subquotient of $M$ is relative holonomic. 
        \item If $M$ is relative holonomic and $0 \to M_{1} \to M \to M_{2} \to 0$ is a s.e.s. of $\mathscr{D}_{X}[s]$-modules, then 
        \[
        Z(B(M)) = Z(B(M_{1})) \cup Z(B(M_{2})).
        \]
    \end{enumerate}
\end{lemma}

\begin{proof}
    For (a): if $M$ is supported at $0$, then because of the conical structure of $\Ch^{\rel}(M)$, we have that $\Ch^{\rel}(M) \subseteq T_{0}^{\star}X \times \Spec \mathbb{C}[s]$. Now use \cite[Proposition 3.2.5]{ZeroLociBSIdeals}. Claim (b) is \cite[Lemma 3.2.4(1)]{ZeroLociBSIdeals}. As for Claim (c), use \cite[Lemma 3.4.1, Lemma 3.2.2(2)]{ZeroLociBSIdeals}.
\end{proof}

\subsection{Side Changing Operations}

As in the case of $\mathscr{D}_{X}$-modules, there is an equivalence of categories between left $\mathscr{D}_{X}[s]$-modules and right $\mathscr{D}_{X}[s]$-modules; this equivalence is given by extending the $\mathscr{D}_{X}$-module equivalence in the natural way. Given a left $\mathscr{D}_{X}[s]$-module $M$ we denote the corresponding right $\mathscr{D}_{X}[s]$-module by $M^{r}$; similarly, given a right $\mathscr{D}_{X}[s]$-module $N$ we denote the corresponding left $\mathscr{D}_{X}[s]$-module by $N^{\ell}$. 

In local coordinates the transformation from left to right $\mathscr{D}_{X}[s]$-modules is given by the transposition operator 
\[
\tau (\sum x^{I} \partial^{J} s^{a}) = (-\partial^{J}) x^{I} s^{a}.
\]
Given a set $L \subseteq \mathscr{D}_{X}[s]$ we denote $\tau(L) = \{\tau(\ell) \mid \ell \in L\}$. In particular, if $I \subseteq \mathscr{D}_{X}[s]$ is a left $\mathscr{D}_{X}[s]$-ideal then
\[
\left( \frac{\mathscr{D}_{X}[s]}{I} \right)^{r} = \frac{\mathscr{D}_{X}[s]}{\tau(I)}
\]
Of course $\tau$ also gives the transformation from right to left $\mathscr{D}_{X}[s]$-modules and we have a similar statement relating right cyclic $\mathscr{D}_{X}[s]$-modules to left ones.

We will use the following easy Lemmas about side-changing operators:

\begin{lemma} \label{lem - side changing b-functions}
    If $M$ is a coherent left $\mathscr{D}_{X}[s]$-module, then $B(M) = B(M^{r})$. Similarly, if $N$ is a coherent right $\mathscr{D}_{X}[s]$-module, then $B(N) = B(N^{\ell})$.
\end{lemma}

\begin{proof}
    By coherence and that the $\mathbb{C}[s]$-action is unchanged when side changing.
\end{proof}

\begin{lemma} \label{lem - side changing, order one annihilator} In local coordinates $(x, \partial)$ if $\delta = \sum a_{i} \partial_{i} \in \Der_{\mathscr{O}_{X}}(-\log g) \subseteq \mathscr{D}_{X}$,  then 
\[
\tau(\delta - s \frac{\delta \bullet g}{g}) = - \delta - \big( \sum_{i} \partial_{i} \bullet a_{i} \big) - s \frac{\delta \bullet g}{g}.
\]
\end{lemma}

\begin{proof}
    Compute.
\end{proof}

\begin{remark} \label{rmk - formal substitution is ok}
    Let us note that everything in the preceding subsections holds when we replace $s + k$ with $\alpha s + \beta$ provided $\alpha \in \mathbb{C}^{\star}$ and $\beta \in \mathbb{C}$. Indeed, $\mathscr{D}_{X}[s] f^{\alpha s + \beta}$ is defined as expected by naive chain rule application and by hand computation shows that $\ann_{\mathscr{D}_{X}[s]}^{(1)} f^{\alpha s + \beta}$ and $\ann_{\mathscr{D}_{X}[s]} f^{\alpha s + \beta}$ are recovered from their $f^{s}$ counterparts by the same formal substitution. 
\end{remark}

\subsection{Ext Modules and Localizations}

Let $S \subseteq \mathbb{C}[s]$ be a multiplicatively closed subset; let $T = S^{-1} \mathbb{C}[s]$. As in the prequel, we may consider left or right modules over $\mathscr{D}_{T} = \mathscr{D}_{X} \otimes_{\mathbb{C}} T$. Furthermore, we may define the \emph{relative order filtration} on $\mathscr{D}_{T}$ and the relative characteristic variety in analogous ways. Again we say a coherent (left or right) $M$ $\mathscr{D}_{T}$-module is \emph{relative holonomic} when $\Ch^{\rel}(M)$ has a decomposition into products of Lagrangians of $T^{\star}X$ and algebraic varieties in $\Spec T$ as in Definition \ref{def - relative holonomic}. Finally, we define the $b$-function of a left/right $\mathscr{D}_{T}$-module $M$ as before: it is the $T$-annihilator of $M$.

When $M$ is a left (resp. right) $\mathscr{D}_{T}$-module, $\Ext_{\mathscr{D}_{T}}^{\bullet}(M, \mathscr{D}_{T})$ has a natural right (resp. left) $\mathscr{D}_T$-module structure. We \emph{always} consider such $\Ext$ modules to have this right/left $\mathscr{D}_{T}$-module structures. So for us,
\[
\{\text{left $\mathscr{D}_{T}$-modules}\} \ni M \mapsto \Ext_{\mathscr{D}_{T}}^{\bullet}(M, \mathscr{D}_{T}) \in \{\text{right $\mathscr{D}_{T}$-modules}\},
\]
and symmetrically as well. 

We record three more Lemmas that will be utilized at the very end of the paper. All but the last are scavenged from \cite{ZeroLociBSIdeals}. First we need a consequence of the flatness of localization:

\begin{lemma} \label{lemma - localization and Ext}
    Let $M$ be a relative holonomic left or right $\mathscr{D}_{X}[s]$-module. Let $S \subseteq \mathbb{C}[s]$ be a multiplicatively closed set and $T = S^{-1} \mathbb{C}[s]$. Then at any point $\mathfrak{x} \in X$ we have an isomorphism of $\mathscr{D}_{T, \mathfrak{x}} = S^{-1} \mathscr{D}_{X,\mathfrak{x}}$-modules
    \[
    S^{-1} \Ext_{\mathscr{D}_{X, \mathfrak{x}}[s]}^{v}(M_{\mathfrak{x}}, \mathscr{D}_{X, \mathfrak{x}}) \simeq \Ext_{\mathscr{D}_{T, \mathfrak{x}}}^{v}(S^{-1} M_{\mathfrak{x}}, \mathscr{D}_{T, \mathfrak{x}})
    \]
\end{lemma}

\begin{proof}
    This is demonstrated at the end of proof of \cite[Lemma 3.5.2]{ZeroLociBSIdeals}.
\end{proof}

For the last two lemmas, we introduce some notation. Let $M$ be a left or right $\mathscr{D}_{T}$-module. We say the \emph{grade} of $M$ at $\mathfrak{x} \in X$ is $v$ when $\Ext_{\mathscr{D}_{T, \mathfrak{x}}}^{j}( M_{\mathfrak{x}}, \mathscr{D}_{T, \mathfrak{x}}) = 0$ for all $j < v$ and is nonzero when $j = v$.  We say $M$ is \emph{$v$-Cohen--Macaulay} at $\mathfrak{x} \in X$ when $\Ext_{\mathscr{D}_{T, \mathfrak{x}}}^{j}( M_{\mathfrak{x}}, \mathscr{D}_{T, \mathfrak{x}}) = 0 \iff  j \neq v$.

\begin{lemma} \label{lemma - localization plays well with b-functions}
    Let $M$ be a relative holonomic left or right $\mathscr{D}_{X}[s]$-module that is $(n+1)$-Cohen--Macaulay. Let $p(s) \in \mathbb{C}[s]$ and $S = \{1, p(s), p(s)^{2}, \dots \}$ so that $T = S^{-1}\mathbb{C}[s] = \mathbb{C}[s, p(s)^{-1}]$. Then for any $\mathfrak{x} \in X$, 
    \[
    Z(B(M_{\mathfrak{x}})) \cap \Spec \mathbb{C}[s, p(s)^{-1}] = Z(B(S^{-1} M_{\mathfrak{x}})),
    \]
    where we regard $S^{-1} M_{\mathfrak{x}}$ as a $\mathscr{D}_{T, \mathfrak{x}}$-module. 
\end{lemma}

\begin{proof}
    Let $\alpha \in \Spec \mathbb{C}[s, p(s)^{-1}]$ and $\mathbb{C}_{\alpha}$ the corresponding residue field. By the commutativity of localization we have an isomorphism of $\mathscr{D}_{\mathbb{C}_{\alpha}, \mathfrak{x}}$ modules
    \[
    M_{\mathfrak{x}} \otimes_{\mathbb{C}[s]} \mathbb{C}_{\alpha} \simeq S^{-1} M_{\mathfrak{x}} \otimes_{\mathbb{C}[s]} \mathbb{C}_{\alpha}.
    \]
    Now use \cite[Proposition 3.4.3]{ZeroLociBSIdeals} and the fact $n+1$-Cohen--Macaulay-ness persists under localizations (Lemma \ref{lemma - localization and Ext}).
\end{proof}

\begin{lemma} \label{lemma - b-function and b-function of dual}
    Suppose that $M \neq 0$ is a left or right $\mathscr{D}_{T}$-module that is relative holonomic. If $M$ is $(n+1)$-Cohen--Macaulay at $\mathfrak{x}$ then
    \begin{enumerate}[label=(\alph*)]
        \item $M_{\mathfrak{x}}$ and $\Ext_{\mathscr{D}_{T, \mathfrak{x}}}^{n+1}(M_{\mathfrak{x}}, \mathscr{D}_{T, \mathfrak{x}})$ have $b$-functions;
        \item The zeroes of their $b$-functions agree: 
        \[
        Z(B(M_{\mathfrak{x}}) = Z(B(\Ext_{\mathscr{D}_{T, \mathfrak{x}}}^{n+1}(M_{\mathfrak{x}}, \mathscr{D}_{T, \mathfrak{x}})).
        \]
    \end{enumerate}
\end{lemma}

\begin{proof}
    To show $M_{\mathfrak{x}}$ has a $b$-function, it suffices by \cite[Lemma 3.4.1]{ZeroLociBSIdeals} to show that $\dim \Ch^{\rel} M_{\mathfrak{x}} = n$. By the complementary numerics of grade and dimension of the relative characteristic variety \cite[Remark 3.3.2]{ZeroLociBSIdeals}, our assumptions give $2n+1 = (n+1) + \dim \Ch^{\rel} M_{\mathfrak{x}}$. So $M_{\mathfrak{x}}$ has a $b$-function.

    As for $\Ext_{\mathscr{D}_{T, \mathfrak{x}}}^{n+1}(M_{\mathfrak{x}}, \mathscr{D}_{T, \mathfrak{x}})$, as $\mathscr{D}_{T, \mathfrak{x}}$ is an Auslander ring \cite[IV.Theorem 4.15]{BjorkAnalyticDmodsApplications}, by the Auslander condition its grade is at least $n+1$; by \cite[Lemma 3.2.4(2), Proposition 3.2.5]{ZeroLociBSIdeals} it is relative holonomic. Thus it has a $b$-function by arguing as before. 

    As for (b), by \cite[Lemma 3.4.1]{ZeroLociBSIdeals} it suffices to show $\Ch^{\rel} (M_{\mathfrak{x}}) = \Ch^{\rel} (\Ext_{\mathscr{D}_{T, \mathfrak{x}}}^{n+1}(M_{\mathfrak{x}}, \mathscr{D}_{T, \mathfrak{x}}).)$  Using \cite[Lemma 3.2.4(2)]{ZeroLociBSIdeals} twice, we have
    \begin{align} \label{eqn - lemma - b-functions double duals 1}
        \Ch^{\rel} \bigg( \Ext_{\mathscr{D}_{T, \mathfrak{x}}}^{n+1}(\Ext_{\mathscr{D}_{T, \mathfrak{x}}}^{n+1}(M_{\mathfrak{x}}, \mathscr{D}_{T, \mathfrak{x}}), \mathscr{D}_{T, \mathfrak{x}}) \bigg) 
        &\subseteq \Ch^{\rel} \bigg( \Ext_{\mathscr{D}_{T, \mathfrak{x}}}^{n+1}(M_{\mathfrak{x}}, \mathscr{D}_{T, \mathfrak{x}}) \bigg) \\
        &\subseteq \Ch^{\rel} (M_{\mathfrak{x}}). \nonumber
    \end{align} 
    So if we show the leftmost and rightmost modules of \eqref{eqn - lemma - b-functions double duals 1} have the same relative characteristic variety, we will have established (b). 
    
    Note that $M_{\mathfrak{x}}$ is $(n+1)$-pure by \cite[Remark 3.3.2]{ZeroLociBSIdeals}, i.e. every (nonzero) submodule of $M_{\mathfrak{x}}$ has grade $n+1$. Consider the bidualizing filtration 
    \[
    M_{\mathfrak{x}} = B_{0} \supseteq B_{-1} \supseteq \cdots
    \]
    Bj{\"o}rk constructs in \cite[IV.1.6]{BjorkAnalyticDmodsApplications}. By \cite[IV.Proposition 2.8]{BjorkAnalyticDmodsApplications}, this filtration is explicitly given by: $M_{\mathfrak{x}} = B_{0} = B_{-1} = \cdots = B_{-(n+1)}$; and $0 = B_{-(n+1) - 1} = B_{-(n+1) - 2} = \cdots $. Then Bj{\"o}rk's short exact sequence \cite[IV.Proposition 2.1]{BjorkAnalyticDmodsApplications} simplifies to 
    \begin{equation} \label{eqn - lemma - b-functions double duals, 2}
    M_{\mathfrak{x}} \simeq \Ext_{\mathscr{D}_{T, \mathfrak{x}}}^{n+1}(\Ext_{\mathscr{D}_{T, \mathfrak{x}}}^{n+1}(M_{\mathfrak{x}}, \mathscr{D}_{T, \mathfrak{x}}), \mathscr{D}_{T, \mathfrak{x}})
    \end{equation}
    In particular, the leftmost and rightmost relative characteristic varieties of \eqref{eqn - lemma - b-functions double duals 1} agree and (b) is verified.
\end{proof}

\begin{remark} \label{rmk - explaining dual}
    Whenever we discuss the dual of a $\mathscr{D}_{X}[s]$-module $M$, we mean $\Ext_{\mathscr{D}_{X}[s]}^{j}(M, \mathscr{D}_{X}[s])$ for $j$ the grade of $M$. Similarly for $\mathscr{D}_{T}$-modules.
\end{remark}

\section{Constructing Useful Complexes}

In this section we introduce a complex $U_{f}^{\bullet}$ (Definition \ref{def - the Ur complex}) of not necessarily projective right $\mathscr{D}_{X}[s]$-modules that, in our case of interest, resolves $\mathscr{D}_{X}[s] f^{s-1}.$ By standard arguments, by replacing every object in this complex with its own free resolution, the resultant double complex $K_{f}^{\bullet, \bullet}$ (Definition \ref{def - the double complex}) will be a free resolution of $\mathscr{D}_{X}[s] f^{s-1}$. We will eventually use this double complex to accomplish our first goal: to, under our main assumptions, ``compute'' $\Ext_{\mathscr{D}_{X}[s]}^{3}(\mathscr{D}_{X}[s] f^{s-1}, \mathscr{D}_{X}[s])$. More accurately, this $\Ext$ module will be the middle object of a short exact sequence where we know the outer objects explicitly. 

This is all viable because in the case of locally quasi-homogeneous divisors $f \in R = \mathbb{C}[x_{1}, x_{2}, x_{3}]$, Walther proved \cite[Theorem 3.26]{uli} that the $\mathscr{D}_{X}[s]$-annihilator of $f^{s}$ is generated by total order at most one operators, i.e. by $\ann_{\mathscr{D}_{X}[s]}^{(1)} f^{s}$. (Actually his result is more general, cf. loc. cit. and our Proposition \ref{prop - U complex resolves}.) Thus the following canonical surjection is actually an isomorphism:
\begin{equation} \label{eqn - uli result, ann fs}
\frac{\mathscr{D}_{X}[s]}{\ann_{\mathscr{D}_{X}[s]}^{(1)} f^{s}} \simeq \mathscr{D}_{X}[s] f^{s} \quad \text{ for ``nice'' $f$.}
\end{equation}
Our whole construction depends on this isomorphism: if there are more exotic elements of $\mathscr{D}_{X}[s]$ killing $f^{s}$ a totally different approach would be required. On the other hand, if one is interested in $\Ext_{\mathscr{D}_{X}[s]}^{\bullet}(\mathscr{D}_{X}[s] / \ann_{\mathscr{D}_{X}[s]}^{(1)}, \mathscr{D}_{X}[s])$ in its own right, many of our assumptions could presumably be weakened.

Finally, let us discuss our setting. We work in the analytic category since it enables the proof of Proposition \ref{prop - U complex resolves}, see Remark \ref{rmk - why analtyic, U complex}. The usage of analyticity in Proposition \ref{prop - U complex resolves} is the \emph{primary} reason the main results of the paper occur in the analytic category. Certainly $U_f^\bullet$ can be defined in the algebraic category and an algebraic version of Proposition \ref{prop - U complex resolves} would extend the results of this section to such an algebraic $U_f^\bullet$ with appropriate modifications. 

Also: note that a \emph{reducedness} assumption appears in Proposition \ref{prop - dual of Ur complex, terminal cohomology} and its usage here is why reducedness is imposed in the paper's main results. After modifying the statement, reducedness can be removed in Proposition \ref{prop - dual of Ur complex, terminal cohomology}, but the required modifications are technical and lead to less potent results, see Remark \ref{rmk - why reduced}.

\subsection{The Complex}

The first step is to construct a new complex of right $\mathscr{D}_{X}[s]$-modules that is constructed out of logarithmic data of a divisor. This can be thought of as a parametric version of the complex of right $\mathscr{D}_{X}$-modules $\deRham_{X}^{-1} ( M_{X,\log}(L), \nabla^{\alpha})$ considered throughout \cite{BathSaitoTLCT}. To be slightly more explicit: it is a parametric version of M. Saito's $\deRham^{-1}$ functor \cite[Definition 1.3]{InducedDModsDifferentialComplexes} applied to the the twisted logarithmic de Rham complex $(\Omega_{X}(\log f), \nabla^{\lambda})$. (The differential $\nabla^{\lambda}$ is $d(-) + \wedge \lambda df/f \wedge (-)$.) By ``parametric'' we mean by replacing the weight $\lambda$ with a new indeterminate $s$. 

We define the complex for general polynomials $f$ even though we will eventually restrict to the case of $f \in R = \mathbb{C}[x_{1}, x_{2},x_{3}]$ locally quasi-homogeneous. 

\begin{define} \label{def - the Ur complex}
    Let $f \in R = \mathbb{C}[x_{1}, \dots, x_{n}]$ and $X = \mathbb{C}^{n}$. We define the complex $U_{f}^{\bullet}$ of right $\mathscr{D}_{X}[s]$-modules
    \begin{align} \label{equn - def of Ur complex}
    U_{f}^{\bullet} 
    &= 0 \xrightarrow[]{\nabla_{s}} \Omega_{X}^{0}(\log f) \otimes_{\mathscr{O}_{X}} \mathscr{D}_{X}[s] \xrightarrow[]{\nabla_{s}} \Omega_{X}^{1}(\log f) \otimes_{\mathscr{O}_{X}} \mathscr{D}_{X}[s] \xrightarrow[]{\nabla_{s}} \\
    &\cdots \xrightarrow[]{\nabla_{s}} \Omega_{X}^{n}(\log f) \otimes_{\mathscr{O}_{X}} \mathscr{D}_{X}[s]  \xrightarrow[]{\nabla_{s}} 0 \nonumber
    \end{align}
    where the right $\mathscr{D}_{X}[s]$-module structure of
    \begin{equation} \label{eqn - def of Ur complex object}
        U_{f}^{p} = \Omega_{X}^{p}(\log f) \otimes_{\mathscr{O}_{X}} \mathscr{D}_{X}[s]
    \end{equation}
    is given trivially by multiplication on the right and the differential is locally given by 
    \begin{equation} \label{eqn - def of Ur complex differential}
        \nabla_{s}(\eta \otimes P(s)) = \big( d(\eta) \otimes P(s) \big) + \big( \frac{df}{f} \wedge \eta \otimes s P(s) \big) + \big( \sum_{1 \leq i \leq n} dx_{i} \wedge \eta \otimes \partial_{i} P(s) \big).
    \end{equation}
\end{define}

\begin{remark} \label{rmk - coordinate change U complex}
    One can check that Definition \ref{def - the Ur complex} is compatible with coordinate changes. Note that such a change causes $dx_{i}$ and $\partial_{i}$ to change in a dual fashion. On the other hand, a careful reading of future arguments shows we never actually conduct coordinate changes on $U_{f}^{\bullet}$, rather on various associated graded complexes. While we eventually use $U_{f}^{\bullet}$ to compute $\Ext_{\mathscr{D}_{X}[s]}^{n}(\mathscr{D}_{X}[s] f^{s}, \mathscr{D}_{X}[s])$ under our standard assumptions (Theorem \ref{thm - spectral sequence dual}), note that this $\Ext$ module is coordinate independent by construction.
\end{remark}

The reason the complex $U_{f}^{\bullet}$ is useful because under our standard assumptions it resolves its terminal cohomology module. Without any assumptions we always can describe its terminal cohomology module: 

\begin{proposition} \label{prop - terminal cohomology Ur complex}
    We have an isomorphism of right $\mathscr{D}_{X}[s]$-modules
    \[
    \frac{U_{f}^{n}}{\im [ \nabla_{s}: U_{f}^{n-1} \to U_{f}^{n}]} \simeq \left( \frac{\mathscr{D}_{X}[s]}{\ann_{\mathscr{D}_{X}[s]}^{(1)} f^{s-1}} \right)^{r}.
    \]
\end{proposition}

\begin{proof}
    Let $ \eta = (1/f) \sum_{i} a_{i} \widehat{dx_{i}} \in \Omega_{X}^{n-1}(\log f)$. Recall, see Remark \ref{rmk - basics alg log forms}.(e), the $\mathscr{O}_{X}$-isomorphism $\phi: \Omega_{X}^{n-1}(\log f) \to \Der_{X}(- \log f)$ given by
    \begin{equation} \label{eqn - def of iso log der n-1 forms}
    \phi ( \frac{\sum_{i} a_{i} \widehat{dx_{i}}}{f} ) = \sum_{1 \leq i \leq n} (-1)^{i-1} a_{i} \partial_{i}.
    \end{equation}

    We compute $\im [\nabla_{s}: U_{f}^{n-1} \to U_{f}^{n} ]$ recalling that $U_{f}^{n} = \omega_{X}(D) \otimes_{\mathscr{O}_{X}} \mathscr{D}_{X}[s].$ It is enough to consider the effect of $\nabla_{s}$ on $\eta \otimes 1$. Observe:
    \begin{align} \label{eqn - long n-1 nabla computation}
    \nabla_{s}(\eta \otimes 1 )
    &= \big( \sum_{1 \leq i \leq n} (-1)^{i-1}(\partial_{i} \bullet a_{i}) \frac{dx}{f} \otimes 1 \big) + \big( \sum_{1 \leq i \leq n}  (-1)^{i} a_{i} (\partial_{i} \bullet f) \frac{dx}{f^{2}} \otimes 1 \big) \\
    &+ \big( \sum_{1 \leq i \leq n}  (-1)^{i-1} a_{i} (\partial_{i} \bullet f) \frac{dx}{f^{2}} \otimes s \big) +  \big(\sum_{1 \leq i \leq n} (-1)^{i-1} a_{i} \frac{dx}{f} \otimes \partial_{i}  \big) \nonumber \\
    &= \bigg( \frac{dx}{f} \otimes \big( \phi(\eta) + (s - 1) \frac{\phi(\eta) \bullet f}{f} \big) \bigg) + \bigg( \frac{dx}{f} \otimes \big(\sum_{1 \leq i \leq n} (-1)^{i-1} (\partial_{i} \bullet a_{i}) \big) \bigg) \nonumber \\
    &= \frac{dx}{f} \otimes  - \bigg( \tau \big( \phi(\eta) - (s-1) (\frac{\phi(\eta) \bullet f}{f}) \big) \bigg) \nonumber
    \end{align}
    Here: the second ``$=$'' follows from the definition of $\phi$ in \eqref{eqn - def of iso log der n-1 forms}; the third ``$=$'' is the transposition identity Lemma \ref{lem - side changing, order one annihilator}.

    Since $\ann_{\mathscr{D}_{X}[s]}^{(1)} f^{s-1}$ is generated by $\delta - (s-1) (\delta \bullet f) / f $ as $\delta$ varies over all elements of $\Der_{X}(-\log f)$, under the natural right $\mathscr{D}_{X}[s]$-isomorphisms $U_{f}^{n} = \omega_{X}(D) \otimes_{\mathscr{O}_{X}} \mathscr{D}_{X}[s] \simeq \mathscr{D}_{X}[s]$ we can use \eqref{eqn - long n-1 nabla computation} to deduce right $\mathscr{D}_{X}[s]$-isomorphisms
    \[
    \frac{U_{f}^{n}}{\im [ \nabla_{s}: U_{f}^{n-1} \to U_{f}^{n} ] } \simeq \frac{\mathscr{D}_{X}[s]}{\tau(\ann_{\mathscr{D}_{X}[s]}^{(1)} f^{s-1} )} \simeq \left( \frac{\mathscr{D}_{X}[s]}{\ann_{\mathscr{D}_{X}[s]}^{(1)} f^{s-1}} \right)^{r}.
    \]
\end{proof}

Our eventual aim is to compute certain $\Ext_{\mathscr{D}_{X}[s]}^{\bullet}(- , \mathscr{D}_{X}[s])$ modules. Hence the behavior of $U_{f}^{\bullet}$ under duality is salient. We record the following observation here, since it does not require any assumptions on $f$ except \emph{reducedness}:

\begin{proposition} \label{prop - dual of Ur complex, terminal cohomology} Let $f$ be reduced. Then applying $\Hom_{\mathscr{D}_{X}[s]}(-, \mathscr{D}_{X}[s])$ to $U_{f}^{\bullet}$ gives the following complex of left $\mathscr{D}_{X}[s]$-modules:
    \begin{align} \label{eqn - dual of Ur complex}
    \Hom_{\mathscr{D}_{X}[s]}(U_{f}^{\bullet}, \mathscr{D}_{X}[s]) \simeq 0 
    &\xleftarrow[]{} \mathscr{D}_{X}[s] \otimes_{\mathscr{O}_{X}} \Hom_{\mathscr{O}_{X}}(\Omega_{X}^{0}(\log f), \mathscr{O}_{X})  \xleftarrow[]{} \cdots \\ &\xleftarrow[]{}\mathscr{D}_{X}[s] \otimes_{\mathscr{O}_{X}}  \Hom_{\mathscr{O}_{X}}(\Omega_{X}^{p}(\log f), \mathscr{O}_{X})  \cdots \nonumber \\
    &\xleftarrow[]{} \mathscr{D}_{X}[s] \otimes_{\mathscr{O}_{X}} \Hom_{\mathscr{O}_{X}}(\Omega_{X}^{n}(\log f), \mathscr{O}_{X}). \nonumber
    \end{align}
Moreover, the leftmost non-trivial cohomology of the above is given by:
\begin{equation} \label{eqn - dual Ur complex, zeroeth cohomology}
    H^{0}(\Hom_{\mathscr{D}_{X}[s]}(U_{f}^{\bullet}, \mathscr{D}_{X}[s])) \simeq \frac{\mathscr{D}_{X}[s]}{\ann_{\mathscr{D}_{X}[s]}^{(1)}f^{-s}}.
\end{equation}
\end{proposition}

\begin{proof}
    That $\Hom_{\mathscr{D}_{X}[s]}(U_{f}^{\bullet}, \mathscr{D}_{X}[s])$ is a complex of left $\mathscr{D}_{X}[s]$-modules is standard. As for computing each element of the complex, we have
    \begin{align*}
        \Hom_{\mathscr{D}_{X}[s]}(\Omega_{X}^{p}(\log f) \otimes_{\mathscr{O}_{X}} \mathscr{D}_{X}[s], \mathscr{D}_{X}[s]) 
        &\simeq \Hom_{\mathscr{O}_{X}}(\Omega_{X}^{p}(\log f), \Hom_{\mathscr{D}_{X}[s]}(\mathscr{D}_{X}[s], \mathscr{D}_{X}[s])) \\
        &\simeq \Hom_{\mathscr{O}_{X}}(\Omega_{X}^{p}(\log f), \mathscr{D}_{X}[s]) \\
        &\simeq \Hom_{\mathscr{O}_{X}}(\Omega_{X}^{p}(\log f), \mathscr{D}_{X}[s] \otimes_{\mathscr{O}_{X}} \mathscr{O}_{X}) \\
        &\simeq \mathscr{D}_{X}[s] \otimes_{\mathscr{O}_{X}} \Hom_{\mathscr{O}_{X}}(\Omega_{X}^{p}(\log f), \mathscr{O}_{X}), \nonumber
    \end{align*}
    where the left $\mathscr{D}_{X}[s]$-module structure is the obvious one. (Note that: the first ``$\simeq$`` is tensor-hom adjunction; the fourth ``$\simeq$'' is a consequence of the flatness of $\mathscr{D}_{X}[s]$ over $\mathscr{O}_{X}$.)

    As for \eqref{eqn - dual Ur complex, zeroeth cohomology} and computing $H^{0}$, we must compute the cokernel of
    \begin{equation} \label{eqn - dual of Ur, zeroeth cohomology, pf pt 1}
    \Hom_{\mathscr{D}_{X}[s]}(\Omega_{X}^{0}(\log f) \otimes_{\mathscr{O}_{X}} \mathscr{D}_{X}[s]) \xleftarrow[]{} \Hom_{\mathscr{D}_{X}[s]}(\Omega_{X}^{1}(\log f) \otimes_{\mathscr{O}_{X}} \mathscr{D}_{X}[s]).
    \end{equation}
    In the original complex $U_{f}^{\bullet}$ we know that
    \begin{align} \label{eqn - dual of Ur, zeroeth cohomology, pf pt 2}
    \Omega_{X}^{0}(\log f) \otimes_{\mathscr{O}_{X}} \mathscr{D}_{X}[s] 
    &\ni 1 \otimes 1 \\
    &\mapsto ( \frac{df}{f} \otimes s) + (\sum_{i} dx_{i} \otimes \partial_{i}) \nonumber \\
    &\in \Omega_{X}^{1}(\log f) \otimes_{\mathscr{O}_{X}} \mathscr{D}_{X}[s]. \nonumber
    \end{align}
    Using the perfect pairing between logarithmic derivations and logarithmic one forms, see Remark \ref{rmk - basics alg log forms}.(f), 
    \begin{align*}
    P \otimes \delta \in  \mathscr{D}_{X}[s] \otimes_{\mathscr{O}_{X}} \Der_{X}(-\log f) 
    &\simeq \mathscr{D}_{X}[s] \otimes_{\mathscr{O}_{X}} \Hom_{\mathscr{O}_{X}}(\Omega_{X}^{1}(\log f), \mathscr{O}_{X}) \\
    &\simeq \Hom_{\mathscr{D}_{X}[s]}(\Omega_{X}^{1}(\log f) \otimes_{\mathscr{O}_{X}} \mathscr{D}_{X}[s], \mathscr{D}_{X}[s])
    \end{align*}
    acts on the image of the element $1$ inside $\Omega_{X}^{1}(\log f) \otimes_{\mathscr{O}_{X}} \mathscr{D}_{X}[s]$ as computed in \eqref{eqn - dual of Ur, zeroeth cohomology, pf pt 2} via
    \begin{align*}
    (P \otimes \delta ) \left( ( \frac{df}{f} \otimes s) + (\sum_{i} dx_{i} \otimes \partial_{i}) \right) 
    &= P \left( ( \iota_{\delta}(\frac{df}{f}) s) + (\sum_{i} \iota_{\delta}(dx_{i}) \partial_{i} \right) \\
    &= P ( s \frac{\delta \bullet f}{f} + \delta) \\
    &= P \big( \delta - (-s \frac{\delta \bullet f}{f}) \big),
    \end{align*}
    where we use $\delta$ to denote the same element in $\Der_{X} \subseteq \mathscr{D}_{X}[s]$. It follows that the cokernel of \eqref{eqn - dual of Ur, zeroeth cohomology, pf pt 1} is exactly the left $\mathscr{D}_{X}[s]$-module $\mathscr{D}_{X}[s] / \ann_{\mathscr{D}_{X}[s]}^{(1)} f^{-s}.$
    \end{proof}

\begin{remark} \label{rmk - why reduced}
    \emph{(Reducedness)} Proposition \ref{prop - dual of Ur complex, terminal cohomology} has a obvious non-reduced version using \cite[Proposition A.8]{BathANote} instead of Remark \ref{rmk - basics alg log forms}.(f). But, this is not compatible with a nice end formula like $\mathscr{D}_{X}[s] f^{-s}$. The careful reader will note that \emph{all} of our results can be reproduced for $f$ non-reduced with suitable changes, but these changes are less aesthetic: the leftmost object of Theorem \ref{thm - spectral sequence dual} changes; as does the partial symmetry within Theorem \ref{thm - localized symmetry of BS poly}.
\end{remark}

\subsection{Conditions for a Resolution}

    We now require additional assumptions on $f$. One consequence of our eventual assumption of $f \in R = \mathbb{C}[x_{1}, x_{2}, x_{3}]$ reduced and locally quasi-homogeneous, is that it will force all but the terminal cohomology $U_{f}^{\bullet}$ to vanish, i.e. it will force $U_{f}^{\bullet}$ to be a resolution of $(\mathscr{D}_{X}[s] / \ann_{\mathscr{D}_{X}[s]}^{(1)} f^{s-1})^{r}$, cf. Proposition \ref{prop - terminal cohomology Ur complex}. However, we can get the same result with weaker assumptions. Hence:

\begin{define} \label{def - tame}
    We say that a divisor $D \subseteq X$ is \emph{locally tame at $\mathfrak{x} \in X$} when, for any defining equation germ $f$ at $\mathfrak{x}$, we have
    \[
    \pdim_{\mathscr{O}_{X,\mathfrak{x}}} \Omega_{X, \mathfrak{x}}^{p} (\log f) \leq p.
    \]
    We say $D$ is \emph{tame} when it is tame locally everywhere. 
\end{define}

\begin{remark} \label{rmk - tame in three dim}
    The tameness condition is immediate for $p = 0, n$ since there the logarithmic forms are free $\mathscr{O}_{X}$-modules. When $X = \mathbb{C}^{3}$, any divisor $D \subseteq X$ is tame: since $\Omega_{X,\mathfrak{x}}^{p}(\log f)$ is reflexive (\cite[Proposition 1.5]{HollandMondLCTIsolated}, \cite[Lemma 2.1]{ComplexDualityChernClassesLogForms}), its projective dimension is at most $3 - 2 = 1$.
\end{remark}

Here is the promised resolution result for $U_{f}^{\bullet}$:

\begin{proposition} \label{prop - U complex resolves}
    Suppose that $f \in R = \mathbb{C}[x_{1}, \dots, x_{n}]$ defines a tame, Saito-holonomic divisor. Then $U_{f}^{\bullet}$ resolves $(\mathscr{D}_{X}[s] / \ann_{\mathscr{D}_{X}[s]}^{(1)} f^{s-1})^{r}$, that is, the augmented complex
    \[
    U_{f}^{\bullet} \to \left( \frac{\mathscr{D}_{X}[s]}{\ann_{\mathscr{D}_{X}[s]}^{(1)} f^{s-1}} \right)^{r}
    \]
    is acyclic. If in addition $f$ is locally quasi-homogeneous, then $U_{f}^{\bullet}$ resolves $\mathscr{D}_{X}[s]f^{s-1}$. In particular, $U_{f}^{\bullet}$ resolves $\mathscr{D}_{X}[s]f^{s-1}$ when $f \in \mathbb{C}[x_{1}, x_{2}, x_{3}]$ is locally quasi-homogeneous.
\end{proposition} 

\begin{proof}
    First of all, the ``in particular claim'' is well-known: such a divisor is tame by Remark \ref{rmk - tame in three dim}; such a divisor also has a finite logarithmic stratification, see Remark \ref{rmk - log stratification}.(c). The claim of the penultimate sentence follows by \cite[Theorem 3.26]{uli} which gives, in this setting, the natural isomorphism 
    \[
    \frac{\mathscr{D}_{X}[s]}{\ann_{\mathscr{D}_{X}[s]}^{(1)} f^{s-1}} \xrightarrow[]{\simeq} \mathscr{D}_{X}[s] f^{s-1}.
    \]

    We are left to justify the second sentence. Each $U_{f}^{p}$ can be given a decreasing filtration by $s$-degree in such a way to make $U_{f}^{\bullet}$ a filtered complex. Moreover, the resultant associated graded complex $\gr U_{f}^{\bullet}$ has (locally) objects and maps
    \begin{align*}
        \Omega_{X}^{p}(\log f) \otimes_{\mathscr{O}_{X}} \mathscr{D}_{X}[s]
        &\ni \eta \otimes P(s) \\
        &\mapsto \bigg(d(\eta) \otimes P(s) \bigg)  + \bigg( \sum_{i} dx_{i} \wedge \eta \otimes \partial_{i} P(s) \bigg) \\
        &\in \Omega_{X}^{p+1}(\log f) \otimes_{\mathscr{O}_{X}} \mathscr{D}_{X}[s].
    \end{align*}
    (To be explicit, first define a decreasing filtration on $U_f^p$ by $s$-degree: $G^{t,p} (U_f^p) = \Omega_X^p(\log f) \otimes_{\mathscr{O}_X} \mathscr{D}_{X} \otimes_{\mathbb{C}} \mathbb{C}[s]_{\geq t}$. Note that $G^{t+1, p} (U_f^p) \subseteq G^{t, p}(U_f^p)$. Now set
    \begin{equation*}
        G^t(U_f^\bullet) =  \left[ \cdots \xrightarrow[]{} G^{t, p-1}(U_f^p) \to G^{t, p}(U_f^p) \to G^{t,p+1}(U_f^{p+1}) \to \cdots \right] 
    \end{equation*}
    with differentials induced by those of $U_f^\bullet$. The aforementioned decreasing filtration on $U_f^\bullet$ is $\cdots \subseteq G^{t+1}(U_f^\bullet) \subseteq G^t(U_f^\bullet) \subseteq \cdots $.) By standard spectral sequence considerations, it suffices to show that $H^{k} (\gr U_{f}^{\bullet}) = 0$ for all $k \neq n$.

    We sketch a proof of this fact exploiting the local finiteness of the logarithmic strata, i.e. Saito-holonomicity. We refer the reader to \cite[Theorem 3.9]{uli} or \cite[Proposition 3.11]{NoncommutativePeskineSzpiro} for more details about Saito-holonomic induction; our proof here has similarities to the cited arguments. 

    First assume $\dim X = 1$. That $H^k(\gr U_f^\bullet) = 0$ for all $k \neq n$ (i.e. for $k=0)$ can be checked by hand since $\gr U_f^0$ is cyclically generated by $1 \otimes 1$ (recall that $\Omega_X^0(\log f) = \Omega_X^0$). Now assume that $\dim X = n$ and the claim holds on smaller dimensional spaces. Fix $\mathfrak{x} \in X$. If $\mathfrak{x}$ belongs to a positive dimensional logarithmic stratum, the Saito-holonomic assumption guarantees a local analytic isomorphism $(X, D, \mathfrak{x}) \simeq (\mathbb{C}, \mathbb{C}, 0) \times (X^{\prime}, D^{\prime}, \mathfrak{x}^{\prime})$ where: $D \subseteq X$ is the divisor germ of $f$ at $\mathfrak{x}$; $\dim X^{\prime} < X$; $D^{\prime} \subseteq X^{\prime}$ is the divisor germ of some $f^{\prime}$ at $\mathfrak{x}^{\prime} \in X^{\prime}$, This local analytic isomorphism arises exactly as in \ref{rmk - log stratification}.(c), since we have assumed there is a non-vanishing logarithmic vector field at $\mathfrak{x}$. Remark \ref{rmk - log stratification}.(c) says that Saito-holonomicity is preserved along such products. And there is a K{\"u}nneth formula for logarithmic forms in such situations (see, for example, \cite[Lemma 3.10]{CriticalPointsAndResonanceHyperplaneArrangements}), meaning tameness is also preserved under such products. So $(X, D^\prime, \mathfrak{x}^\prime)$ falls into our inductive set-up. Since this K{\"u}nneth formula is compatible with our complex $\gr U_f^\bullet$, the induction hypotheses implies the desired cohomological vanishing at $\mathfrak{x}$. 

    We deduce that, for all $k \neq n$, any non-vanishing $H^k(\gr U_f^\bullet)$ has zero-dimensional support. We conclude $H^k(\gr U_f^\bullet)$ vanishes for all $k \neq 0$ by invoking \cite[Lemma 3.6]{NoncommutativePeskineSzpiro}, see also \cite[Remark 3.7]{NoncommutativePeskineSzpiro}, which is sort of a non-commutative variant of the Peskine-Szpiro acyclicity lemma.
    \end{proof}

\begin{remark} \label{rmk - why analtyic, U complex} (Analytic versus algebraic category)
    The complex $U_f^\bullet$ can also be defined in the algebraic category by replacing the analytic $\mathscr{D}_X[s]$ with a (a polynomial ring extension) of the algebraic sheaf of $\mathbb{C}$-linear differential operators (or the Weyl algebra) and replacing $\mathscr{O}_X$ with the algebraic structure sheaf (or $R$). However the proof of Proposition \ref{prop - U complex resolves} requires analyticity: the locally quasi-homogeneous assumptions guarantees an analytic isomorphism $(X, D , \mathfrak{x}) \simeq (\mathbb{C}, \mathbb{C}, 0) \times (X^\prime, D^\prime, \mathfrak{x}^\prime)$ at suitable $\mathfrak{x}$, but the map need not be algebraic. The isomorphism is explicitly constructed in \cite[(3.5), (3.6)]{SaitoLogarithmicForms} where its analyticity is evident in equation $(\star \star)$ of $(3.5)$. Note that when $D$ is a hyperplane arrangement, Proposition \ref{prop - U complex resolves} holds algebraically, since the aforementioned isomorphisms are linear and algebraic. 
\end{remark}

\begin{remark} \label{rmk - on U complex resolution prop}
    \noindent \begin{enumerate}[label=(\alph*)]
        \item While \cite[Theorem 3.26]{uli} says $\ann_{\mathscr{D}_{X}[s]} f^{s} = \ann_{\mathscr{D}_{X}[s]}^{(1)} f^{s}$, the proof actually shows $\ann_{\mathscr{D}_{X}} f^{s}$ is generated $\Der_{X}(-\log_{0} f)$. Using the weighted-homogeneity, it follows that $\ann_{\mathscr{D}_{X}[s]} f^{\alpha s + \beta} = \ann_{\mathscr{D}_{X}[s]}^{(1)} f^{\alpha s + \beta}$, for any $\alpha \in \mathbb{C}^{\star}$, $\beta \in \mathbb{C}$, under the assumptions of loc. cit.
        \item The complex $U_{f}^{\bullet}$ can be thought of as a generalization of the (parametric) Spencer complex for free divisors in \cite[Section 5]{ModulefsLocallyQuasiHomogeneousFree}, \cite[Appendix]{DualityApproachSymmetryBSpolys}. However, their constructions intrinsically utilize freeness.
        \item There is a filtration of $U_{f}^{\bullet}$ so that the attached associated graded complex recovers the Liouville complex of \cite{uli} up to: replacing $\Omega_{X}^{p}(\log_{0} f)$ with $\Omega_{X}^{p}(\log f)$; a (cosmetic) faithfully flat tensor with $\mathbb{C}[s]$.
        \item One can also define a multivariate version of $U_{f}$ where $s df/f$ is replaced with $\sum s_{k} df_{k}/f_{k}$ where $f = f_{1} \cdots f_{r}$ is some factorization of $f$, $s_{1}, \dots, s_{r}$ new variables. A two step filtration process recovers, as the eventual associated graded complex, the Logarithmic complex of \cite{CriticalPointsAndResonanceHyperplaneArrangements}.
    \end{enumerate}
\end{remark}

\subsection{The Double Complex}

In this subsection we finally restrict to $f \in R = \mathbb{C}[x_{1}, x_{2}, x_{3}]$ reduced and locally quasi-homogeneous. We will construct a double complex $K_{f}^{\bullet, \bullet}$ out of $U_{f}$; we will use this double complex to start a computation of $\Ext_{\mathscr{D}_{X}[s]}^{3}((\mathscr{D}_{X}[s] f^{s})^{r}, \mathscr{D}_{X}[s])$. Actually, using a standard spectral sequence approach we will uncover what map of $\mathscr{D}_{X}[s]$-modules we need to understand to perform this computation (see Lemma \ref{lem - spectral sequence dual, page 2}). So the $\Ext$ module computation will have to wait until the next section, wherein we analyze this map. 

To set-up the double complex's foundation, recall (Proposition \ref{prop - analytic free res, log 1, 2 forms}, Definition \ref{def - fixing analytic free res, log 1, 2 forms}) our fixed free $\mathscr{O}_{X}$-resolutions
\[
Q_{\bullet}^{2} = 0 \to Q_{1}^{2} \to Q_{0}^{2} \to \Omega_{X}^{2}(\log f) \to 0,
\]
\[
Q_{\bullet}^{1} = 0 \to Q_{1}^{1} \to Q_{0}^{2} \to \Omega_{X}^{1}(\log f) \to 0.
\]
As $\mathscr{D}_{X}[s]$ is a flat $\mathscr{O}_{X}$-module, $Q_{\bullet}^{2} \otimes_{\mathscr{O}_{X}} \mathscr{D}_{X}[s]$ and $Q_{\bullet}^{1} \otimes_{\mathscr{O}_{X}} \mathscr{D}_{X}[s]$ and free, right $\mathscr{D}_{X}[s]$-resolutions of $\Omega_{X}^{2}(\log f) \otimes_{\mathscr{O}_{X}} \mathscr{D}_{X}[s]$ and $\Omega_{X}^{1}(\log f) \otimes_{\mathscr{O}_{X}} \mathscr{D}_{X}[s]$ respectively.

By general facts about projective modules, we can lift the differentials of the all but terminally acyclic $U_{f}^{\bullet}$ (Proposition \ref{prop - U complex resolves}) along the data of our preferred $\D_{X}[s]$-resolutions obtaining a double complex. Thus we are entitled to define:

\begin{define} \label{def - the double complex}
    Let $f \in R = \mathbb{C}[x_{1}, x_{2}, x_{3}]$ be locally quasi-homogeneous. Consider a double complex $K_{f}^{\bullet, \bullet}$ of free right $\mathscr{D}_{X}[s]$-modules and maps
    \begin{equation} \label{eqn - def, the double complex}
K_{f}^{\bullet, \bullet} = 
\begin{tikzcd}
    \enspace
        & 0 \dar{} \rar{}
            & 0 \dar{}
                & \enspace \\
    0 \dar \rar{}
        & Q_{1}^{1} \otimes_{\mathscr{O}_{X}} \mathscr{D}_{X}[s] \dar{\nu_{1}^{1}} \rar{\sigma_{1}^{1}}
            & Q_{1}^{2} \otimes_{\mathscr{O}_{X}} \mathscr{D}_{X}[s] \dar{\nu_{1}^{2}} \rar{}
                & 0 \dar \\
    Q_{0}^{0} \otimes_{\mathscr{O}_{X}} \mathscr{D}_{X}[s] \arrow[d, dotted] \rar{\sigma_{0}^{0}}
        & Q_{0}^{1} \otimes_{\mathscr{O}_{X}} \mathscr{D}_{X}[s] \arrow[d, dotted] \rar{\sigma_{0}^{1}}
            & Q_{0}^{2} \otimes_{\mathscr{O}_{X}} \mathscr{D}_{X}[s] \arrow[d, dotted] \rar{\sigma_{0}^{2}}
                & Q_{0}^{3} \otimes_{\mathscr{O}_{X}} \mathscr{D}_{X}[s] \arrow[d, dotted] \\
    U_{f}^{0} \rar{\nabla_{s}}
        & U_{f}^{1} \rar{\nabla_{s}}
            & U_{f}^{2} \rar{\nabla_{s}}
                & U_{f}^{3}
\end{tikzcd}
\end{equation}
such that
\begin{enumerate}[label=(\roman*)]
    \item $K^{p,q} = 0$ for all $p, q \neq 0, 1$;
    \item $K^{p, q} = Q_{q}^{p} \otimes_{\mathscr{O}_X} \mathscr{D}_X[s]$; 
    \item the vertical differentials are $\nu_{q}^{p}: K^{p, q} \to K^{p, q-1}$;
    \item the horizontal differentials are $\sigma_{q}^{p} : K^{p,q} \to K^{p+1, q}$;
    \item the augmentation of $K^{1,\bullet}$ is/induces the acyclic complex
    \[
    0 \to Q_{1}^{1} \otimes_{\mathscr{O}_{X}} \mathscr{D}_{X}[s] \xrightarrow[]{\nu_{1}^{1}} Q_{0}^{1} \otimes_{\mathscr{O}_{X}} \mathscr{D}_{X}[s] \xrightarrow[]{} \Omega_{X}^{1}(\log f) \otimes_{\mathscr{O}_{X}} \mathscr{D}_{X}[s] \xrightarrow[]{} 0
    \]
    \item the augmentation of $K^{2, \bullet}$ is/induces the acyclic complex
    \[
    0 \to Q_{1}^{2} \otimes_{\mathscr{O}_{X}} \mathscr{D}_{X}[s] \xrightarrow[]{\nu_{1}^{2}} Q_{0}^{2} \otimes_{\mathscr{O}_{X}} \mathscr{D}_{X}[s] \xrightarrow[]{} \Omega_{X}^{2}(\log f) \otimes_{\mathscr{O}_{X}} \mathscr{D}_{X}[s] \xrightarrow[]{} 0
    \]
    \item all the maps of \eqref{eqn - def, the double complex} are right $\mathscr{D}_{X}[s]$-linear and all the squares of \eqref{eqn - def, the double complex} commute.
\end{enumerate}
Finally, use $\Tot(K_{f}^{\bullet, \bullet})$ to denote the total complex of $K^{\bullet, \bullet}$.
\end{define}

\begin{remark} \label{rmk - choice of lift not unique}
    Note that $K_{f}^{\bullet, \bullet}$ depends on the \emph{choice} of the lifts $\sigma_{0}^{1}$ and $\sigma_{1}^{1}$ in \eqref{eqn - def, the double complex}. Since the columns of $K_{f}^{\bullet, \bullet}$ are free resolutions, such lifts do exist. When we refer to $K_{f}^{\bullet, \bullet}$ we implicitly make a choice of lifts; the point is that such a choice is possible.
\end{remark}

Definition \ref{def - the double complex} is useful to us because:

\begin{proposition} \label{prop - total complex free resolution}
    Let $f \in R = \mathbb{C}[x_{1}, x_{2}, x_{3}]$ be locally quasi-homogeneous. Then $\Tot(K_{f}^{\bullet, \bullet})$ is a free (right) resolution of $(\mathscr{D}_{X}[s] f^{s-1})^{r}$.
\end{proposition}

\begin{proof}
    Certainly $\Tot(K_{f}^{\bullet, \bullet})$ is a complex of free (right) $\mathscr{D}_{X}[s]$-modules. We compute the cohomology of $\Tot(K_{f}^{\bullet, \bullet})$ by using the spectral seqeuence that first computes cohomology vertically. By the construction, the first page of the spectral sequence has only non-zero entries on the $x$-axis and this row is exactly $U_{f}^{\bullet}$. The second page of the spectral sequence is $H^{\bullet}(U_{f}^{\bullet})$, as positioned on the $x$-axis. By Proposition \ref{prop - U complex resolves} all these cohomologies vanish except for \[
    H^{n}(U_{f}^{\bullet}) \simeq \frac{\mathscr{D}_{X}[s]}{\ann_{\mathscr{D}_{X}[s]}^{(1)} f^{s-1}} \xrightarrow[]{\simeq} \mathscr{D}_{X}[s] f^{s-1}.
    \]
\end{proof}

\begin{remark} \label{rmk - why analytic, total complex} (Algebraic versus analytic category)
    We only know that Proposition \ref{prop - total complex free resolution} holds in the analytic category since the proof utilizes Proposition \ref{prop - U complex resolves}, see Remark \ref{rmk - why analtyic, U complex}
\end{remark}

For our computation of $\Ext_{\mathscr{D}_{X}[s]}^{3}((\mathscr{D}_{X}[s])^{r}, \mathscr{D}_{X}[s])$ we will $\Hom$ the total complex $\Tot(K_{f}^{\bullet, \bullet})$ into $\mathscr{D}_{X}[s]$. This gives a double complex vulnerable to spectral sequence tactics. The $\infty$-page of our spectral sequence will describe the associated graded pieces of a filtration on $\Ext_{\mathscr{D}_{X}[s]}^{\bullet}((\mathscr{D}_{X}[s])^{r}, \mathscr{D}_{X}[s])$ and it is from this our first main result Theorem \ref{thm - spectral sequence dual} follows.

Understanding this spectral sequence will require a fair amount of work. The opening salvo, which makes explicit our required labors, requires \emph{reducedness} and is as follows:

\begin{lemma} \label{lem - spectral sequence dual, page 2}
    Let $f \in R = \mathbb{C}[x_{1}, x_{2}, x_{3}]$ be reduced and locally quasi-homogeneous. Let $E_{v}^{\bullet, \bullet}$ denote the $v^{\text{th}}$ page of the spectral sequence of the double complex of left $\mathscr{D}_{X}[s]$-modules $\Hom_{\mathscr{D}_{X}[s]}(\Tot(K_{f}^{\bullet, \bullet}), \mathscr{D}_{X}[s])$ where one first computes cohomology vertically and then horizontally. Then the second page of the spectral sequence is
    \begin{equation} \label{eqn - lemma, spectral sequence dual, page 2}
        E_{2}^{\bullet, \bullet} = \begin{tikzcd}
            0 
                & E_{2}^{1, 1} = \text{\normalfont cokernel} 
                    & E_{2}^{2, 1} = \text{\normalfont kernel} \arrow[dll] \
                        & 0 \\
            E_{2}^{0,0} = \frac{\mathscr{D}_{X}[s]}{\ann_{\mathscr{D}_{X}[s]}^{(1)} f^{-s}} 
                & ?
                    & ? 
                        & ?\\
        \end{tikzcd}
    \end{equation}
    where: all entries outside of \eqref{eqn - lemma, spectral sequence dual, page 2} are zero; the $E_{2}^{1,1}$ and $E_{2}^{2,1}$ entries refer to the cokernel and kernel, respectively, of the natural map
    \begin{equation} \label{eqn - lemma, spectral sequence dual, map on page 2}
        \mathscr{D}_{X}[s] \otimes_{\mathscr{O}_{X}} \Ext_{\mathscr{O}_{X}}^{1}(\Omega_{X}^{2}(\log f), \mathscr{O}_{X}) \xrightarrow[]{\Hom_{\mathscr{D}_{X}[s]}(\sigma_{1}^{1}, \mathscr{D}_{X}[s])} \mathscr{D}_{X}[s] \otimes_{\mathscr{O}_{X}} \Ext_{\mathscr{O}_{X}}^{1}(\Omega_{X}^{1}(\log f), \mathscr{O}_{X}).
    \end{equation}
\end{lemma} 

\begin{proof}
     First consider the complex $\Hom_{\mathscr{D}_{X}[s]}(\Tot(K_{f}), \mathscr{D}_{X}[s])$, which can be viewed as the total complex of a double complex. The $p^{\text{th}}$-column is of the form $\Hom_{\mathscr{D}_{X}[s]}(Q_{\bullet}^{p} \otimes_{\mathscr{O}_X} \mathscr{D}_{X}[s], \mathscr{D}_{X}[s])$. By arguing as in the first display equation of \eqref{eqn - dual Ur complex, zeroeth cohomology}, we may rewrite this complex of left $\mathscr{D}_{X}[s]$-modules as $\mathscr{D}_{X}[s] \otimes_{\mathscr{O}_{X}} \Hom_{\mathscr{O}_{X}}(Q_{\bullet}^{p}, \mathscr{O}_{X})$. Since $\mathscr{D}_{X}[s] \otimes_{\mathscr{O}_{X}} -$ is flat and hence exact, applying this functor commutes with taking cohomology. By the construction of $Q_{\bullet}^{p}$ we deduce the only nonzero entries on the first page $E_{1}^{\bullet, \bullet}$ of the spectral sequence are:
    \begin{align} \label{eqn - computing ext, first page spectral, pt 1}
    E_{1}^{p, 1} &= \begin{cases} 
    \mathscr{D}_{X}[s] \otimes_{\mathscr{O}_{X}} \Ext_{\mathscr{O}_{X}}^{1}(\Omega_{X}^{p}(\log f), \mathscr{O}_{X}) \text{ for } p = 1,2
    \\ 0 \text{ otherwise}; 
    \end{cases} \\
    E_{1}^{p,0} &= \begin{cases}
        \mathscr{D}_{X}[s] \otimes_{\mathscr{O}_{X}} \Hom_{\mathscr{O}_{X}}(\Omega_{X}^{p}(\log f), \mathscr{O}_{X}).
    \end{cases} \nonumber
    \end{align}

    To compute $E_{2}^{\bullet, \bullet}$ we compute cohomology of the rows of the first page. Proposition \ref{prop - dual of Ur complex, terminal cohomology} computes the terminal cohomology of row $E_{1}^{0, \bullet}$: this gives $E_{2}^{0,0}$ as described in \eqref{eqn - lemma, spectral sequence dual, page 2}. As for row $E_{1}^{1, \bullet}$ we are tasked with computing cohomology of the complex
    \begin{align*} \label{eqn - computing ext, first page spectral, part 2}
        0 \xleftarrow[]{} \mathscr{D}_{X}[s] &\otimes_{\mathscr{O}_{X}} \Ext_{\mathscr{O}_{X}}^{1}(\Omega_{X}^{1}(\log f), \mathscr{O}_{X}) \\
        &\xleftarrow[]{\Hom_{\mathscr{D}_{X}[s]}(\sigma_{1}^{1}, \mathscr{D}_{X}[s])} \mathscr{D}_{X}[s] \otimes_{\mathscr{O}_{X}} \Ext_{\mathscr{O}_{X}}^{1}(\Omega_{X}^{2}(\log f), \mathscr{O}_{X}) \xleftarrow[]{} 0, \nonumber
    \end{align*}
   which agrees with the description of $E_{2}^{1,1}$ and $E_{2}^{2,1}$ in \eqref{eqn - lemma, spectral sequence dual, page 2}, \eqref{eqn - lemma, spectral sequence dual, map on page 2}.
\end{proof}

Lemma \ref{lem - spectral sequence dual, page 2} makes it clear we must understand the aforementioned map \eqref{eqn - lemma, spectral sequence dual, map on page 2}. It will turn out this map is injective (see the proof of Theorem \ref{thm - spectral sequence dual}). Therefore $E_{2}^{1,1}$ and $E_{2}^{0,0}$ of \eqref{eqn - lemma, spectral sequence dual, page 2} will be the associated graded pieces of a filtration on $\Ext_{\mathscr{D}_{X}[s]}^{3}((\mathscr{D}_{X}[s] f^{s-1})^{r}, \mathscr{D}_{X}[s])$, which in turn describe this $\Ext$ module in terms of a short exact sequence.

\section{Lifting Maps}

In this section we perform the ground work for our computation/short exact sequence description of $(\Ext_{\mathscr{D}_{X}[s]}^{3}(\mathscr{D}_{X}[s] f^{s-1}, \mathscr{D}_{X}[s]))^{\ell}$, under our main assumptions on $f$. This is Theorem \ref{thm - spectral sequence dual}. 

The strategy, as developed in the previous section, has been to use the free right $\mathscr{D}_{X}[s]$ resolution $\Tot(K_{f}^{\bullet, \bullet})$ of $(\mathscr{D}_{X}[s] f^{s-1})^{r}$ and understand the dual complex via one of the standard spectral sequences of a double complex. Lemma \ref{eqn - lemma, spectral sequence dual, page 2} makes it clear that we must wrestle with the map 
\begin{equation} \label{eqn - sec 3 intro, sigma dual map}
    \mathscr{D}_{X}[s] \otimes_{\mathscr{O}_{X}} \Ext_{\mathscr{O}_{X}}^{1}(\Omega_{X}^{2}(\log f), \mathscr{O}_{X}) \xrightarrow[]{\Hom_{\mathscr{D}_{X}[s]}(\sigma_{1}^{1}, \mathscr{D}_{X}[s])} \mathscr{D}_{X}[s] \otimes_{\mathscr{O}_{X}} \Ext_{\mathscr{O}_{X}}^{1}(\Omega_{X}^{1}(\log f), \mathscr{O}_{X}).
\end{equation}
The map \eqref{eqn - sec 3 intro, sigma dual map} depends on $\sigma_{1}^{1}$, a particular part of a lift of $\nabla_{s}: U_{f}^{1} \to U_{f}^{2}$ along specific free resolutions, cf. Definition \ref{def - the double complex}, Remark \ref{rmk - choice of lift not unique}. It turns out (\cite[Lemma 00LT]{stacks-project}) that if we replace $\sigma_{1}^{1}$ with another choice of lift along these same specific free resolutions, this does not change the naturally induced map \eqref{eqn - sec 3 intro, sigma dual map}. 

The next two subsections are devoted to constructing a nice choice of lift to replace $\sigma_{1}^{1}$. First algebraically, then analytically. In latter sections we \emph{only} use the analytic version of the lift since we only know that $\Tot(K_{f}^{\bullet, \bullet})$ is a resolution analytically, c.f. Remark \ref{rmk - why analtyic, U complex}, Remark \ref{rmk - why analytic, total complex}. However it is simpler to first construct the lift algebraically and then analytify to obtain Proposition \ref{prop - computing analytic lifts}. (Remark \ref{rmk- why alg then anal lifts} discusses the virtues of this approach.) In the last subsection we begin to study duals of this map in pursuit of our strategy. Throughout, the restriction $R = \mathbb{C}[x_{1}, x_{2}, x_{3}]$ and $X = \mathbb{C}^{3}$ is very important; without it, these lifts are more subtle.

\subsection{Lift of an Algebraic Map}

Throughout this subsection, $f \in R = \mathbb{C}[x_{1}, x_{2}, x_{3}]$ is quasi-homogeneous with weighted-homogeneity $E$. 

We start our construction of a suitable replacement for the lift $\sigma_{1}^{1}$ with an algebraic analogue. We will use the algebraic version to build the analytic version we eventually require. 

While we have defined $U_{f}^{\bullet}$ in the analytic category, we can also define it in the algebraic category. In particular, this will give a map after taking global sections. As prophesized, in this subsection we focus on lifting one such map along free resolutions of the domain and codomain. 

To set this up, let $\mathbb{A}_{3}$ be the Weyl algebra in three variables, i.e. the global sections of the algebraic sheaf of differential operators over $\mathbb{A}_{\mathbb{C}}^{3}$; let $\mathbb{A}_{3}[s]$ be a polynomial ring extension. We will compute explicit lifts of the right $\mathbb{A}_{3}[s]$-linear map 
\begin{equation} \label{eqn - global alg map to lift}
\nabla_{s}: \Omega_{R}^{1}(\log f) \otimes_{R} \mathbb{A}_{3}[s] \to \Omega_{R}^{2}(\log f) \otimes_{R} \mathbb{A}_{3}[s].
\end{equation}
along explicit free resolutions of the domain and codomain.

Recall our fixed graded minimal free resolution of $\Omega_{R}^{2}(\log_{0} f)$
\[
F_{\bullet} = 0 \to F_{1} \xrightarrow[]{\beta_{1}} F_{0} \xrightarrow[]{\beta_{0}} \Omega_{R}^{2}(\log_{0} f) \to 0
\]
and the induced minimal graded free resolutions (cf. Proposition \ref{prop - graded free resolution of log 1, 2 forms}, Definition \ref{def - fixing minimal graded free res, log 1, 2 forms})
\[
P_{\bullet}^{2} = 0 \to P_{1}^{2} = F_{1} \xrightarrow[]{(\beta_{1},0)} P_{0}^{2} = F_{0} \oplus R(-d + \sum w_{i}) \xrightarrow[]{(\beta_{0}, \iota_{E} dx/f)} \Omega_{R}^{2}(\log f) \to 0,
\]
\[
P_{\bullet}^{1} = 0 \to P_{1}^{1} = F_{1} \xrightarrow[]{(\beta_{1}, 0)} P_{0}^{0} = F_{0} \oplus R \xrightarrow[]{(\iota_{E} \beta_{0}, \cdot df/f)} \Omega_{R}^{1}(\log f) \to 0.
\]

Since $\mathbb{A}_{3}[s]$ is flat over $R$ we see that $P_{\bullet}^{2} \otimes \otimes_{R} \Weyl_{3}[s]$ and $P_{\bullet}^{1} \otimes_{R} \Weyl_{3}[s]$ are free resolutions of $\Omega_{R}^{2}(\log f) \otimes_{R} \Weyl_{3}[s]$ and $\Omega_{R}^{1}(\log f) \otimes_{R} \Weyl_{3}[s]$ respectively. We want to construct right $\mathbb{A}_{3}[s]$-maps $\epsilon_{0}$ and $\epsilon_{1}$ so that the following diagram commutes:

    \begin{equation} \label{eqn - comm diagram, lifting nabla with epsilons}
    \begin{tikzcd}
        F_{1} \otimes_{R} \Weyl_{3}[s] \dar{(\beta_{1} \otimes \id, 0)} \rar[dotted]{\epsilon_{1}} 
            & F_{1} \otimes_{R} \Weyl_{3}[s] \dar{(\beta_{1} \otimes \id, 0)} \\
        (F_{0} \otimes_{R} \Weyl_{3}[s]) \oplus (R \otimes \Weyl_{3}[s]) \dar{(\iota_{E}\beta_{0} \otimes \id) \oplus (df/f \cdot \otimes \id)} \rar[dotted]{\epsilon_{0}}
            & (F_{0} \otimes_{R} \Weyl_{3}[s]) \oplus R(n-d) \otimes_{R} \Weyl_{3}[s] \dar{(\beta_{0} \otimes \id, \iota_{E}dx/f \cdot \otimes \id)} \\
        \Omega_{R}^{1}(\log f) \otimes_{R} \mathbb{A}_{3}[s] \rar{\nabla_{s}}
            & \Omega_{R}^{2}(\log f) \otimes_{R} \mathbb{A}_{3}[s].
    \end{tikzcd}
    \end{equation}
(The left column is $P_{\bullet}^{1} \otimes_{R} \mathbb{A}_{3}[s]$ written explicitly, the right is similar.) While the promised lifts of \eqref{eqn - global alg map to lift} exist by general facts about projective resolutions, we will construct lifts so that $\epsilon_{1}$ is essentially given by left multiplication with $\wdeg(f)s + q + E$ where $q$ is degree data depending on the input of $\epsilon_{1}$. The rest of this subsection is devoted to setting up and proving the following:

\begin{proposition} \label{prop - nice lift of nabla s constructed}
    There exist right $\mathbb{A}_{3}[s]$-linear lifts $\epsilon_{0}, \epsilon_{1}$ of $\nabla_{s}: \Omega_{R}^{1}(\log f) \otimes_{R} \mathbb{A}_{3}[s] \to \Omega_{R}^{2}(\log f) \otimes_{R} \mathbb{A}_{3}[s]$ along the (vertical) free resolutions of \eqref{eqn - comm diagram, lifting nabla with epsilons} so that
    \begin{enumerate}[label=(\alph*)]
        \item the squares of \eqref{eqn - comm diagram, lifting nabla with epsilons} commute;
        \item if $b \in F_{1}$ is homogeneous, then 
        \[
        \epsilon_{1}(b \otimes B) = b \otimes (\wdeg(f)s + \wdeg(b) + E) B.
        \]
    \end{enumerate}
\end{proposition}

We lay the foundation for Proposition \ref{prop - nice lift of nabla s constructed}'s proof. Return to the original map $\nabla_{s}$ and recall the direct sum decomposition of $\Omega_{R}^{p}(\log f) = \Omega_{R}^{p}(\log_{0} f) \oplus \iota_{E} \Omega_{R}^{p+1}(\log_{0} f)$. Our strategy depends on the behavior of $\nabla_{s}$ on members of each direct summand, which in turn hinges on a sort of Lie derivative trick. We denote the Lie derivative with respect to the Euler derivation $E$ as $L_{E}$. It is easy to check that our grading conventions on $\Omega_{R}^{p}(\log f)$ guarantee that $L_{E}(\eta) = \wdeg(\eta) \eta$, provided $\eta$ is homogeneous. If $\eta \in \Omega_{R}^{p}(\log f)$ is homogeneous, by basic properties of Lie derivatives and contractions,
    \begin{align} \label{eqn - lift Ur complex, Lie derivative computation}
        \nabla_{s}(\iota_{E}\eta \otimes P) 
        &= d(\iota_{E} \eta) \otimes P + df/f \wedge \iota_{E} \eta \otimes sP + \sum_{i} \big( dx_{i} \wedge \iota_{E} \eta \otimes \partial_{i} P \big) \\
        &= \bigg( L_{E}(\eta) - \iota_{E} d (\eta) \bigg) \otimes P + \bigg( \iota_{E}(df/f) \wedge \eta - \iota_{E}(df/f \wedge \eta) \bigg) \otimes sP \nonumber \\
        &+ \sum_{i} \bigg( \big( \iota_{E}(dx_{i}) \wedge \eta - \iota_{E}(dx_{i} \wedge \eta) \big) \otimes \partial_{i} P \bigg) \nonumber \\
        & = \bigg( L_E(\eta) \otimes P + (\iota_E(df/f) \wedge \eta) \otimes sP + \sum_i \big( (\iota_E(dx_i) \wedge \eta) \otimes \partial_i P \big) \bigg) \nonumber \\
        & - \bigg( \iota_E d(\eta) \otimes P + \iota_E( df/f \wedge \eta) \otimes sP + \sum_i \big( \iota_E(dx_i \wedge \eta) \otimes \partial_i P \big) \bigg) \nonumber \\
        & = \bigg( \eta \otimes ( \wdeg(f)s + \wdeg(\eta) + E) P \bigg) 
        - \bigg( \iota_E (\nabla_s(\eta \otimes P)) \bigg). \nonumber 
    \end{align}
(In the expression $\iota_E (\nabla_s(\eta \otimes P))$ of the final line, we are implicitly letting the interior derivative act on all terms of the complex $\Omega_R^\bullet(\log f) \otimes_R \Weyl_3[s]$, by linearly extending its action across $\Weyl_3[s]$.) Note that if $\eta \in \Omega_R^2(\log_0 f)$, then in the final line of \eqref{eqn - lift Ur complex, Lie derivative computation} the first summand lives in the direct summand $\Omega_{R}^{2}(\log_{0} f) \otimes_R \Weyl_3[s]$ whereas the second summand in the direct summand $\iota_{E}\Omega_{R}^{3}(\log f) \otimes_R \Weyl_3[s]$. (See Remark \ref{rmk - basics alg log forms}.(b).)

Now we return to the lifting problem where we compute $\epsilon_{1}$ in two steps. We first restrict the domain of $\Omega_{R}^{2}(\log f) \otimes_{R} \Weyl_{3}[s] \xrightarrow[]{\nabla_{s}} \Omega_{R}^{3}(\log f) \otimes_{R} \Weyl_{3}[s]$ by considering the induced map $\iota_{E}\Omega_{R}^{2}(\log_{0} f) \otimes_{R} \mathbb{A}_{3}[s] \xrightarrow[]{\nabla_{s}} \Omega_{R}^{2}(\log f) \otimes_{R} \mathbb{A}_{3}$, cf. Remark \ref{rmk - basics alg log forms}.(b). We compute certain lifts $\mu_{0}, \mu_{1}$ of the induced map $\nabla_{s}$ as depicted below:
     \begin{equation} \label{eqn - comm diag, lifting nabla along restricted domain}
    \begin{tikzcd}
        F_{1} \otimes_{R} \Weyl_{3}[s] \dar{(\beta_{1} \otimes \id)} \rar[dotted]{\mu_{1}} 
            & F_{1} \otimes_{R} \Weyl_{3}[s] \dar{(\beta_{1} \otimes \id, 0)} \\
        F_{0} \otimes_{R} \Weyl_{3}[s] \dar{(\iota_{E}\beta_{0} \otimes \id)} \rar[dotted]{\mu_{0}}
            & (F_{0} \otimes_{R} \Weyl_{3}[s]) \oplus R(n-d) \otimes_{R} \Weyl_{3}[s] \dar{(\beta_{0} \otimes \id, \iota_{E}dx/f \cdot \otimes \id)} \\
        \iota_{E}\Omega_{R}^{2}(\log_{0} f) \otimes_{R} \mathbb{A}_{3}[s] \rar{\nabla_{s}}
            & \Omega_{R}^{2}(\log f) \otimes_{R} \mathbb{A}_{3}[s]
    \end{tikzcd}
    \end{equation}

    \begin{proposition} \label{prop - construction of mu 0}
        Consider the map $\mu_{0} : F_{0} \otimes_{R} \Weyl_{3}[s] \to (F_{0} \otimes_{R} \Weyl_{3}[s]) \oplus R(n-d) \otimes_{R} \Weyl_{3}[s]$ which, for $a \in F_{0}$ homogeneous, is given by
        \[
        \mu_{0}(a \otimes A) = \bigg( a \otimes (\wdeg(f)s + \wdeg(a) + E)A \bigg) - \bigg( (\iota_{E} dx/f \otimes \id)^{-1} \big( (\iota_{E} \circ \nabla_{s})(\beta_{0}(a) \otimes A) \big) \bigg)
        \]
        is well-defined, right $\mathbb{A}_{3}[s]$-linear, and makes the lower square of \eqref{eqn - comm diag, lifting nabla along restricted domain} commute.
    \end{proposition}

    \begin{proof}
        First note that $(\iota_{E} dx/f \otimes \id) : R(n-d) \otimes_{R} \mathbb{A}_{3}[s] \to \iota_{E} \Omega_{R}^{3}(\log f) \otimes_{R} \mathbb{A}_{3}[s]$ is an isomorphism, so its inverse uniquely selects an element. To confirm well-definedness, since $\nabla_{s}$ and $\iota_{E}$ are well-defined right-$\mathbb{A}_{3}[s]$ maps, we must check the assignment
        \begin{equation} \label{eqn - mu 0, degree summand assignment}
        a \otimes A \mapsto a \otimes (\wdeg(f)s + \wdeg(a) + E) A
        \end{equation}
        is well-defined. In other words, that if $a \in F_{0}$ is homogeneous and $r \in R$ homogeneous, then \eqref{eqn - mu 0, degree summand assignment} has the same output on $ar \otimes A$ and $a \otimes rA$. This reduces to the commutator identity $rE = Er - [E,r] = Er - E \bullet r = Er - \wdeg(r)$ and the obvious fact $\wdeg(ra) = \wdeg(a) + \wdeg(r)$.

        Now we confirm $\mu_{0}$ makes the lower square of \eqref{eqn - comm diag, lifting nabla along restricted domain} commute. Let $a \in F_{0}$ be homogeneous. On one hand,
        \begin{align*}
        \big( \nabla_{s} &\circ (\iota_{E} \beta_{0} \otimes \id) \big) (a \otimes A) \\
        &= \nabla_{s} \big( \iota_{E} \beta_{0} (a) \otimes A \big) \\
        &= \big( \beta_{0}(a) \otimes (\wdeg(f)s + \wdeg(\beta_{0}(a)) + E) A \big) - \big( \iota_{E} (\nabla_{s} (\beta_{0}(a) \otimes A)) \big) \\
        &= \big( \beta_{0}(a) \otimes (\wdeg(f)s + \wdeg(a) + E) A \big) - \big( \iota_{E} (\nabla_{s} (\beta_{0}(a) \otimes A)) \big).
        \end{align*}
        Here: the second ``$=$'' is \eqref{eqn - lift Ur complex, Lie derivative computation}; the third ``$=$'' uses that $\beta_{0}$ is degree preserving. On the other hand,
        \begin{align*}
            (\beta_{0} \otimes \id , \iota_{E} df/f \otimes \id) \big( \mu_{0}(a \otimes A) \big) 
            &= \big( \beta_{0}(a) \otimes (\wdeg(f)s + \wdeg(a) + E)A \big) \\
            &- (\iota_{E} \circ \nabla_{s}) (\beta_{0}(a) \otimes A).
        \end{align*}
    \end{proof}

    \begin{proposition} \label{prop - construction of mu 1}
        Consider the map $\mu_{1}: F_{1} \otimes_{R} \mathbb{A}_{3}[s] \to F_{1} \otimes_{R} \mathbb{A}_{3}[s]$ which, for $b \in F_{1}$ homogeneous, is given by
        \[
        \mu_{1}(b \otimes B) = b \otimes (\wdeg(f)s + \wdeg(b) + E)B.
        \]
        Then $\mu_1$ is well-defined and right $\mathbb{A}_{3}[s]$-linear. Moreover, with $\mu_{0}$ as defined in Proposition \ref{prop - construction of mu 0}, $\mu_{1}$ makes the upper square of \eqref{eqn - comm diag, lifting nabla along restricted domain} commute.
    \end{proposition}

    \begin{proof}
        The well-definedness of $\mu_{1}$ follows exactly from the commutator identity used in checking the well-definedness of $\mu_{0}$, cf. Proposition \ref{prop - construction of mu 0}. To check the upper square of \eqref{eqn - comm diag, lifting nabla along restricted domain} commutes, lets compute. On one hand,
        \begin{align*}
            \big( \mu_{0} \circ (\beta_{1} \otimes \id) \big) (b \otimes B) 
            &= \mu_{0} \big( \beta_{1}(b) \otimes B) \\
            &= \beta_{1}(b) \otimes (\wdeg(f)s + \wdeg(\beta_{1}(b)) + E)B \\
            &= \beta_{1}(b) \otimes (\wdeg(f)s + \wdeg(b) + E)B.
        \end{align*}
        Here: the second ``$=$'' uses the fact $(\beta_{0} \circ \beta_{1})(b) = 0$; the third ``$=$'' uses the fact $\beta_{1}$ preserves degree. On the other hand,
        \begin{align*}
            \big((\beta_{1} \otimes \id, 0) \circ \mu_{1} \big) (b \otimes B) 
            &= (\beta_{1} \otimes \id, 0) \big( b \otimes (\wdeg(f) s + \wdeg(b) + E)B \big) \\
            &= \beta_{1}(b) \otimes (\wdeg(f) s + \wdeg(b) + E)B.
        \end{align*}   
    \end{proof}

Now we can return to the original problem of constructing $\epsilon_{0}, \epsilon_{1}$ and filling in \eqref{eqn - comm diagram, lifting nabla with epsilons}. Restrict the domain of $\nabla_{s}$ in a different way by considering $\nabla_{s}: \Omega_{R}^{1}(\log_{0} f) \otimes_{\mathscr{O}_{X}} \mathscr{D}_{X}[s] \to \Omega_{R}^{2}(\log f) \otimes_{\mathscr{O}_{X}} \mathscr{D}_{X}[s]$, cf. Remark \ref{rmk - basics alg log forms}.(b). By basic facts about projective resolutions, there \emph{exists} a right $\mathbb{A}_{3}[s]$-linear maps $\nu_{0}$ so that the following diagram commutes:

\begin{equation} \label{eqn - comm diag, lifting nabla restricted to killing log 1 forms, def of nu 0}
    \begin{tikzcd}
        0 \dar{0} \rar[dotted]{0} 
            & F_{1} \otimes_{R} \Weyl_{3}[s] \dar{(\beta_{1} \otimes \id, 0)} \\
        R \otimes_{R} \Weyl_{3}[s] \dar{(df/f \otimes \id)} \rar[dotted]{\nu_{0}}
            & (F_{0} \otimes_{R} \Weyl_{3}[s]) \oplus R(n-d) \otimes_{R} \Weyl_{3}[s] \dar{(\beta_{0} \otimes \id, \iota_{E}dx/f \cdot \otimes \id)} \\
        \Omega_{R}^{1}(\log_{0} f) \otimes_{R} \mathbb{A}_{3}[s] \rar{\nabla_{s}}
            & \Omega_{R}^{2}(\log f) \otimes_{R} \mathbb{A}_{3}[s].
    \end{tikzcd}
    \end{equation}
    (By Remark \ref{rmk - basics alg log forms}.(d) the left column in\eqref{eqn - comm diag, lifting nabla restricted to killing log 1 forms, def of nu 0} is a free resolution; by Proposition \ref{prop - graded free resolution of log 1, 2 forms}; Definition \ref{def - fixing minimal graded free res, log 1, 2 forms} the same holds for the right column.)

We are ready to prove the subsection's main result Proposition \ref{prop - nice lift of nabla s constructed}:

\vspace{5mm}

\emph{Proof of Proposition \ref{prop - nice lift of nabla s constructed}:}
    First define 
    \[
    \epsilon_{0} : (F_{0} \otimes_{R} \mathbb{A}_{3}[s]) \oplus (R \otimes_{R} \mathbb{A}_{3}[s]) \to (F_{0} \otimes_{R} \mathbb{A}_{3}[s]) \oplus (R(n-d) \otimes_{R} \mathbb{A}_{3}[s])
    \]
    by $\epsilon_{0} = \mu_{0} \oplus \nu_{0}$, where $\mu_{0}$ is as defined in Proposition \ref{prop - construction of mu 0} and $\nu_{0}$ is as promised by \eqref{eqn - comm diag, lifting nabla restricted to killing log 1 forms, def of nu 0}. As the summands of $\epsilon_{0}$ are right $\mathbb{A}_{3}[s]$-linear, so is $\epsilon_{0}$; as the summands of $\epsilon_{0}$ make the lowest squares of \eqref{eqn - comm diag, lifting nabla along restricted domain} and \eqref{eqn - comm diag, lifting nabla restricted to killing log 1 forms, def of nu 0} commute, this choice of $\epsilon_{0}$ makes the lowest square of \eqref{eqn - comm diagram, lifting nabla with epsilons} commute. 

    Now define
    \[
    \epsilon_{1} : F_{1} \otimes_{R} \mathbb{A}_{3}[s] \to F_{1} \otimes_{R} \mathbb{A}_{3}[s]
    \]
    by $\epsilon_{1} = \mu_{1}$. Proposition \eqref{prop - construction of mu 1} shows $\mu_{1}$ is right $\mathbb{A}_{3}[s]$-linear and that the upper square of \eqref{eqn - comm diag, lifting nabla along restricted domain} commute. By the definition of $\epsilon_{0}$, the commutativity of said square is equivalent to the commutativity of the upper square of \eqref{eqn - comm diagram, lifting nabla with epsilons}. We have checked (a). That (b) holds is by definition of $\epsilon_{1} = \mu_{1}$, see Proposition \ref{prop - construction of mu 1}.
\qed

\subsection{Lift of an Analytic Map and its Behavior on $\Ext$}

In this subsection $f \in R = \mathbb{C}[x_{1}, x_{2}, x_{3}]$ is quasi-homogeneous with weighted-homogeneity $E$. 

We return to the analytic setting. We use the previous subsection to compute a lift of the map
\begin{equation} \label{eqn - analytic map to lift}
\nabla_{s} : \Omega_{X}^{1}(\log f) \otimes_{\mathscr{O}_{X}} \mathscr{D}_{X}[s] \xrightarrow[]{} \Omega_{X}^{2}(\log f) \otimes_{\mathscr{O}_{X}} \mathscr{D}_{X}[s]
\end{equation}
along explicit free resolutions of the domain and codomain. 

Again, we have set fixed a minimal graded free resolution of $\Omega_{R}^{2}(\log f)$
\[
F_{\bullet} = 0 \to F_{1} = \bigoplus_{i} R(b_{i}) \xrightarrow[]{\beta_{1}} F_{0} = \bigoplus_{j} R(a_{j}) \xrightarrow[]{\beta_{0}} \Omega_{R}^{2}(\log_{0} f) \to 0
\]
which induces (cf. Proposition \ref{prop - graded free resolution of log 1, 2 forms}, Definition \ref{def - fixing minimal graded free res, log 1, 2 forms}, Proposition \ref{prop - analytic free res, log 1, 2 forms}, Definition \ref{def - fixing analytic free res, log 1, 2 forms}) free $\mathscr{O}_{X}$-resolutions 
\[
Q_{\bullet}^{2} = 0 \to Q_{1}^{2} \to Q_{0}^{2} \to \Omega_{X}^{2}(\log f) \to 0,
\]
\[
Q_{\bullet}^{1} = 0 \to Q_{1}^{1} \to Q_{0}^{2} \to \Omega_{X}^{1}(\log f) \to 0.
\]
As we saw earlier, tensoring on the right with $\otimes_{\mathscr{O}_{X}} \mathscr{D}_{X}[s]$ preserves the property of ``being a resolution'' thanks to flatness. So we know there are lifts of \eqref{eqn - analytic map to lift} along free resolutions so that the following diagram commutes:
\begin{equation} \label{eqn - diagram, analytic lift of maps constructed}
    \begin{tikzcd}
        Q_{1}^{1} \otimes_{\mathscr{O}_{X}} \mathscr{D}_{X}[s] \rar{\theta_{1}} \dar{}
            & Q_{1}^{2} \otimes_{\mathscr{O}_{X}} \mathscr{D}_{X}[s] \dar{} \\
        Q_{0}^{1} \otimes_{\mathscr{O}_{X}} \mathscr{D}_{X}[s] \rar{\theta_{0}} \dar{}
            & Q_{0}^{2} \otimes_{\mathscr{O}_{X}} \mathscr{D}_{X}[s] \dar{} \\
        \Omega_{X}^{1}(\log f) \otimes_{\mathscr{O}_{X}} \mathscr{D}_{X}[s] \rar{\nabla_{s}}
            & \Omega_{X}^{2}(\log f) \otimes_{\mathscr{O}_{X}} \mathscr{D}_{X}[s].
    \end{tikzcd}
\end{equation}

More specifically:

\begin{proposition} \label{prop - computing analytic lifts}
    In the set-up of this section, recall that: $F_{1} = \bigoplus_{i} R(b_{i})$; $Q_{1}^{2} = Q_{1}^{1}$ are the (algebraic) sheafification and analytification of $F_{1}$. Let $\textbf{e}_{t}$ be the $t^{\text{th}}$ indicator vector of $Q_{1}^{2} = Q_{1}^{1}$. Then we may construct right $\mathscr{D}_{X}[s]$-maps $\theta_{1}, \theta_{0}$ making \eqref{eqn - diagram, analytic lift of maps constructed} commute where 
    \begin{equation} \label{eqn - description of theta 1}
    \theta_{1} (\textbf{e}_{t} \otimes 1) = \textbf{e}_{t} \otimes (\wdeg(f)s - b_{t} + E).
    \end{equation}
\end{proposition}

\begin{proof}
    The ``recall that'' statement is just Definition \ref{def - fixing minimal graded free res, log 1, 2 forms}, Definition \ref{def - fixing analytic free res, log 1, 2 forms}. Our task is to compute suitable $\theta_{0}, \theta_{1}$. 
    
    To do this, set $\theta_{j}$ to be the map induced by the $\epsilon_{j}$ constructed in the previous subsection making \eqref{eqn - comm diagram, lifting nabla with epsilons} commute. Indeed, $P_{j}^{1} \otimes_{R} \mathbb{A}_{3}[s]$ is a free $\mathbb{A}_{3}[s]$-module and so $\epsilon_{j}$ can be defined by its behavior on a basis, reducing it to left matrix multiplication. Then each $\theta_{j}$ is given by the same map. (The matrices defining $\epsilon_{j}$ are globally defined with (homogeneous) polynomial entries.) That the squares of \eqref{eqn - diagram, analytic lift of maps constructed} commute amounts to checking a matrix identity, which we know is true by Proposition \ref{prop - nice lift of nabla s constructed}. Finally, $\epsilon_{1}$ sends $\textbf{e}_{t} \otimes 1$, where $\textbf{e}_{t}$ is the $t^{\text{th}}$ indicator vector of $F_{1}$ to $\textbf{e}_{t} \otimes (\wdeg(f) s - b_{t} + E)$, since $\textbf{e}_{t} \in F_{1} = \bigoplus_{i} R(b_{i})$ has degree $-b_{i}$. So the matrix defining $\epsilon_{1}$ is a diagonal matrix with $(t,t)$-entry $(\wdeg(f)s - b_{t} + E)$ and the same matrix will define $\theta_{1}$.

    This confirms the upper square of \eqref{eqn - diagram, analytic lift of maps constructed} commutes and confirms \eqref{eqn - description of theta 1}. We will be done once we show the bottom square of \eqref{eqn - diagram, analytic lift of maps constructed} commutes. But since $f$ is a polynomial, $\Omega_{X}^{1}(\log f)$ is (locally everywhere) generated by polynomials. So the commutativity of the lower square of \eqref{eqn - diagram, analytic lift of maps constructed} follows from the commutativity of the lower square of \eqref{eqn - comm diag, lifting nabla along restricted domain} and the construction of $\theta_{0}$. 
\end{proof}

\begin{remark} \label{rmk- why alg then anal lifts} (Algebraic versus analytic category)
    It seems possible to prove Proposition \ref{prop - computing analytic lifts} without any appeal to the algebraic arguments of the previous subsection. We do not do this because:
    \begin{enumerate}[label=(\alph*)]
        \item Proceeding only analytically seems to require repeating the Lie derivative trick  \eqref{eqn - lift Ur complex, Lie derivative computation} and Proposition \ref{prop - construction of mu 0} in the analytic category, which in turn requires many technical facts about the direct product, not direct sum, decompositions of $\Omega_X^j(\log f)$ and $\mathscr{O}_X$ into ``quasi-homogeneous'' pieces. Similar issues of quasi-homogeneity involving the interplay between formal and analytic power series appear in \cite{BathTLCT}, especially subsection 2.3. These issues mostly supplant explanatory value with notation.
        \item When $f$ defines a hyperplane arrangement, Proposition \ref{prop - U complex resolves} holds in the algebraic category (see the last sentence of Remark \ref{rmk - why analtyic, U complex}). So in the arrangement case, one can mimic this entire manuscript staying in the algebraic category alone. This is advantageous since the graded data of logarithmic $p$-forms of arrangements is well studied. In future work we hope to exploit this and revisit the lifting problem in the algebraic category for arrangements in $\mathbb{C}^n$, with $n$ arbitrary.
    \end{enumerate}
\end{remark}

As our eventual aim is to understand the behavior of lifts of \eqref{eqn - analytic map to lift} after applying $\Hom_{\mathscr{D}_{X}[s]}(- , \mathscr{D}_{X}[s])$ we need the following: 

\begin{proposition} \label{prop - computing analytic lifts, dualizing}
    In the set-up of Proposition \ref{prop - computing analytic lifts}, applying $\Hom_{\mathscr{D}_{X}[s]}(- , \mathscr{D}_{X}[s])$ to \eqref{eqn - diagram, analytic lift of maps constructed} gives the map of left $\mathscr{D}_{X}[s]$-modules
    \begin{equation} \label{eqn - theta 1 dual, pt 1}
        \mathscr{D}_{X}[s] \otimes_{\mathscr{O}_{X}} \Hom_{\mathscr{O}_{X}}(Q_{1}^{1}, \mathscr{O}_{X}) \xleftarrow[]{\Hom_{\mathscr{D}_{X}[s]}(\theta_{1}, \mathscr{D}_{X}[s])} \mathscr{D}_{X}[s] \otimes_{\mathscr{O}_{X}} \Hom_{\mathscr{O}_{X}}(Q_{1}^{2}, \mathscr{O}_{X}).
    \end{equation}
    If $\textbf{e}_{t}$ is the $t^{\text{th}}$-indicator vector of $\mathscr{D}_{X}[s] \otimes_{\mathscr{O}_{X}} \Hom_{\mathscr{O}_{X}}(Q_{1}^{2}, \mathscr{O}_{X}) = \mathscr{D}_{X}[s] \otimes_{\mathscr{O}_{X}} \Hom_{\mathscr{O}_{X}}(Q_{1}^{1}, \mathscr{O}_{X})$ then
    \begin{equation} \label{eqn - theta 1 dual, pt 2}
        \Hom_{\mathscr{D}_{X}[s]}(\theta_{1}, \mathscr{D}_{X}[s]) \bigg( 1 \otimes \textbf{e}_{t} \bigg) = (\wdeg(f) s - b_{t} + E) \otimes \textbf{e}_{t}.
    \end{equation}
\end{proposition}

\begin{proof}
    That $\Hom_{\mathscr{D}_{X}[s]}(Q_{1}^{j} \otimes_{\mathscr{O}_{X}} \mathscr{D}_{X}[s], \mathscr{D}_{X}[s]) \simeq \mathscr{D}_{X}[s] \otimes \Hom_{\mathscr{O}_{X}}( Q_{1}^{j}, \mathscr{O}_{X})$ follows exactly as in the first part of Proposition \ref{prop - dual of Ur complex, terminal cohomology}. Then \eqref{eqn - description of theta 1} implies \eqref{eqn - theta 1 dual, pt 2}.
\end{proof}

The preceding result immediately gives an induced map on $\mathscr{D}_{X}[s] \otimes_{\mathscr{O}_{X}} \Ext$-modules. Since these $\Ext$ modules vanish everywhere except, possibly, the origin, (Proposition \ref{prop - finite hom gen set of Ext modules}), it suffices to describe the map at $0$:

\begin{proposition} \label{prop - computing analytic lifts, induced map on Ext}
    Stay in the set-up of Proposition \ref{prop - computing analytic lifts} and Proposition \ref{prop - computing analytic lifts, dualizing}. Let $\Delta_{f}$ be a finite homogeneous generating set of $\Ext_{R}^{1}(\Omega_{R}^{1}(\log f), R)$ as a $\mathbb{C}$-vector space; for $u \in \Delta_{f}$ let $\wdeg(u)$ denotes its degree. Then $\Delta_{f}$ is a finite homogeneous generating set of both $\Ext_{\mathscr{O}_{X,0}}^{1}(\Omega_{X,0}^{1}(\log f), \mathscr{O}_{X,0})$ and $\Ext_{\mathscr{O}_{X,0}}^{1}(\log f, \mathscr{O}_{X,0})$ and the induced map
    \begin{align} \label{eqn - map for lift on analytic ext modules}
        \mathscr{D}_{X,0}[s] &\otimes_{\mathscr{O}_{X,0}} \Ext_{\mathscr{O}_{X,0}}^{1}(\Omega_{X,0}^{1}(\log f), \mathscr{O}_{X,0}) \\
        &\xleftarrow[]{\Hom_{\mathscr{D}_{X,0}[s]}(\theta_{1}, \mathscr{D}_{X,0}[s])} \mathscr{D}_{X,0}[s] \otimes_{\mathscr{O}_{X,0}} \Ext_{\mathscr{O}_{X,0}}^{1}(\Omega_{X,0}^{2}(\log f), \mathscr{O}_{X,0}) \nonumber
    \end{align}
    is given, and completely determined, by 
    \begin{equation} \label{eqn - fml for lift on analytic ext modules}
    \Hom_{\mathscr{D}_{X,0}[s]}(\theta_{1}, \mathscr{D}_{X,0}[s]) \bigg( 1 \otimes u \bigg) = (\wdeg(f)s - \wdeg(u) + E) \otimes u
    \end{equation}
    for all $u \in \Delta_{f}$.
\end{proposition}

\begin{proof}
    Thanks to Proposition \ref{prop - finite hom gen set of Ext modules}, all we must prove is \eqref{eqn - fml for lift on analytic ext modules}. (As $\Delta_{f}$ will generate both $\Ext$ modules, this formula would completely determine the map \eqref{eqn - map for lift on analytic ext modules}.) Take $u \in \Delta_{f}$ viewed inside 
    \[
    \Ext_{\mathscr{O}_{X,0}}^{1}(\Omega_{X, 0}^{2}(\log f), \mathscr{O}_{X,0}) \simeq H^{1} (\Hom_{\mathscr{O}_{X,0}}(Q_{\bullet}^{2}, \mathscr{O}_{X,0})).
    \]
    Then we can find weighted homogeneous polynomials $p_{t,k}$ such that $u = \overline{\sum_{t,k} p_{t,k} \textbf{e}_{t}}$ where $\textbf{e}_{t}$ is the $t^{\text{th}}$-indicator vector of $\Hom_{\mathscr{O}_{X,0}}(Q_{1}^{2}, \mathscr{O}_{X})$ and $\overline{(-)}$ denotes the element in the corresponding quotient computing $H^{1} (\Hom_{\mathscr{O}_{X,0}}(Q_{\bullet}^{2}, \mathscr{O}_{X,0}))$. Then
    \begin{align*}
        \Hom_{\mathscr{D}_{X,0}[s]}(\theta_{1}, \mathscr{D}_{X,0}[s]) \bigg( 1 \otimes u \bigg) 
        &= \Hom_{\mathscr{D}_{X,0}[s]}(\theta_{1}, \mathscr{D}_{X,0}[s]) \bigg( \sum_{t,k} p_{t,k} \otimes \overline{\textbf{e}_{t}} \bigg) \\
        &= \sum_{t,k} \big( p_{t,k}(\wdeg(f)s - b_{t} + E) \big) \otimes \overline{\textbf{e}_{t}} \\
        &= \sum_{t,k} \big( \wdeg(f)s - b_{t} - \wdeg(p_{t,k}) + E \big) p_{t,k} \otimes \overline{\textbf{e}_{t}} \\
        &= (\wdeg(f)s - \wdeg(n) + E) \otimes u.
    \end{align*}
    (The second ``$=$'' uses Proposition \ref{prop - computing analytic lifts, dualizing}; the third ``$=$'' uses $[E, p_{t,k}] = \wdeg(p_{t,k}) p_{t,k}$; the fourth ``$=$'' uses $\wdeg(n) = \wdeg(p_{t,k}) + \wdeg(\textbf{e}_{t}) = \wdeg(p_{t,k}) + b_{t}$ for all valid $k$, where in this parenthetical we are thinking of $\textbf{e}_{t}$ as the indicator vector in the algebraic setting so we can compute weighted degrees.)
\end{proof}

\section{Computing Duals}

In this section we finish our short exact sequence description of $\Ext_{\mathscr{D}_{X}[s]}^{3}(\mathscr{D}_{X}[s] f^{s-1}, \mathscr{D}_{X}[s])^{\ell}$ under our main hypotheses, see Theorem \ref{thm - spectral sequence dual}. We saw in Lemma \ref{lem - spectral sequence dual, page 2} this required understanding the dual of a particular lifted map in $\Tot(K_{f}^{\bullet, \bullet})$; in Proposition \ref{prop - computing analytic lifts, dualizing} and Proposition \ref{prop - computing analytic lifts, induced map on Ext} we constructed a particular lift, described its dual map, and described the induced map on $\Ext$. 

In the first subsection we study the cokernel of this induced map on $\Ext$. As this cokernel appeared as entry $E_{2}^{1,1}$ in our spectral sequence analysis, cf. Lemma \ref{lem - spectral sequence dual, page 2}, we must understand it completely to understand $\Ext_{\mathscr{D}_{X}[s]}^{3}(\mathscr{D}_{X}[s] f^{s-1}, \mathscr{D}_{X}[s])^{\ell}$. In the second subsection we finally compute this $\Ext$ module in terms of a short exact sequence, cf. Theorem \ref{thm - spectral sequence dual}. All of our applications rely on this computation. Because we use Propostion \ref{prop - U complex resolves} to compute duals, in this section (and the rest of the manuscript) we work entirely in the analytic category, cf. Remark \ref{rmk - why analtyic, U complex}, Remark \ref{rmk - why analytic, total complex}.

\subsection{A Study in Cokernel}

We study the cokernel of the map in Proposition \ref{prop - computing analytic lifts, induced map on Ext}. (As the domain and codomain vanish at other stalks (Proposition \ref{prop - finite hom gen set of Ext modules}), the description of Proposition \ref{prop - computing analytic lifts, induced map on Ext} can be extended to the sheaf level trivially.) Recall this map naturally arises from our specific choice of a nice lift $\theta_{1}$ of $U_{f}^{1} \to U_{f}^{2}$ along our specific free resolutions of $U_{f}^{1}$ and $U_{f}^{2}$. 

Let us bestow a name on our cokernel:

\begin{define} \label{def - the cokernel, ext dual object}
    For a quasi-homogeneous $f \in R = \mathbb{C}[x_{1}, x_{2}, x_{3}]$ we denote by $L_{f}$ the left $\mathscr{D}_{X}[s]$ module that is the cokernel
    \begin{align} \label{eqn - def, map for lift on analytic ext modules}
        L_{f} = \text{cokernel}: \bigg[ \Hom_{\mathscr{D}_{X}[s]}(\theta_{1}, \mathscr{D}_{X}[s]): 
        &\mathscr{D}_{X}[s] \otimes_{\mathscr{O}_{X}} \Ext_{\mathscr{O}_{X}}^{1}(\Omega_{X}^{2}(\log f), \mathscr{O}_{X}) \to \\
        &\mathscr{D}_{X}[s] \otimes_{\mathscr{O}_{X}} \Ext_{\mathscr{O}_{X}}^{1}(\Omega_{X}^{1}(\log f), \mathscr{O}_{X}) \bigg] \nonumber
    \end{align}
    Recall $\theta_{1}$ is as constructed in Proposition \ref{prop - computing analytic lifts} and the map appearing in \eqref{eqn - def, map for lift on analytic ext modules} was explicitly studied in Proposition \ref{prop - computing analytic lifts, dualizing}, Proposition \ref{prop - computing analytic lifts, induced map on Ext}.
\end{define}

\begin{remark} \label{rmk - cokernel object, endomorphism description}
    Using analytic versions of the isomorphisms within Proposition \ref{prop - weighted data dictionary, ext log 1 forms, local cohomology milnor algebra}, we can realize $L_{f}$ as the cokernel of a $\mathscr{D}_{X}[s]$-endomorphism of $\mathscr{D}_{X}[s] \otimes_{\mathscr{O}_{X}} \Ext_{\mathscr{O}_{X}}^{3}(\mathscr{O}_{X} / (\partial f), \mathscr{O}_{X}).$
\end{remark}

We now show that $L_{f,0}$ has a $b$-function and we are able to compute it exactly. Our formula for $B(L_{f,0})$ is often used in later applications.

\begin{proposition} \label{prop - funny b-function}
    Let $f \in R = \mathbb{C}[x_{1}, x_{2}, x_{3}]$ be quasi-homogeneous with weighted-homogeneity $E = \sum w_{i} x_{i} \partial_{i}$. Then $L_{f}$ vanishes at all stalks $\mathfrak{x} \neq 0$ and at $0$, the stalk $L_{f,0}$ has a b-function. Explicitly,
    \[
    B(L_{f,0}) = \prod_{t \in \wdeg(\Ext_{R}^{1}(\Omega_{R}^{1}(\log f), R))} \left( s - \frac{t + \sum w_{i}}{\wdeg(f)} \right)
    \]
    and we take $B (L_{f,0}) = 1$ if $\wdeg(\Ext_{R}^{1}(\Omega_{R}^{1}(\log f), R)) = \emptyset$.
\end{proposition}

\begin{proof}
    We may assume that $L_{f,0} \neq 0$ as otherwise there is nothing to prove. (The case $L_{f,0} = 0$ is the ``and we take'' part of the proposition.) Select a $\mathbb{C}$-basis $u_{1}, \dots, u_{r}$ of $\Ext_{R}^{1}(\Omega_{R}^{1}(\log f), R)$ of homogeneous elements and regard them as a $\mathbb{C}$-basis of $\Ext_{\mathscr{O}_{X,0}}^{1}(\Omega_{X,0}^{1}(\log f), \mathscr{O}_{X,0})$ as we have been entitled to do. Then we may order the $u_{k}$ so that 
    \begin{equation} \label{eqn - funny b-function, pt 0}
        \mathfrak{m}_{0} \cdot u_{\ell} \in \sum_{k \geq \ell+1} \mathscr{O}_{X,0} \cdot u_{k} \quad \text{ and } u_{\ell} \notin \sum_{k \geq \ell+1} \mathscr{O}_{X,0} \cdot u_{k} \text{ for all } u_{\ell}.
    \end{equation}
    By the ordering \eqref{eqn - funny b-function, pt 0}, $\wdeg(u_{k}) \leq \wdeg(u_{k+1})$ for all $k$. Let $m_{1}, \dots, m_{q}$ denote the indices at which $\wdeg(u_{k})$ jumps, i.e. 
    \begin{equation} \label{eqn - funny b-function, pt 1}
        \wdeg(u_{k}) \leq \wdeg(u_{m_{\ell}}) \enspace \forall k \leq \ell \text{ and } \wdeg(u_{k}) > \wdeg(u_{m_{\ell}}) \enspace \forall k > \ell.
    \end{equation}
    This lets us define a decreasing filtration $L_{f} = N_{0} \supsetneq N_{1} \supsetneq \cdots \supsetneq N_{q} = 0$ of $L_{f}$ where 
    \begin{equation} \label{eqn - funny b-function, pt 2}
    N_{\ell} = \sum_{k > m_{\ell}} \mathscr{D}_{X,0}[s] \cdot \overline{1 \otimes u_{k}} \subseteq L_{f}.
    \end{equation}
    (We take $m_{0} \ll 0$ so that $N_{0} = L_{f}$ is the initial term of the filtration.) 
    We first require:
    
    \vspace{5mm}
    \emph{Claim 1: $N_{\ell-1} / N_{\ell} \neq 0$ for all $1 \leq \ell \leq q$.}

    \noindent \emph{Proof of Claim 1:} Fix such an $\ell$ and select $u_{j} \in \{u_{1}, \dots, u_{r}\}$ such that $\wdeg(u_{j}) = m_{\ell}$. To prove $\emph{Claim 1}$ it suffices to prove that 
    \begin{equation} \label{eqn - funny b-function, claim 1, pt 1}
    0 \neq \overline{1 \otimes u_{j}} \in N_{\ell - 1} / N_{\ell}.
    \end{equation}
    If \eqref{eqn - funny b-function, claim 1, pt 1} were false, then by the definition of $L_{f}$ we could write $1 \otimes u_{\ell} \in \mathscr{D}_{X,0}[s] \otimes_{\mathscr{O}_{X,0}} \Ext_{\mathscr{O}_{X,0}}^{1}(\Omega_{X,0}^{1}(\log f), \mathscr{O}_{X,0})$ as
    \begin{align} \label{eqn - funny b-function, claim 1, pt 2}
        1 \otimes u_{\ell} 
        &= \bigg(\sum_{k} A_{k}(\wdeg(f) s - \wdeg(u_{k}) + E)(1 \otimes u_{k}) \bigg) \\
        &+ \bigg( \sum_{\{j \mid \wdeg(u_{j}) > m_{\ell}\}} A_{j}^{\prime} (1 \otimes u_{j}) \bigg) \nonumber \\
        &= \bigg( \sum_{\{k \mid \wdeg(u_{k}) \leq m_{\ell}\}} A_{k}(\wdeg(f) s - \wdeg(u_{k}) + E)(1 \otimes u_{k}) \bigg) \nonumber \\
        &+ \bigg( \sum_{\{j \mid \wdeg(u_{j}) > m_{\ell}\}} B_{j} (1 \otimes u_{j}) \bigg), \nonumber
    \end{align}
    for suitable $A_{k}, A_{j}^{\prime},$ and $B_{j} \in \mathscr{D}_{X,0}[s]$. Write $A_{k} = \sum_{a} p_{a}(\partial, s) \alpha_{a}(x)$ where $p_{a} \in \mathbb{C}[\partial, s]$ and $\alpha_{a}(x) \in \mathscr{O}_{X,0}$; write $B_{j} = \sum_{b} q_{b}(\partial, s) \beta_{b}(s)$ with the corresponding conventions. Since $\Ext_{\mathscr{O}_{X,0}}^{1}(\Omega_{X,0}(\log f), \mathscr{O}_{X,0})$ is killed by high enough powers of $\mathfrak{m}_{0}$ we may assume that each $\alpha_{a}(x)$ and each $\beta_{b}(x)$ are actually polynomials as opposed to just convergent power series. To wit, write $A_{k} = \sum_{m} p_{m,k}(\partial, s) g_{m,k}$ where $p_{m,j}(\partial, s) \in \mathbb{C}[\partial, s]$ and $g_{m,j} \in R_{m}$; write $B_{j} = \sum_{m} q_{m,k}(\partial_{s}) h_{m,k}$ similarly. (Here $R_{m}$ refers the $w$-grading of $R$.) Using that 
    \begin{align*}
        g_{m,k}E = E g_{m,k} - [E, g_{m,k}] = E g_{m,k} - m g_{m,k} 
        &= (\sum_{i} \partial_{i} w_{i} x_{i} - w_{i}) g_{m,k} - m g_{m,k} \\
        &= ( - m - \sum_{i} w_{i} + \sum_{i} \partial_{i} w_{i} x_{i} ) g_{m,k}
    \end{align*}
    and that our assumption \eqref{eqn - funny b-function, pt 0} implies that $g_{m,k} u_{k} \in \sum_{v > k} \mathscr{O}_{X,0} \cdot u_{v}$ whenever $m > 0$ (along with a similar statement for $h_{m,j} u_{j}$) we deduce that if \eqref{eqn - funny b-function, claim 1, pt 2} holds then we may write $1 \otimes u_{\ell} \in \mathscr{D}_{X,0}[s] \otimes_{\mathscr{O}_{X,0}} \Ext_{\mathscr{O}_{X,0}}^{1}(\Omega_{X,0}^{1}(\log f), \mathscr{O}_{X,0})$ as
    \begin{align} \label{eqn - funny b-function, claim 1, pt 3}
        1 \otimes u_{\ell} 
        &= \sum_{\{k \mid \wdeg(u_{k}) \leq m_{\ell}\}} p_{k}(\partial, s) \bigg(\wdeg(f)s + m_{k} + \sum \partial_{i} w_{i} x_{i} \bigg) (1 \otimes u_{k}) \\
        &= \sum_{\{j \mid \wdeg(u_{j}) > m_{\ell}\}} q_{j}(\partial, s) (1 \otimes u_{j}) \nonumber
    \end{align}
    where $p_{k}(\partial, s), q_{j}(\partial, s) \in \mathbb{C}[\partial,s]$ and $m_{k} \in \mathbb{C}$ for each $k$.

    We will be done with \emph{Claim 1} once we show \eqref{eqn - funny b-function, claim 1, pt 3} is impossible. Let $v$ be the smallest index such that $p_{k}(\partial, s)$ or $q_{j}(\partial, s)$ is nonzero. If $v < \ell$, then $\wdeg(u_{v}) \leq \wdeg(u_{\ell}) = m_{\ell}$ and our assumption is that $p_{v}(\partial, s) \neq 0$. Then rearranging \eqref{eqn - funny b-function, claim 1, pt 3} yields
    \begin{equation*}
        p_{v}(\partial, s) (\wdeg(f)s + m_{v}) (1 \otimes u_{v}) \in \mathbb{C}[s] \sum_{k^{\prime} \geq v + 1} \mathbb{C}[\partial, s] \cdot (1 \otimes u_{k^{\prime}})
    \end{equation*}
    in violation of our ordering \eqref{eqn - funny b-function, pt 0}. If $v = \ell$ then rearringing \eqref{eqn - funny b-function, claim 1, pt 3} yields
    \begin{equation*}
        p_{\ell}(\partial, s)(\wdeg(f)s + m_{\ell} - 1)(1 \otimes u_{\ell}) \in \sum_{k^{\prime} \geq \ell + 1} \mathbb{C}[\partial, s] \cdot (1 \otimes u_{k^{\prime}}).
    \end{equation*}
    Thanks to our ordering assumption \eqref{eqn - funny b-function, pt 0}, this is only possible if $p_{\ell}(\partial, s)(\wdeg(f)s + m_{\ell} - 1) \in \mathbb{C}^{\star}$. But this is impossible since $(\wdeg(f)s + m_{\ell} - 1)$ has positive $s$-degree. Finally, if $v > \ell$, then \eqref{eqn - funny b-function, claim 1, pt 3} implies that
    \[
    1 \otimes u_{\ell} \in \sum_{k^{\prime} \geq \ell + 1} \mathbb{C}[\partial,s] \cdot (1 \otimes u_{k^{\prime}})
    \]
    which is again impossible by our ordering assumption \eqref{eqn - funny b-function, pt 0}. Thus \eqref{eqn - funny b-function, claim 1, pt 3} is false which implies \eqref{eqn - funny b-function, claim 1, pt 1} is true, proving \emph{Claim 1}.

    \vspace{5mm}

    Before proceeding, let us establish some useful identities in $L_{f}$. First, take $u \in \{u_{1}, \dots, u_{r}\}$ is arbitrary. By definition 
    \begin{equation} \label{eqn - funny b-function, pt 3}
        0 = \overline{(\wdeg(f)s - \wdeg(u)) + E \otimes u)} \in L_{f}.
    \end{equation}
    But 
    \begin{align} \label{eqn - funny b-function, pt 4}
    (\wdeg(f)s - \wdeg(u) + E) \otimes u 
    &= (\wdeg(f)s - \wdeg(u) + \sum (\partial_{i} w_{i} x_{i}) - w_{i}) \otimes u) \\
    &= \big( s - \frac{\wdeg(u) + \sum w_{i}}{\wdeg(f)} \big) \otimes u + \sum_{i} \partial_{i} w_{i} x_{i} \otimes u \nonumber.
    \end{align}
    So combining \eqref{eqn - funny b-function, pt 3} and \eqref{eqn - funny b-function, pt 4} we see that
    \begin{equation} \label{eqn - funny b-function, pt 5}
        \left( s - \frac{\wdeg(u) + \sum w_{i}}{\wdeg(f)} \right) (\overline{1 \otimes u}) = - \sum_{i} \partial_{i} \overline{w_{i} x_{i} \otimes u}.
    \end{equation}
    This leads us to:

    \vspace{5mm}

    \emph{Claim 2: For all $1 \leq \ell \leq q$, the $b$-function of $N_{\ell-1} / N_{\ell}$ is }
    \begin{equation}
        B(N_{\ell - 1} / N_{\ell}) = s - \frac{m_{\ell} + \sum w_{i}}{\wdeg(f)}.
    \end{equation}
    \noindent \emph{Proof of Claim 2:} Clearly $N_{\ell - 1}/ N_{\ell}$ is generated as a $\mathscr{D}_{X,0}[s]$-module by the collection of $\overline{1 \otimes u}$ such that $\wdeg(u) = m_{\ell}$; by \emph{Claim 1}, at least one of these generators is nonzero. By \eqref{eqn - funny b-function, pt 5}, the proposed $b$-function of \emph{Claim 2} kills $N_{\ell - 1} / N_{\ell}$. Since $N_{\ell - 1}/N_{\ell} \neq 0$ and it has a nontrivial $\mathbb{C}[s]$-annihilator, its $b$-function must be a non-unit dividing $s - \frac{m_{\ell} + \sum w_{i}}{\wdeg(f)}$; since this polynomial is linear and hence irreducible, its only non-unit divisor is (up to contant) itself. So \emph{Claim 2} is true.

    \vspace{5mm}

    We are finally reading for \emph{Claim 3}, which immediately implies the truth of the proposition:

    \vspace{5mm}

    \emph{Claim 3: For all $0 \leq \ell \leq q-1$, the $b$-function of $N_{\ell}$ is given by}
    \begin{equation} \label{eqn - funny b-function, claim 3 statement}
        B(N_{\ell}) = \prod_{\substack{t \in \{\wdeg(u_{1}), \dots, \wdeg(u_{r}) \} \\ t > m_{\ell} }} \left( s - \frac{t + \sum w_{i}}{\wdeg(f)} \right).
    \end{equation} 
    \noindent \emph{Proof of Claim 3:} We prove this by descending induction on $\ell$. The base case $\ell = q-1$ is \emph{Claim 2} applies to $N_{q-1}/N_{q} = N_{q-1}$. Now assume that \emph{Claim 3} holds for $N_{\ell}$. By combining \emph{Claim 2} and the inductive hypothesis, we see that the $b$-function of $N_{\ell-1}$ must divide the $\mathbb{C}[s]$-polynomial appearing in \eqref{eqn - funny b-function, claim 3 statement}. Since the $\mathbb{C}[s]$-polynomial is a product of pairwise coprime linears, to complete the inductive step it suffices to prove that the zeroes of the $b$-function of $N_{\ell-1}$ agree with the zeroes of the $\mathbb{C}[s]$-polynomial depicted in \eqref{eqn - funny b-function, claim 3 statement}. 
    
    Consider the s.e.s. of $\mathscr{D}_{X,0}[s]$-modules
    \begin{equation} \label{eqn - funny b-function, claim 3, pt 1}
        0 \to N_{\ell } \to N_{\ell - 1} \to N_{\ell - 1} / N_{\ell} \to 0.
    \end{equation}
    Since each of the $\mathscr{D}_{X,0}[s]$-modules of \eqref{eqn - funny b-function, claim 3, pt 1} is supported at $0$, they are all relative holonomic and so the zeroes of their $b$-functions are additive on short exact sequences, cf. Lemma \ref{lemma - rel hol basics}. In other words
    \begin{equation} \label{eqn - funny b-function, claim 3, pt 2}
    Z(B(N_{\ell - 1})) = Z(B(N_{\ell})) \cup Z(B(N_{\ell-1} / N_{\ell})).
    \end{equation}
    We know $Z(B(N_{\ell}))$ by our inductive assumption; we know $Z(B(N_{\ell - 1}/N_{\ell}))$ by \emph{Claim 2}. Moreover, these zero sets are disjoint. Plugging this information into \eqref{eqn - funny b-function, claim 3, pt 2} shows $Z(B(N_{\ell - 1}))$ equals the zeroes of the $\mathbb{C}[s]$-polynomial of \eqref{eqn - funny b-function, claim 3 statement}. By our previous analysis, this completes the inductive step in the proof of \emph{Claim 3} and consequently validates the entirety of the proposition.
\end{proof}

\subsection{Computing Ext}

We are now ready to complete our spectral sequence study of $\Ext_{\mathscr{D}_{X}[s]}^{3} \left( \mathscr{D}_{X}[s] f^{s-1}, \mathscr{D}_{X}[s] \right)^{\ell}$. The following theorem realizes it as the middle term of a short exact sequence, where the outer terms are (reasonably) well-understood. Note the imposition of reducedness, cf. Proposition \ref{prop - dual of Ur complex, terminal cohomology}, Remark \ref{rmk - why reduced}.

\begin{theorem} \label{thm - spectral sequence dual}
    Let $f \in R = \mathbb{C}[x_{1}, x_{2}, x_{3}]$ be reduced and locally quasi-homogeneous. Then we have a short exact sequence of left $\mathscr{D}_{X}[s]$-modules
    \begin{equation} \label{eqn - thm, spectral sequence dual, statement, ses}
        0 \to \mathscr{D}_{X}[s]f^{-s} \to \Ext_{\mathscr{D}_{X}[s]}^{3} \left( \mathscr{D}_{X}[s] f^{s-1}, \mathscr{D}_{X}[s] \right)^{\ell} \to L_{f} \to 0,
    \end{equation}
    where $L_{f}$ is the left $\mathscr{D}_{X}[s]$-module supported at $0$, or is itself $0$, cf. Definition \ref{def - the cokernel, ext dual object}. 
    
    Moreover, $L_{f}$ has a $b$-function, namely
    \begin{equation} \label{eqn - thm, spectral sequence dual, statement, funny b-function}
        B(L_{f}) = \prod_{t \in \wdeg(H_{\mathfrak{m}}^{0}(R / (\partial f)))} \left( s - \frac{ - t + 2 \wdeg(f) - \sum w_{i}}{\wdeg(f)} \right) 
    \end{equation}
    (If $\wdeg(H_{\mathfrak{m}}^{3}(R / (\partial f))) = \emptyset$ we take $B(L_{f}) = 1$ as this equates to $L_{f} = 0.)$ 

\end{theorem}

\begin{proof}
    First of all, since side-changing from left to right $\mathscr{D}_{X}[s]$ is an equivalence of categories to prove \eqref{eqn - thm, spectral sequence dual, statement, ses} it suffices to prove the following is a s.e.s.:
    \begin{equation} \label{eqn - thm, spectral sequence dual, ses restatement}
        0 \to \mathscr{D}_{X}[s]f^{-s} \to \Ext_{\mathscr{D}_{X}[s]}^{3} \left( \bigg( \mathscr{D}_{X}[s] f^{s-1} \bigg)^{r}, \mathscr{D}_{X}[s] \right) \to L_{f} \to 0.
    \end{equation}
    By Proposition \ref{prop - total complex free resolution}, the double complex $\Tot(K_{f}^{\bullet, \bullet})$ is a free (right) resolution of $(\mathscr{D}_{X}[s] f^{s-1})^{r}$. So we can use $\Hom_{\mathscr{D}_{X}[s]}(\Tot(K_{f}^{\bullet, \bullet}), \mathscr{D}_{X}[s])$ to compute $\Ext_{\mathscr{D}_{X}[s]}^{3}((\mathscr{D}_{X}[s] f^{s-1})^{r}, \mathscr{D}_{X}[s])$. 
    
    Let $E_{v}^{\bullet, \bullet}$ be the $v^{\text{th}}$-page of the spectral sequence of $\Hom_{\mathscr{D}_{X}[s]}(\Tot(K_{f}^{\bullet, \bullet}), \mathscr{D}_{X}[s])$ that first computes cohomology vertically and then horizontally. Recall that Lemma \ref{lem - spectral sequence dual, page 2} describes the second page, $E_{2}^{\bullet, \bullet}$, of this spectral sequence. In particular, it shows $E_{2}^{1,1}$ (resp. $E_{2}^{2,1}$) are the cokernel (resp. kernel) of the induced map
    \begin{equation} \label{eqn - thm, spectral sequence dual, part 1}
        \mathscr{D}_{X}[s] \otimes_{\mathscr{O}_{X}} \Ext_{\mathscr{O}_{X}}^{1}(\Omega_{X}^{2}(\log f), \mathscr{O}_{X}) \xrightarrow[]{\Hom_{\mathscr{D}_{X}[s]}(\sigma_{1}^{1}, \mathscr{D}_{X}[s])} \mathscr{D}_{X}[s] \otimes_{\mathscr{O}_{X}} \Ext_{\mathscr{O}_{X}}^{1}(\Omega_{X}^{1}(\log f), \mathscr{O}_{X}).
    \end{equation}
    
    Recall the construction of $\sigma_{1}^{1}$: we lifted $\nabla_{s}: U_{f}^{1} \to U_{f}^{2}$ along the augmented free resolutions $Q_{\bullet}^{1} \otimes_{\mathscr{O}_{X}} \mathscr{D}_{X}[s] \to U_{f}^{1}$ and $Q_{\bullet}^{2} \otimes_{\mathscr{O}_{X}} \mathscr{D}_{X}[s] \to U_{f}^{2}$ to obtain lifted maps $\sigma_{j}^{1}: Q_{j}^{1} \otimes_{\mathscr{O}_{X}} \mathscr{D}_{X}[s] \to Q_{j}^{2} \otimes_{\mathscr{O}_{X}}$ so that the squares in the double complex $K_{f}^{\bullet, \bullet}$ commuted, cf. \eqref{eqn - def, the double complex}. By \cite[Lemma 00LT]{stacks-project}, the map 
    \begin{align} \label{eqn - thm, spectral sequence dual, part 2}
    H^{1} \left( \Hom_{\mathscr{D}_{X}[s]}(\mathscr{D}_{X}[s] \otimes_{\mathscr{O}_{X}} Q_{\bullet}^{2}, \mathscr{D}_{X}[s]) \right) \to \left( H^{1} \Hom_{\mathscr{D}_{X}[s]}(\mathscr{D}_{X}[s] \otimes_{\mathscr{O}_{X}} Q_{\bullet}^{1}, \mathscr{D}_{X}[s]) \right)
    \end{align}
    induced by the lifts of $\nabla_{s}$ is independent of the choice of lifts. In particular, we may use the lift $\theta_{1}: Q_{1}^{1} \otimes_{\mathscr{O}_{X}} \mathscr{D}_{X}[s] \to Q_{1}^{2} \otimes_{\mathscr{O}_{X}} \mathscr{D}_{X}[s]$ that was constructed in Proposition \ref{prop - computing analytic lifts}. Using this lift, \eqref{eqn - thm, spectral sequence dual, part 2} is exactly the map we studied in Proposition \ref{prop - computing analytic lifts, dualizing}, namely
    \begin{equation} \label{eqn - thm, spectral sequence dual, part 3}
        \mathscr{D}_{X}[s] \otimes_{\mathscr{O}_{X}} \Ext_{\mathscr{O}_{X}}^{1}(\Omega_{X}^{2}(\log f), \mathscr{O}_{X}) \xrightarrow[]{\Hom_{\mathscr{D}_{X}[s]}(\theta_{1}, \mathscr{D}_{X}[s])} \mathscr{D}_{X}[s] \otimes_{\mathscr{O}_{X}} \Ext_{\mathscr{O}_{X}}^{1}(\Omega_{X}^{1}(\log f), \mathscr{O}_{X}).
    \end{equation}
    Moreover, $E_{2}^{1,1}$ (resp. $E_{2}^{2,1})$ are the cokernel (resp. kernel) of \eqref{eqn - thm, spectral sequence dual, part 3}. 

    \vspace{5mm}

    \emph{Claim: The map \eqref{eqn - thm, spectral sequence dual, part 3} is injective.}
    
    \noindent \emph{Proof of Claim:} At any stalk outside the origin, both the domain and codomain of \eqref{eqn - thm, spectral sequence dual, part 3} vanish. So it suffices to check injectivity at the $0$ stalk. Because of the explicit description of the map \eqref{eqn - thm, spectral sequence dual, part 3} at the $0$ stalk given in Proposition \ref{prop - computing analytic lifts, induced map on Ext}, we have, for each $t \in \mathbb{Z}_{\geq -1}$, the induced map
    \begin{equation*}
    G_t = \left[ 
        \begin{aligned}
        & \mathscr{D}_{X,0} \otimes_{\mathbb{C}} \mathbb{C}[s]_{\leq t} \otimes_{\mathscr{O}_{X,0}} \Ext_{\mathscr{O}_{X,0}}^1 (\Omega_X^2(\log f), \mathscr{O}_{X,0}) \\
        & \longrightarrow \mathscr{D}_{X,0} \otimes_{\mathbb{C}} \mathbb{C}[s]_{\leq t+1} \otimes_{\mathscr{O}_{X,0}} \Ext_{\mathscr{O}_{X,0}}^1 (\Omega_X^2(\log f), \mathscr{O}_{X,0})
        \end{aligned}
    \right].
    \end{equation*}
    Regarding the map \eqref{eqn - thm, spectral sequence dual, part 3} at the $0$ stalk as a complex, $0 \subseteq G_{-1} \subseteq G_0 \subseteq G_1 \subseteq \cdots $ is a filtration by subcomplexes. The induced map on the associated graded complex is given by (left) multiplication with $\wdeg(f) s$:
    \begin{equation} \label{eqn - thm, spectral sequence dual, part 4}
        \mathscr{D}_{X,0}[s] \otimes_{\mathscr{O}_{X,0}} \Ext_{\mathscr{O}_{X,0}}^{1}(\Omega_{X}^{2}(\log f), \mathscr{O}_{X,0}) \xrightarrow[]{\wdeg(f) s \cdot } \mathscr{D}_{X,0}[s] \otimes_{\mathscr{O}_{X,0}} \Ext_{\mathscr{O}_{X}}^{1}(\Omega_{X,0}^{1}(\log f), \mathscr{O}_{X,0}).
    \end{equation}
    Since neither the domain nor codomain of \eqref{eqn - thm, spectral sequence dual, part 4} has $s$-torsion, \eqref{eqn - thm, spectral sequence dual, part 4} is injective and hence \eqref{eqn - thm, spectral sequence dual, part 3} is injective at the $0$ stalk. So the \emph{Claim} is true.

    \vspace{5mm}

    Thanks to the \emph{Claim} and the preceding discussion, $E_{2}^{2,1} = \ker \text{ of \eqref{eqn - thm, spectral sequence dual, part 3}} = 0$. Moreover, $E_{2}^{1,1} = \text{cokernel of \eqref{eqn - thm, spectral sequence dual, part 3}} = L_{f}$, where $L_{f}$ is as defined in Definition \ref{def - the cokernel, ext dual object}. In particular the second page of the spectral sequence is 
    \begin{equation} \label{eqn - thm, spectral sequence dual, part 5}
        \begin{tikzcd}
            0
                &E_{2}^{1,1} = L_{f}
                    & E_{2}^{2,1} = 0 \arrow[lld]
                        & 0  \\
            E_{2}^{0,0} = \frac{\mathscr{D}_{X}[s]}{\ann_{\mathscr{D}_{X}[s]}^{(1) f^{-s}}}
                & ?
                    & ? 
                        & ?
        \end{tikzcd}
    \end{equation}
    where: the description of $E_{2}^{0,0}$ is from Lemma \ref{lem - spectral sequence dual, page 2}; all entries of $E_{2}^{\bullet, \bullet}$ outside of \eqref{eqn - thm, spectral sequence dual, part 5} vanish. It follows that $E_{\infty}^{0,0} = E_{2}^{0,0}$ and $E_{\infty}^{1,1} = E_{2}^{1,1}$ and no other entries of this diagonal appear on the infinity page. This means we have a short exact sequence
    \begin{equation} \label{eqn - thm, spectral sequence dual, part 6}
        0 \to E_{2}^{0,0} \to \Ext_{\mathscr{D}_{X}[s]}^{3} \left( \bigg( \mathscr{D}_{X}[s] f^{s-1} \bigg)^{r}, \mathscr{D}_{X}[s] \right) \to E_{2}^{1,1} \to 0.
    \end{equation}
    (Indeed the spectral sequence was constructed by filtering by columns and so $E_{\infty}^{0,0} = E_{2}^{0,0} \xhookrightarrow{} H^{0}(\Hom_{\mathscr{D}_{X}[s]}(\Tot(K_{f}^{\bullet, \bullet})))$.) The promised s.e.s. \eqref{eqn - thm, spectral sequence dual, ses restatement}, and hence \eqref{eqn - thm, spectral sequence dual, statement, ses}, follows from \eqref{eqn - thm, spectral sequence dual, part 6} once we make the identification $E_{2}^{0,0} \simeq \mathscr{D}_{X}[s] f^{-s}$; this follows from the description of $E_{2}^{0,0}$ in \eqref{eqn - thm, spectral sequence dual, part 5} by arguing exactly as in the first paragraph of Proposition \ref{prop - U complex resolves}.

    All that remains is to verify \eqref{eqn - thm, spectral sequence dual, statement, funny b-function}. But we already computed the $b$-function of $L_{f}$ in Proposition \ref{prop - funny b-function}. While that computation used $\wdeg(\Ext_{R}^{1}(\Omega_{R}^{1}(\log f), R))$ data, in the weighted homogeneous case we can translate into $\wdeg(H_{\mathfrak{m}}^{0}(R / (\partial f)))$ data using the dictionary Proposition \ref{prop - weighted data dictionary, ext log 1 forms, local cohomology milnor algebra}.
\end{proof}

\vspace{5mm}

\section{Applications to Bernstein--Sato Polynomials}

Here we use Theorem \ref{thm - spectral sequence dual} to extract applications to the Bernstein--Sato polynomial $b_{f}(s)$ of our class of $\mathbb{C}^{3}$ divisors. Recall Lemma \ref{lemma - b-function and b-function of dual}: the zeroes of the local Bernstein--Sato polynomial of $f$ at $0$ agree with the zeroes of the $b$-function of $\Ext_{\mathscr{D}_{X,0}}^{4}(\mathscr{D}_{X,0}[s] f^{s} / \mathscr{D}_{X,0}[s] f^{s+1}, \mathscr{D}_{X,0})$. So to use the $\Ext$ computation of Theorem \ref{thm - spectral sequence dual}, we must study the relationship between the $\mathscr{D}_{X}[s]$-modules $\mathscr{D}_{X}[s] f^{s}$ and $\mathscr{D}_{X}[s] f^{s+1}$. 

In the first subsection we note that said Theorem applies to different ``parametric powers'' of $f$. In the subsequent subsections we extract applications:
\begin{enumerate}[label=(\roman*)]
    \item Theorem \ref{thm - new roots of BS-poly} show the degree data of $H_{\mathfrak{m}}^{3} (R / (\partial f))$ gives zeroes of $b_{f}(s)$;
    \item Theorem \ref{thm - localized symmetry of BS poly} gives a partial symmetry of $Z(b_{f}(s))$ about $-1$;
    \item Corollary \ref{cor - small roots BS poly} computes the zeroes $Z(b_{f}(s)) \cap (-3,-2]$ in terms of degree data of $H_{\mathfrak{m}}^{0} (R / (\partial f))$;
    \item Corollary \ref{cor - twisted LCT} characterizes, in terms of Milnor algebra graded data, exactly the weights $\lambda \in (-\infty, 0]$ for which a twisted Logarithmic Comparison Theorem holds;
    \item Theorem \ref{thm - BS poly arrangement} gives a \emph{complete} formula for the zeroes of $b_{f}(s)$ when $f$ defines a hyperplane arrangement. 
\end{enumerate}

In this section, we only work analytically as we require the duality computation Theorem \ref{thm - spectral sequence dual} and we do not know if this holds algebraically, cf. Remark \ref{rmk - why analtyic, U complex}, Remark \ref{rmk - why analytic, total complex}.

\subsection{Formal Substitutions}

We first discuss how to extend Theorem \ref{thm - spectral sequence dual} to study $\mathscr{D}_{X}[s] f^{s + k}$, etc. This amounts to making a formal variable substitution of $s$ with some other linear polynomial in $s$.

\begin{define} \label{def - twisted Ur complex}
    Let $f \in R = \mathbb{C}[x_{1}, \dots, x_{n}]$. Let $\alpha \in \mathbb{C}^{\star}$ and $\beta \in \mathbb{C}$. Define the complex $(U_{f}^{\bullet}, \nabla_{\alpha s + \beta})$ of right $\mathscr{D}_{X}[s]$-modules
    \begin{align} \label{equn - def of twisted Ur complex}
    (U_{f}^{\bullet}, \nabla_{\alpha s + \beta}) 
    &= 0 \xrightarrow[]{\nabla_{\alpha s + \beta}} \Omega_{X}^{0}(\log f) \otimes_{\mathscr{O}_{X}} \mathscr{D}_{X}[s] \xrightarrow[]{\nabla_{\alpha s + \beta}} \Omega_{X}^{1}(\log f) \otimes_{\mathscr{O}_{X}} \mathscr{D}_{X}[s] \xrightarrow[]{\nabla_{\alpha s + \beta}} \\
    &\cdots \xrightarrow[]{\nabla_{\alpha s + \beta}} \Omega_{X}^{n}(\log f) \otimes_{\mathscr{O}_{X}} \mathscr{D}_{X}[s]  \xrightarrow[]{\nabla_{\alpha s + \beta}} 0 \nonumber.
    \end{align}
    As in Definition \ref{def - the Ur complex}, the right $\mathscr{D}_{X}[s]$-module structure of
    \begin{equation} \label{eqn - def of twisted Ur complex object}
        U_{f}^{p} = \Omega_{X}^{p}(\log f) \otimes_{\mathscr{O}_{X}} \mathscr{D}_{X}[s]
    \end{equation}
    is given trivially by multiplication on the right. The differential is as in Definition \ref{def - the Ur complex}, except we formally change ``$s$'' to ``$\alpha s + \beta$''. So the differential is locally given by 
    \begin{equation} \label{eqn - def of twisted Ur complex differential}
        \nabla_{\alpha s + \beta}(\eta \otimes P(s)) = \big( d(\eta) \otimes P(s) \big) + \big( \frac{df}{f} \wedge \eta \otimes (\alpha s + \beta) P(s) \big) + \big( \sum_{1 \leq i \leq n} dx_{i} \wedge \eta \otimes \partial_{i} P(s) \big).
    \end{equation}
    We also have a double complex of free right $\mathscr{D}_{X}[s]$-modules $K_{f, \alpha s + \beta}^{\bullet, \bullet}$ defined in a completely analogous way as Definition \ref{def - the double complex}. It will have the same as objects as $K_{f}^{\bullet, \bullet}$ and differentials obtained by exchanging $s$ with $\alpha s + \beta$. We denote its total complex as $\Tot (K_{f, \alpha s + \beta}^{\bullet, \bullet})$.
\end{define}

In other words, $U_{f, \alpha s + \beta}^{\bullet}$ and $\Tot (K_{f, \alpha s + \beta}^{\bullet, \bullet})$ are obtained from $U_{f}$ and $\Tot (K_{f}^{\bullet, \bullet})$ by making the formal substitution $s \mapsto \alpha s + \beta$ in all the maps. (Note that $U_{f}$ is $U_{f, s}$ in this new notation.) This formal substitution changes nothing about the arguments or constructions in the preceding sections: it is truly formal, cf. Remark \ref{rmk - formal substitution is ok}. This allows us to rewrite Theorem \ref{thm - spectral sequence dual} to study duals of $\mathscr{D}_{X}[s] f^{\alpha s + \beta}$. So define

\begin{define} \label{def - twisted cokernel object}
    Let $f \in R = \mathbb{C}[x_{1}, x_{2}, x_{3}]$ be quasi-homogeneous. Select $\alpha \in \mathbb{C}^{\star}$ and $\beta \in \mathbb{C}$. Denote by $L_{f}$ the left $\mathscr{D}_{X}[s]$-module that is the cokernel of 
    \begin{align*}
        L_{f, \alpha s + \beta} = \text{cokernel} \bigg[ \Hom_{\mathscr{D}_{X}[s]}(\theta_{1, \alpha s + \beta}, \mathscr{D}_{X}[s]): 
        &\mathscr{D}_{X}[s] \otimes_{\mathscr{O}_{X}} \Ext_{\mathscr{O}_{X}}^{1}(\Omega_{X}^{2}(\log f), \mathscr{O}_{X}) \to \\
        &\mathscr{D}_{X}[s] \otimes_{\mathscr{O}_{X}} \Ext_{\mathscr{O}_{X}}^{1}(\Omega_{X}^{1}(\log f), \mathscr{O}_{X}) \bigg] \nonumber
    \end{align*}
    Here $\theta_{1, \alpha s + \beta}$ is as constructed in Proposition \ref{prop - computing analytic lifts} after the formal substitution $s \mapsto \alpha s + \beta$. Using the notation of Proposition \ref{prop - computing analytic lifts}, 
    \begin{equation*}
        \theta_{1, \alpha s + \beta}(\mathbf{e}_{t} \otimes 1) = \mathbf{e}_{t} \otimes (\wdeg(f)(\alpha s + \beta) - b_{t} + E).
    \end{equation*}
\end{define}

Then we obtain the following variant of Theorem \ref{thm - spectral sequence dual}:

\begin{corollary} \label{cor - spectral sequence dual, twisted}
    Let $f \in R = \mathbb{C}[x_{1}, x_{2}, x_{3}]$ be reduced and locally quasi-homogeneous. Select $\alpha \in \mathbb{C}^{\star}$ and $\beta \in \mathbb{C}$. Then we have a short exact sequence of left $\mathscr{D}_{X}[s]$-modules
    \begin{equation*}
        0 \to \mathscr{D}_{X}[s]f^{-(\alpha s + \beta)} \to \Ext_{\mathscr{D}_{X}[s]}^{3} \left( \mathscr{D}_{X}[s] f^{\alpha s + \beta - 1}, \mathscr{D}_{X}[s] \right)^{\ell} \to L_{f, \alpha s + \beta} \to 0,
    \end{equation*}
    where $L_{f, \alpha s + \beta}$ is the left $\mathscr{D}_{X}[s]$-module supported at $0$, or is itself $0$, described in Definition \ref{def - twisted cokernel object}/Definition \ref{def - the cokernel, ext dual object}.
    
    Moreover, $L_{f, \alpha s + \beta}$ has a $b$-function, namely
    \begin{equation*} 
        B(L_{f, \alpha s + \beta}) = \prod_{t \in \wdeg(H_{\mathfrak{m}}^{0}(R / (\partial f)))} \left( \alpha s + \beta - \frac{ - t + 2 \wdeg(f) - \sum w_{i}}{\wdeg(f)} \right) 
    \end{equation*}
    (If $\wdeg(H_{\mathfrak{m}}^{0}(R/ (\partial f))) = \emptyset$ we take $B(L_{f, \alpha s + \beta}) = 1$ as this equates to $L_{f, \alpha s + \beta} = 0.)$ 

\end{corollary}

\begin{proof}
    This follows exactly as in Theorem \ref{thm - spectral sequence dual} after the formal variable substitution $s \mapsto \alpha s + \beta$. The only part that is not purely formal is showing cohomological vanishing of the associated graded complex considered used in the proof of Proposition \ref{prop - U complex resolves}. It will now have differentials twisted by a $\beta df / f \wedge$ term, but still yields to \cite{NoncommutativePeskineSzpiro} as before.
\end{proof}

\subsection{Milnor Algebra Data Gives Roots of the Bernstein--Sato Polynomial}

We now use Corollary \ref{cor - spectral sequence dual, twisted}, the modified version of Theorem \ref{thm - spectral sequence dual}, to extract data about Bernstein--Sato polynomials of $f$ by studying duals of $\mathscr{D}_{X}[s] f^{s} / \mathscr{D}_{X}[s] f^{s+1}$. What follows allows us to extend the s.e.s. 
\begin{equation} \label{eqn - basic s.e.s.}
    0 \to \mathscr{D}_{X}[s] f^{s+1} \to \mathscr{D}_{X}[s] f^{s} \to \mathscr{D}_{X}[s] f^{s} / \mathscr{D}_{X}[s] f^{s+1} \to 0 
\end{equation}
along our previously studied resolutions of the non-cokernel terms.

\begin{proposition} \label{prop - chain map, twisted U and double complexes}
    There is a chain map of right $\mathscr{D}_{X}[s]$-modules/complexes
    \begin{equation} \label{eqn - chain map, twisted U}
        U_{f, \nabla_{s+2}} \xrightarrow[]{f \cdot } U_{f, \nabla_{s+1}}
    \end{equation}
    given by left multiplication with $f$ and this extends to a chain map of right $\mathscr{D}_{X}[s]$-modules/complexes
    \begin{equation} \label{eqn - chain map, twisted double complex}
        \Tot (K_{f, s+2}^{\bullet, \bullet}) \xrightarrow[]{f \cdot } \Tot (K_{f, s+1}^{\bullet, \bullet})
    \end{equation}
    again given by left multiplication with $f$.
\end{proposition}

\begin{proof}

Let $\eta \in \Omega_{X}^{p}(\log f)$. To check \eqref{eqn - chain map, twisted U} it suffices to that $f \cdot \nabla_{s+2} (\eta \otimes P(s)) = \nabla_{s+1} (f \eta \otimes P(s))$, which is a straightforward computation. 

To check \eqref{eqn - chain map, twisted double complex} it suffices to check that left multiplication with $f$ commutes with every horizontal and vertical map depicted in \eqref{eqn - def, the double complex}, adjusted appropriately for the substitution of variables. For the vertical maps this is obvious: by the construction, each vertical map is of the form $- \otimes \id$ where $-$ is a map of $\mathscr{O}_{X}$-modules. As for the horizontal maps, let us first consider the picture for $\sigma_{0}^{1}$ in \eqref{eqn - def, the double complex}, again adjusted appropriately for our variable substitutions. We have the following cube of right $\mathscr{D}_{X}[s]$-maps:
\[
\begin{tikzcd}
    & Q_{0}^{1} \otimes_{\mathscr{O}_{X}} \mathscr{D}_{X}[s]  \ar{rr} \ar{dd} \ar{dl}{f \cdot} & &   Q_{0}^{2} \otimes_{\mathscr{O}_{X}} \mathscr{D}_{X}[s]  \ar{dd} \ar{dl}{f \cdot} \\
    Q_{0}^{1} \otimes_{\mathscr{O}_{X}} \mathscr{D}_{X}[s] \ar[crossing over]{rr} \ar{dd} & & Q_{0}^{2} \otimes_{\mathscr{O}_{X}} \mathscr{D}_{X}[s] \\
      & U_{f}^{1}  \ar{rr}{\nabla_{s+2} \quad \quad \quad \quad \quad} \ar{dl}{f \cdot} & &  U_{f}^{2}  \ar{dl}{f \cdot} \\
    U_{f}^{1} \ar{rr}{\nabla_{s+1}} && U_{f}^{2} \ar[from=uu,crossing over].
\end{tikzcd}
\]
Here the back face corresponds to $\Tot(K_{f, s+2}^{\bullet, \bullet})$ and in particular the corresponding lower central square in the construction of \eqref{eqn - def, the double complex}. The front face similarly corresponds to $\Tot(K_{f, s+1}^{\bullet, \bullet})$. 

We know the front and back face commute by the construction of the corresponding double complexes; we know the bottom square commutes by \eqref{eqn - chain map, twisted U}; we know the left and right squares commute as these maps are given by $ - \otimes \id$ where $-$ is a map of $\mathscr{O}_{X}$-modules. Therefore the top square must also commute. That is, multiplication with $f$ on the left commutes with the corresponding $\sigma_{0}^{1}$'s, adjusted for the variable substitution. 

We can clearly repeat this commutative cube argument throughout the construction \eqref{eqn - def, the double complex}, demonstrating \eqref{eqn - chain map, twisted double complex}.
\end{proof}

\begin{proposition} \label{prop - twisted U complex, image containment}
    Let $f \in R = \mathbb{C}[x_{1}, x_{2}, x_{3}]$ be reduced and locally quasi-homogeneous. The chain map $\Tot(K_{f, s+2}^{\bullet, \bullet}) \xrightarrow[]{f \cdot} \Tot(K_{f, s+1}^{\bullet, \bullet})$ induces the map dual complexes
    \[
    \Hom_{\mathscr{D}_{X}[s]}(\Tot(K_{f, s+1}^{\bullet, \bullet}),  \mathscr{D}_{X}[s])) \xrightarrow[]{\cdot f} \Hom_{\mathscr{D}_{X}[s]}(\Tot(K_{f, s+2}^{\bullet, \bullet}),  \mathscr{D}_{X}[s])).
    \]
    In particular this induces a map of $\Ext$ modules satisfying
    \begin{align} \label{eqn - twisted u complex, image containment}
        \im \bigg[ \cdot f :  \Ext_{\mathscr{D}_{X}[s]}^{3}( \mathscr{D}_{X}[s] f^{s}, \mathscr{D}_{X}[s])^{\ell} &\to \Ext_{\mathscr{D}_{X}[s]}^{3}( \mathscr{D}_{X}[s] f^{s+1}, \mathscr{D}_{X}[s])^{\ell} \bigg] \\
        & \subseteq \mathscr{D}_{X}[s] f^{-s-2} \nonumber \\
        & \subseteq  \Ext_{\mathscr{D}_{X}[s]}^{3}( \mathscr{D}_{X}[s] f^{s+1}, \mathscr{D}_{X}[s])^{\ell}. \nonumber
    \end{align}
\end{proposition}

\begin{proof}
    The statement about induced maps on dual complexes is straightforward. As for \eqref{eqn - twisted u complex, image containment}, let $[\xi] \in H^{0} (\Hom_{\mathscr{D}_{X}[s]}(\Tot(K_{f, s+1}^{\bullet, \bullet}), \mathscr{D}_{X}[s])$. Then the image of $[\xi]$ in $H^{0} (\Hom_{\mathscr{D}_{X}[s]}(\Tot(K_{f, s+1}^{\bullet, \bullet}), \mathscr{D}_{X}[s])$ is $[\xi f].$ Since $(\Hom_{\mathscr{D}_{X}[s]}(\Tot(K_{f, s+1}^{\bullet, \bullet}), \mathscr{D}_{X}[s])$ and $(\Hom_{\mathscr{D}_{X}[s]}(\Tot(K_{f, s+2}^{\bullet, \bullet}), \mathscr{D}_{X}[s])$ have the same objects and since they are triangular, we may write the image of $[\xi]$ as $[\xi] f \in H^{0} (\Hom_{\mathscr{D}_{X}[s]}(\Tot(K_{f, s+1}^{\bullet, \bullet}), \mathscr{D}_{X}[s])$. By Corollary \ref{cor - spectral sequence dual, twisted}/Theorem \ref{thm - spectral sequence dual}, it suffices to show that 
    \begin{equation} \label{eqn - twisted U containment, 1}
    0 = \overline{[\xi]} f \in \frac{\Ext_{\mathscr{D}_{X}[s]}^{3}(\mathscr{D}_{X}[s] f^{s+1}, \mathscr{D}_{X}[s])^{\ell}}{\mathscr{D}_{X}[s] f^{-s-2}}.
    \end{equation}
    So to conclude \eqref{eqn - twisted u complex, image containment}, it suffices to show the left $\mathscr{D}_{X}[s]$-module in \eqref{eqn - twisted U containment, 1} is killed by right multiplication with $f$.
    
    By the explicit description of $L_{f, s+2}$ in Definition \ref{def - twisted cokernel object}/Definition \ref{def - the cokernel, ext dual object}, we see the left $\mathscr{D}_{X}[s]$-module of \eqref{eqn - twisted U containment, 1} is a quotient of 
    \begin{align} \label{eqn - twisted U containment, 2}
    \mathscr{D}_{X}[s] \otimes_{\mathscr{O}_{X}} \Ext_{\mathscr{O}_{X}}^{1}(\Omega_{X}^{1}(\log f), \mathscr{O}_{X}) 
    &\simeq \mathscr{D}_{X}[s] \otimes_{\mathscr{O}_{X}} \Ext_{\mathscr{O}_{X}}^{1}(\Omega_{X}^{2}(\log f), \mathscr{O}_{X}) \\
    &\simeq \mathscr{D}_{X}[s] \otimes_{\mathscr{O}_{X}} \Ext_{\mathscr{O}_{X}}^{3}( \mathscr{O}_{X} / (\partial f), \mathscr{O}_{X}). \nonumber
    \end{align}
    Here the first ``$\simeq$'' of \eqref{eqn - twisted U containment, 2} is Proposition \ref{prop - finite hom gen set of Ext modules}; the second follows by an entirely similar analytic variant of the argument in Proposition \ref{prop - weighted data dictionary, ext log 1 forms, local cohomology milnor algebra}. (For consistency, we use $(\partial f)$ to denote the Jacobian ideal of $f$ which is coordinate independent.) As $f$ is locally quasi-homogeneous, $\mathscr{O}_{X} / (\partial f)$ is killed by right multiplication with $f$. This annihilation passes to $\Ext$ modules and also to $\mathscr{D}_{X}[s] \otimes_{\mathscr{O}_{X}} \Ext_{\mathscr{O}_{X}}^{3}(\mathscr{O}_{X} / (\partial f), \mathscr{O}_{X})$. Any subsequent quotient will also be killed by right multiplication with $f$, and we are done.
\end{proof}

We have shown that \eqref{eqn - basic s.e.s.} is well-behaved with respect to our free resolutions of the non-cokernel terms and started studying dual behavior. We are now in position to approximate the dual of $\mathscr{D}_{X}[s] f^{s} / \mathscr{D}_{X}[s] f^{s+1}$. Our approximation involves $L_{f, s+2}$ and hence, through its $b$-function (see Proposition \ref{prop - funny b-function}), degree data of $H_{\mathfrak{m}}^{0}(R / (\partial f))$. As promised, we now show this degree data induces roots of $b_{f}(s)$:

\begin{theorem} \label{thm - new roots of BS-poly}
    Let $f \in R = \mathbb{C}[x_{1}, x_{2}, x_{3}]$ be reduced and locally quasi-homogeneous. Then we have a surjection of left $\mathscr{D}_{X}[s]$-modules
    \begin{equation} \label{eqn - new roots of BS-poly, statement 1}
        \Ext_{\mathscr{D}_{X}[s]}^{4} \left( \frac{\mathscr{D}_{X}[s] f^{s}}{\mathscr{D}_{X}[s] f^{s+1}}, \mathscr{D}_{X}[s] \right)^{\ell} \twoheadrightarrow L_{f, s+2}
    \end{equation}
    where $L_{f, s+2}$ is explicitly defined in Definition \ref{def - twisted cokernel object}/Definition \ref{def - the cokernel, ext dual object}. In particular, the zeroes of the $b$-function of $L_{f, s+2}$ are contained in the zeroes of the Bernstein--Sato polynomial of $f$:
    \begin{equation} \label{eqn - new roots of BS-poly, statement 2}
        Z(b_{f}(s)) \supseteq Z(B(L_{f, s+2})) = \left\{ \frac{ - (t + \sum w_{i}) }{\wdeg(f)} \mid t \in \wdeg(H_{\mathfrak{m}}^{0} (R / (\partial f))) \right\}.
    \end{equation}

\end{theorem}

\begin{proof}
    Consider the s.e.s of left $\mathscr{D}_{X}[s]$-modules
    \begin{equation} \label{eqn - canonical s.e.s. fs}
    0 \to \frac{\mathscr{D}_{X}[s]}{\ann_{\mathscr{D}_{X}[s]} f^{s+1}} \xrightarrow[]{\cdot f} \frac{\mathscr{D}_{X}[s]}{\ann_{\mathscr{D}_{X}[s]} f^{s}} \to \frac{\mathscr{D}_{X}[s] f^{s}}{\mathscr{D}_{X}[s] f^{s+1}} \to 0
    \end{equation}
    where the first map is given by right multiplication with $f$. Applying the side-changing functor $(-)^{r}$ gives a s.e.s. of right $\mathscr{D}_{X}[s]$-modules which induces a long exact sequence of $\Ext$ modules. This abbreviates to the s.e.s. of left $\mathscr{D}_{X}[s]$-modules
    \begin{align} \label{eqn - new roots of BS-poly, 2}
    0 \to \Ext_{\mathscr{D}_{X}[s]}^{3}(\mathscr{D}_{X}[s] f^{s}, \mathscr{D}_{X}[s])^{\ell}
    &\xrightarrow[]{\cdot f} \Ext_{\mathscr{D}_{X}[s]}^{3}(\mathscr{D}_{X}[s] f^{s+1}, \mathscr{D}_{X}[s])^{\ell} \\
    &\to \Ext_{\mathscr{D}_{X}[s]}^{4}(\frac{\mathscr{D}_{X}[s] f^{s}}{\mathscr{D}_{X}[s] f^{s+1}}, \mathscr{D}_{X}[s])^{\ell} \to 0 \nonumber
    \end{align}
    where we have used the fact $\mathscr{D}_{X}[s] f^{s} / \mathscr{D}_{X}[s] f^{s+1}$ is $4$-Cohen--Macaulay (cf. subsection 3.5, \cite[Lemma 3.3.3, Theorem 3.5.1]{ZeroLociBSIdeals}). 
    
    Now we use our assumptions on $f$. The following maps verify \eqref{eqn - new roots of BS-poly, statement 1}:
    \begin{align} \label{eqn - Ext surj Lf for referee}
        \Ext&_{\mathscr{D}_{X}[s]}^{4}(\frac{\mathscr{D}_{X}[s] f^{s}}{\mathscr{D}_{X}[s] f^{s+1}}, \mathscr{D}_{X}[s])^{\ell} \\
        &\simeq \frac{\Ext_{\mathscr{D}_{X}[s]}^{3}(\mathscr{D}_{X}[s] f^{s+1}, \mathscr{D}_{X}[s])^{\ell}}{\im \big[ \Ext_{\mathscr{D}_{X}[s]}^{3}(\mathscr{D}_{X}[s] f^{s}, \mathscr{D}_{X}[s])^{\ell} \xrightarrow[]{\cdot f} \Ext_{\mathscr{D}_{X}[s]}^{3}(\mathscr{D}_{X}[s] f^{s+1}, \mathscr{D}_{X}[s])^{\ell} \big]} \nonumber \\
        &\twoheadrightarrow \frac{\Ext_{\mathscr{D}_{X}[s]}^{3}(\mathscr{D}_{X}[s] f^{s+1}, \mathscr{D}_{X}[s])^{\ell}}{\mathscr{D}_{X}[s]f^{-s-2}} \nonumber \\
        &\simeq L_{f,s+2}. \nonumber
    \end{align}
    To confirm \eqref{eqn - Ext surj Lf for referee} note that: the first ``$\simeq$'' is \eqref{eqn - new roots of BS-poly, 2}; the ``$\twoheadrightarrow$'' is Proposition \ref{prop - twisted U complex, image containment}; the last ``$\simeq$'' is Corollary \ref{cor - spectral sequence dual, twisted}/Theorem \ref{thm - spectral sequence dual}.

    As for \eqref{eqn - new roots of BS-poly, statement 2}, by considering the local surjection \eqref{eqn - new roots of BS-poly, statement 1} at $0$ and using Lemma \ref{lemma - rel hol basics}, we see that 
    \begin{equation} \label{eqn - new roots of BS-poly, 5}
    Z(B(L_{f, s+2, 0})) \subseteq Z \bigg( B \big(\Ext_{\mathscr{D}_{X, 0}[s]}^{4}(\frac{\mathscr{D}_{X, 0}[s] f^{s}}{\mathscr{D}_{X, 0}[s] f^{s+1}}, \mathscr{D}_{X, 0}[s])^{\ell} \big) \bigg).
    \end{equation}
    By Lemma \ref{lem - side changing b-functions} and Lemma \ref{lemma - b-function and b-function of dual}, the containment \eqref{eqn - new roots of BS-poly, 5} is the same as 
    \begin{equation}
    Z(B(L_{f, s+2, 0})) \subseteq Z \bigg( B \big( \frac{\mathscr{D}_{X, 0}[s] f^{s}}{\mathscr{D}_{X, 0}[s] f^{s+1}} \big) \bigg) = Z(b_{f,0}(s)),
    \end{equation}
    where $b_{f,0}(s)$ is the local Bernstein--Sato polynomial of $f$ at $0$. Since $b_{f}(s)$ is the least common multiple of all the local Bernstein--Sato polynomials of $f$, we deduce the first ``$\supseteq$'' of \eqref{eqn - new roots of BS-poly, statement 2}. The ``$=$'' in \eqref{eqn - new roots of BS-poly, statement 2} is Corollary \ref{cor - spectral sequence dual, twisted}/Theorem \ref{thm - spectral sequence dual}. (Recall $L_{f, s+2}$ vanishes outside of $0$, so $B(L_{f, s+2}) = B(L_{f, s+2, 0})$.)
\end{proof}

\begin{remark} \label{eqn - Saito caveat}
    Restrict to the case $f$ defines the affine cone of $Z \subseteq \mathbb{P}^{2}$, where $Z$ has only isolated quasi-homogeneous singularities. Then \eqref{eqn - new roots of BS-poly, statement 2} recovers M. Saito's \cite[Corollary 3]{SaitoBSProjectiveWeightedHomogeneous}. Because his methods depend on blow-up techniques, it does not seem possible to generalize his arguments to the locally quasi-homogeneous setting of Theorem \ref{thm - new roots of BS-poly}. That our strategies are so different is curious.
\end{remark}

\subsection{Partial Symmetry of the Bernstein--Sato Polynomial}

Now we study partial symmetry properties of $Z(b_{f}(s))$. Recall Narv\'{a}ez Macarro's approach in the free case: here $\Ext_{\mathscr{D}_{X}[s]}^{n}(\mathscr{D}_{X}[s] f^{s}, \mathscr{D}_{X}[s]) \simeq \mathscr{D}_{X}[s] f^{-s-1}$; consequently, the dual of $\mathscr{D}_{X}[s] f^{s} / \mathscr{D}_{X}[s] f^{s+1}$ is isomorphic to $\mathscr{D}_{X}[s] f^{-s-2} / \mathscr{D}_{X}[s] f^{-s+1}$. Looking at Corollary \ref{cor - spectral sequence dual, twisted}/Theorem \ref{thm - spectral sequence dual}, we could repeat this if we localized in such a way to kill the $L_{f}$ terms. Morally, this is exactly what we do: the technique was developed in \cite{ZeroLociBSIdeals} and we introduced enough tools in subsection 3.5 to utilize the idea.

\begin{proposition} \label{prop - localized symmetry of fs}
    Let $f \in R = \mathbb{C}[x_{1},x_{2}, x_{3}]$ be reduced and locally quasi-homogeneous. Pick $\alpha \in \mathbb{C}^{\star}$, $\beta \in \mathbb{C}$, and a multiplicatively closed set $S \subseteq \mathbb{C}[s]$ such that $B(L_{f, \alpha s + \beta}) \in S$. Set $T = S^{-1} \mathbb{C}[s]$. Then
    \[
    \Ext_{\mathscr{D}_{T}}^{3} ( \mathscr{D}_{T} f^{\alpha s + \beta - 1}, \mathscr{D}_{T}) ^{\ell} \simeq \mathscr{D}_{T} f^{ - (\alpha s + \beta)}.
    \]
\end{proposition}

\begin{proof}
    Consider the short exact sequence of Corollary \ref{cor - spectral sequence dual, twisted}/Theorem \ref{thm - spectral sequence dual}. Now apply the localization functor $ - \otimes_{\mathbb{C}[s]} S^{-1} \mathbb{C}[s]$, use its flatness, and invoke Lemma \ref{lemma - localization and Ext}.
\end{proof}

Partial symmetry quickly follows. The obstruction to symmetry is determined by degree data of $H_{\mathfrak{m}}^{0}(R / (\partial f))$.

\begin{theorem} \label{thm - localized symmetry of BS poly}
    Let $f \in R = \mathbb{C}[x_{1}, x_{2}, x_{3}]$ be reduced and locally quasi-homogeneous. Define 
    \begin{equation} \label{eqn - thm - localized symmetry of BS poly, statement 1}
        \Xi_{f} = \bigcup_{t \in \wdeg(H_{\mathfrak{m}}^{0} (R / (\partial f)))} \bigg\{ \frac{-(t + \sum w_{i})}{\wdeg(f)} \bigg\} \cup \bigg\{ \frac{-(t + \sum w_{i}) + \wdeg(f)}{\wdeg(f)}  \bigg\}.
    \end{equation}
    Then we have partial symmetry of the Bernstein--Sato polynomial's roots about $-1$:
    \begin{equation} \label{eqn - thm - localized symmetry of BS poly, statement 2}
    \alpha \in \bigg( Z(b_{f}(s)) \setminus \Xi_{f} \bigg) \iff - \alpha - 2 \in \bigg( Z(b_{f}(s)) \setminus \Xi_{f} \bigg).
    \end{equation}
\end{theorem}

\begin{proof}
   Let $S \subseteq \mathbb{C}[s]$ be the multiplicatively closed set generated by $s - \frac{ -(t+ \sum w_{i})}{\wdeg(f)}$ and $s - \frac{-(t + \sum w_{i}) + \wdeg(f)}{\wdeg(f)}$ as $t$ ranges through $\wdeg (H_{\mathfrak{m}}^{0} (R / (\partial f)))$. Set $T = S^{-1} \mathbb{C}[s]$. By Lemma \ref{lemma - localization plays well with b-functions}, it suffices to prove that 
    \begin{equation} \label{eqn - thm - localized symmetry of BS poly, 1}
        \alpha \in Z(B ( \frac{\mathscr{D}_{T} f^{s}}{\mathscr{D}_{T} f^{s+1}} )) \iff - \alpha - 2 \in Z(B( \frac{\mathscr{D}_{T} f^{s}}{\mathscr{D}_{T} f^{s+1}})).
    \end{equation}

    Consider the canonical s.e.s of left $\mathscr{D}_X$-modules
    \begin{equation*}
        0 \to \mathscr{D}_{X}[s] f^{s+1} \xrightarrow[]{\cdot f} \mathscr{D}_{X}[s] f^s \to \frac{\mathscr{D}_X[s] f^s} {\mathscr{D}_{X}[s] f^{s+1}} \to 0.
    \end{equation*}
    Apply the flat localization $- \otimes_{\mathbb{C}[s]} T$ to this s.e.s., consider the l.e.s of $\Ext_{\mathscr{D}_T}^\bullet(- , \mathscr{D}_T)$ modules, and side change to obtain the s.e.s of left $\mathscr{D}_{T}$-modules
    \begin{equation*}
        0 \to \Ext_{\mathscr{D}_T}^3(\mathscr{D}_T f^{s}, \mathscr{D}_T)^\ell \to \Ext_{\mathscr{D}_T}^3(\mathscr{D}_T f^{s+1}, \mathscr{D}_T)^\ell \to \Ext_{\mathscr{D}_T}^4( \frac{\mathscr{D}_T f^{s}}{\mathscr{D}_T f^{s+1}} , \mathscr{D}_T)^\ell \to 0 .
    \end{equation*}
    By Proposition \ref{prop - localized symmetry of fs}, this is the s.e.s. of left $\mathscr{D}_T$-modules
    \begin{equation*}
        0 \to \mathscr{D}_T f^{-s-1} \xrightarrow[]{\cdot f} \mathscr{D}_T f^{-s-2} \to \Ext_{\mathscr{D}_T}^4( \frac{\mathscr{D}_T f^{s}}{\mathscr{D}_T f^{s+1}}, \mathscr{D}_T)^\ell \to 0
    \end{equation*}
    where the first map is given by right multiplication with $f$. (Here is where we use our definition of $\Xi_{f}$: it ensures the vanishing of $S^{-1} L_{f, s+2}$ and $S^{-1} L_{f, s+1}$.) In other words,
    \begin{equation*} \label{eqn - thm - localized symmetry of BS poly, 3}
        \Ext_{\mathscr{D}_{T}}^{4} (\frac{\mathscr{D}_{T} f^{s}}{\mathscr{D}_{T} f^{s+1}}, \mathscr{D}_{T})^{\ell} \simeq \frac{\mathscr{D}_{T} f^{-s-2}}{\mathscr{D}_{T} f^{-s-1}}.
    \end{equation*}
    By Lemma \ref{lem - side changing b-functions} and Lemma \ref{lemma - b-function and b-function of dual} we deduce
    \begin{equation} \label{eqn - thm - localized symmetry of BS poly, 4}
    Z(B(\frac{\mathscr{D}_{T} f^{s}}{\mathscr{D}_{T} f^{s+1}})) = Z(B(\Ext_{\mathscr{D}_{T}}^{4} (\frac{\mathscr{D}_{T} f^{s}}{\mathscr{D}_{T} f^{s+1}}, \mathscr{D}_{T})^{\ell})) =  Z(B(\frac{\mathscr{D}_{T} f^{-s-2}}{\mathscr{D}_{T} f^{-s-1}}).
    \end{equation}
    And \eqref{eqn - thm - localized symmetry of BS poly, 4} implies \eqref{eqn - thm - localized symmetry of BS poly, 1}. Indeed, the $b$-function of the rightmost object of \eqref{eqn - thm - localized symmetry of BS poly, 4} is obtained from the $b$-function of the leftmost object of \eqref{eqn - thm - localized symmetry of BS poly, 4} by formally substituting $s \mapsto -s - 2$. A symmetric process demonstrates \eqref{eqn - thm - localized symmetry of BS poly, 1}.
\end{proof}

\subsection{Small Roots and the Logarithmic Comparison Theorem} We know that for $f \in R = \mathbb{C}[x_{1}, \dots, x_{n}]$ that (without our assumptions) $Z(b_{f}(s)) \subseteq (-n,0) \cap \mathbb{Q}$, \cite{KashiwaraRationalityRootsBSpolys}, \cite{SaitoOnMicrolocal}. Under our assumptions, the local Bernstein--Sato polynomials at $\mathfrak{x} \neq 0$ have roots contained in $(-2,0)$\footnote{Under our assumptions on $f \in \mathbb{C}[x_1,x_2,x_3]$, this follows from Remark \ref{rmk - log stratification}.(d) and Narv\'{a}ez Macarro's symmetry result for Bernstein--Sato polynomials of locally quasi-homogeneous free divisors \cite[Corollary 4.2]{DualityApproachSymmetryBSpolys}. Alternatively, $f$ is Saito-holonomic (Remark \ref{rmk - log stratification}.(c)) and we have the local analytic isomorphism \eqref{eqn - pos weighted homogeneous, away from origin, local product} at $\mathfrak{x} \neq 0$. This implies that the local Bernstein--Sato polynomial of $\Var(f)$ at $\mathfrak{x} \neq 0$ equals the Bernstein--Sato polynomial of an analytic divisor germ in $\mathbb{C}^2$.}; moreover these are readily computable by the locally quasi-homogeneous assumption, cf. \eqref{eqn - intro - BS poly, isolated singularity}.

So the difficulty in computing $Z(b_{f}(s))$ is determining which zeroes come from only from the local Bernstein--Sato polynomial at $0$. We cannot answer this completely. But any root in $Z(b_{f}(s)) \cap [-2,3)$ comes only from $b_{f,0}(s)$ and we can compute this completely:

\begin{corollary} \label{cor - small roots BS poly}
    Let $f = R = \mathbb{C}[x_{1}, x_{2}, x_{3}]$ be reduced and locally quasi-homogeneous. Then
    \begin{equation*}
        Z(b_{f}(s)) \cap (-3,-2] = \left\{ \frac{-(t + \sum w_{i})}{\wdeg(f)} \mid t \in \wdeg( H_{\mathfrak{m}}^{3} (R / (\partial f))) \right\} \cap (-3,-2].
    \end{equation*}
\end{corollary}

\begin{proof}
    The containment $\supseteq$ is Theorem \ref{thm - new roots of BS-poly}. For the reverse, take $\beta \in Z(b_{f}(s)) \cap (-3,-2]$. Since $- \beta - 2 \in [0,1)$ which is disjoint from $Z(b_{f}(s))$, we may use Theorem \ref{thm - localized symmetry of BS poly} to deduce $\beta \in \Xi_{f}$, where $\Xi_{f}$ is as defined in Theorem \ref{thm - localized symmetry of BS poly}. We will be done once we show that 
    \begin{equation} \label{eqn - cor - small roots BS poly, 1}
        \beta \notin \{ \frac{ -(t + \sum w_{i}) + \wdeg(f)}{\wdeg(f)} \mid t \in \wdeg( H_{\mathfrak{m}}^{0} (R / (\partial f))) \}.
    \end{equation} 
    Combining Theorem \ref{thm - new roots of BS-poly} and containment $Z(b_{f}(s)) \subseteq (-3,0)$, demonstrates
    \begin{equation*}
        \{\frac{-(t + \sum w_{i})}{\wdeg(f)} \mid t \in \wdeg(H_{\mathfrak{m}}^{0} (R / (\partial f))) \} \subseteq (-3,0).
    \end{equation*}
    So \eqref{eqn - cor - small roots BS poly, 1} is true since
    \begin{equation*}
        \{ \frac{-(t + \sum w_{i}) + \wdeg(f)}{\wdeg(f)} \mid t \in \wdeg(H_{\mathfrak{m}}^{0} (R / (\partial f))) \} \subseteq (-2, 1)
    \end{equation*}
\end{proof}

This lets us quickly derive an application to the existence of a twisted Logarithmic Comparison Theorem. We give a rough introduction to this problem; see \cite{BathSaitoTLCT} for a thorough treatment.

\begin{define} \label{def - twisted LCT}
    Let $D$ be an analytic divisor defined by $f \in \mathscr{O}_{X}$. Let $\lambda \in \mathbb{C}$. The \emph{logarithmic de Rham complex twisted by $\lambda$} is locally given by
    \[
    (\Omega_{X}^{\bullet}(\log f), \nabla^{\lambda}) = 0 \to \cdots \to \Omega_{X}^{p}(\log f) \xrightarrow[]{d(-) + \lambda df/f \wedge} \cdots \to 0.
    \]
    It is a subcomplex of the \emph{twisted meromorphic de Rham complex twisted by $\lambda$}
    \[
    (\Omega_{X}^{\bullet}(\star f), \nabla^{\lambda}) = 0 \to \cdots \to \Omega_{X}^{p}(\star f) \xrightarrow[]{d(-) + \lambda df/f \wedge} \cdots \to 0.
    \]
    We say the \emph{twisted Logarithmic Comparison Theorem with respect to $\lambda$} holds when the natural subcomplex inclusion is a quasi-isomorphism:
    \[
    (\Omega_{X}^{\bullet}(\log f), \nabla^{\lambda}) \xhookrightarrow{\qi} (\Omega_{X}^{\bullet}(\star f), \nabla^{\lambda}).
    \]
\end{define}

By Grothendieck and Deligne's Comparison Theorems, when the twisted Logarithmic Comparison Theorem holds, we have
    \[
        (\Omega_{X}^{\bullet}(\log f), \nabla^{\lambda}) \xhookrightarrow{\qi} (\Omega_{X}^{\bullet}(\star f), \nabla^{\lambda}) \simeq Rj_{\star} L_{\text{exp}(\lambda)}
    \]
    where $j : X \setminus D \to X$ is the inclusion and $L_{\text{exp}(\lambda)}$ is the rank one local system determined by the monodromy $\text{exp}(\lambda)$. 
    
    There has been a lot of work on which divisors admit a (un)twisted Logarithmic Comparison Theorem, see \cite{CohomologyComplementFreeDivisor}, \cite{HollandMondLCTIsolated}, \cite{WiensYuzvinskyLCTTameArrangements}, \cite{CalderonNarvaezFrenchLCT}, \cite{LCTforIntegrableLogarithmicConnections}, \cite{DualityApproachSymmetryBSpolys}, \cite{BathTLCT}, \cite{BathSaitoTLCT}. Reading \cite{BathSaitoTLCT} in the context of our main hypotheses, we quickly obtain the following characterization of whether or not a twisted Logarithmic Comparison Theorem w/r/t $\lambda$ holds; compare to \cite[Proposition 1]{BathSaitoTLCT}. 

\begin{corollary} \label{cor - twisted LCT}
    Let $f = R = \mathbb{C}[x_{1}, x_{2}, x_{3}]$ be reduced and locally quasi-homogeneous. Let $\lambda \in (-\infty, 0]$. Then the following are equivalent:
    \begin{enumerate}[label=(\alph*)]
        \item The twisted Logarithmic Comparison Theorem with respect to $\lambda$ holds,
        \item $- (\lambda - 2) \cdot \wdeg(f) - \sum w_{i} \notin \wdeg( H_{\mathfrak{m}}^{0}(R / (\partial f)))$.
    \end{enumerate}
\end{corollary}

\begin{proof}
    By \cite[Theorem 2]{BathSaitoTLCT} and the assumptions on $f$, (a) is equivalent to $(\lambda + \mathbb{Z}_{\leq -2}) \cap Z(b_{f}(s)) = \emptyset$. Because $Z(b_{f}(s)) \subseteq (-3,0)$ we have both: if $\lambda \in (\infty, -1]$ then (a) trivially holds; by Theorem \ref{thm - new roots of BS-poly}, it follows $\wdeg (H_{\mathfrak{m}}^{0} (R / (\partial f))) \subseteq (0, 3\wdeg(f) - \sum w_{i})$. Thus: if $\lambda \in (-\infty, -1]$ then (a) and (b) are equivalent. If $\lambda \in (-1,0]$, then $\lambda -2 \in (-3,2]$ and the equivalence of (a) and (b) follows from Corollary \ref{cor - small roots BS poly}. 
\end{proof}

\begin{remark}
    Corollary \ref{cor - twisted LCT} is consistent with \cite[Proposition 1]{BathSaitoTLCT} where true homogeneity of $f$ is also assumed. This follows from \cite[Theorem 4.7]{StratenDucoWarmtGorensteinDuality}, \cite[Theorem 3.4]{LocalCohomologyJacobianRing} which give degree symmetry of $H_{\mathfrak{m}}^{0} ( R / (\partial f))$ about $(3 \deg(f) - 6) / 2$, cf. \eqref{eqn - local cohomology degree duality}. 
\end{remark}

\subsection{Adding a True Homogeneity Assumption}

In this subsection we add an assumption of homogeneity, not just homogeneity with respect to a positive weight system, to our divisors. In this case the degree data of $H_{\mathfrak{m}}^{0}(R / (\partial f))$ has particularly nice properties: in \cite[Theorem 4.7]{StratenDucoWarmtGorensteinDuality}, \cite[Theorem 3.4]{LocalCohomologyJacobianRing} they show the nonvanishing degrees are symmetric about $(3 \deg(f) - 6)/2$, cf. \eqref{eqn - local cohomology degree duality}; in \cite[Theorem 4.1]{LefschetzAlmostCompleteIntersections}, \cite[Section 2]{BrennerKaidSyzygyBundles} they show $H_{\mathfrak{m}}^{0}(R / (\partial f))$ admits a Leftschetz element and, in particular, the Hilbert vector is unimodal. 

We can use these two properties to improve our analysis of $Z(b_{f}(s))$. First, the symmetry of degree data allows us to refine the numbers $\Xi_{f}$ appearing in Theorem \ref{thm - localized symmetry of BS poly}. Recall $\Xi_{f}$ is nothing more than $Z(B(L_{f, s+2})) \cup Z(B(L_{f, s+1}))$.

\begin{lemma} \label{lemma - homogeneous, involution}
    Let $f \in R = \mathbb{C}[x_{1}, x_{2}, x_{3}]$ be reduced and homogeneous; let $\sigma: \mathbb{R} \to \mathbb{R}$ be the involution $\alpha \mapsto -2 - \alpha$ about $-1$. Then
    \begin{align} \label{eqn - homogeneous, involution, statement}
       \{\frac{- (t + 3)}{\deg(f)} \mid t \in \deg H_{\mathfrak{m}}^{0} (R / (\partial f)) \}
       &= Z(B(L_{f, s+2})) \\
       &= \sigma \circ Z(B(L_{f, s+1})) \nonumber \\
       &= \sigma \circ \{\frac{- (t + 3) + \deg(f)}{\deg(f)} \mid t \in \deg H_{\mathfrak{m}}^{0} (R / (\partial f))\}. \nonumber
    \end{align}
\end{lemma}

\begin{proof}
    We can simplify the formula of Corollary \ref{cor - spectral sequence dual, twisted}/Proposition \ref{prop - funny b-function} for $B(L_{f,s+2})$:
    \begin{align*}
        B(L_{f,s+2}) 
        &= \prod_{t \in \deg H_{\mathfrak{m}}^{0} (R / (\partial f))} \bigg( s + 2 - \frac{-t + 2 \deg(f) - 3}{\deg(f)} \bigg)   \\
        &= \prod_{t \in \deg H_{\mathfrak{m}}^{0} (R / (\partial f))} \bigg( s - \frac{-t - 3}{\deg(f)} \bigg).
    \end{align*}
    Hence $Z(B(L_{f, s+2})) = \{\frac{- (t + 3)}{\deg(f)} \mid t \in \deg H_{\mathfrak{m}}^{0} (R / (\partial f)) \}$. That $Z(B(L_{f, s+1})) = \{\frac{- (t + 3) + \deg(f)}{\deg(f)} \mid t \in \deg H_{\mathfrak{m}}^{0} (R / (\partial f)) \}$ is similar. 

    As for the involution, we claim that
    \begin{align*}
        \sigma &\circ \bigg\{\frac{- (t + 3) + \deg(f)}{\deg(f)} \mid t \in \deg H_{\mathfrak{m}}^{0} (R / (\partial f)) \bigg\} \\
        &= \sigma \circ \bigg\{\frac{- (3 \deg(f) - 6 - t + 3) + \deg(f)}{\deg(f)} \mid t \in \deg H_{\mathfrak{m}}^{0} (R / (\partial f)) \bigg\} \\
        &= \sigma \circ \bigg\{ \frac{- (2 \deg(f) - t - 3)}{\deg(f)} \mid t \in \deg H_{\mathfrak{m}}^{0} (R / (\partial f)) \bigg\} \\
        &= \bigg\{ - 2 + \frac{2 \deg(f) - t - 3}{\deg(f)} \mid t \in \deg H_{\mathfrak{m}}^{0} (R / (\partial f)) \bigg\} \\
        &= \bigg\{ \frac{- (t + 3)}{\deg(f)} \mid t \in \deg H_{\mathfrak{m}}^{0} (R / (\partial f)) \bigg\}.
    \end{align*}
    The first equality here uses \cite[Theorem 4.7]{StratenDucoWarmtGorensteinDuality}, \cite[Theorem 3.4]{LocalCohomologyJacobianRing} which say $t \in \deg H_{\mathfrak{m}}^{0} (R / (\partial f)) \iff 3d - 6 - t \in \deg H_{\mathfrak{m}}^{0} (R / (\partial f))$; the rest are algebraic manipulations or the definition of $\sigma$. This completes the verification of \eqref{eqn - homogeneous, involution, statement}.
    \end{proof}

Lemma \ref{lemma - homogeneous, involution} says that that $\Xi_{f} = Z(B(L_{f, s+2})) \cup \sigma (Z(B(L_{f, s+1})))$, cf. Theorem \ref{thm - localized symmetry of BS poly}. Combined with the aforementioned unimodality property of $H_{\mathfrak{m}}^{0}(R / (\partial f))$, we have the following precise account of $Z(b_{f}(s))$. In short: $Z(b_{f}(s))$ is determined by $Z(b_{f}(s)) \cap (-1,0)$ and local cohomology data:

\begin{theorem} \label{thm - homogeneous, BS poly description}
    Let $f \in R = \mathbb{C}[x_{1}, x_{2}, x_{3}]$ be reduced, homogeneous, and locally quasi-homogeneous. With $\tau = \min \{t \mid t \in \deg H_{\mathfrak{m}}^{0} (R / (\partial f))\}$, set 
    \begin{equation*}
        \Upsilon_{f} = \frac{1}{\deg(f)} \cdot \left( \mathbb{Z} \cap [- 3 \deg(f) + \tau + 3, - (\tau + 3)] \right).
    \end{equation*}
    Moreover, let $\sigma : \mathbb{R} \to \mathbb{R}$ the involution $\alpha \mapsto - 2 - \alpha$ about $-1$. Then
    \begin{enumerate}[label=(\alph*)]
        \item $ \Upsilon_{f} \subseteq Z(b_{f}(s))$
        \item $\beta \in Z(b_{f}(s)) \setminus (\Upsilon_{f} \cup \sigma(\Upsilon_{f})) \iff \sigma (\beta) \in Z(b_{f}(s)) \setminus (\Upsilon_{f} \cup \sigma(\Upsilon_{f})) $
    \end{enumerate}
   And we have the taxonomy:
    \begin{align} \label{eqn - thm - homogeneous, BS poly description, statement 2}
        \beta \in Z(b_{f}(s)) \cap (-3, -2] \iff &\beta \in \Upsilon_{f} \cap (-3,-2]; \\
        \beta \in Z(b_{f}(s)) \cap (-2, -1) \iff &\beta \in \Upsilon_{f} \cap (-2, -1) \nonumber \\
        \enspace &\text{ and/or } \sigma(\beta) \in Z(b_{f}(s)) \cap (-1, 0); \nonumber \\
        \beta \in Z(b_{f}(s)) \cap [-1, 0) \implies &\sigma(\beta) \in Z(b_{f}(s)) \cap (-2,-1]. \nonumber 
    \end{align}
     Hence $Z(b_{f}(s))$ is determined by
     \[
     \tau, \deg(f), \text{ and } \big[ Z(b_{f}(s)) \cap [-1,0) \big] \setminus \Upsilon_{f} \quad  (\Upsilon_{f} \text{ is determined by $\tau$ and $\deg(f)$}.)
     \]
\end{theorem}

\begin{proof}
    By \cite[Theorem 4.1]{LefschetzAlmostCompleteIntersections}, \cite[Section 2]{BrennerKaidSyzygyBundles}, which shows the unimodality of the Hilbert vector $H_{\mathfrak{m}}^{0} (R / (\partial f))$ along with \cite[Theorem 4.7]{StratenDucoWarmtGorensteinDuality}, \cite[Theorem 3.4]{LocalCohomologyJacobianRing}, which shows said Hilbert vector's symmetry about $(3 \deg(f) - 6)/2$, we have 
    \begin{equation} \label{eqn - thm, homogeneous, BS poly description, 1}
        \deg H_{\mathfrak{m}}^{0} (R / (\partial f)) = \mathbb{Z} \cap [\tau, 3 \deg(f) - 6 - \tau].
    \end{equation}
    Then (a) is a repackaging of the last part of Theorem \ref{thm - new roots of BS-poly} and \eqref{eqn - thm, homogeneous, BS poly description, 1}. A similar repackaging, but now also using Lemma \ref{lemma - homogeneous, involution}, applied to Theorem \ref{thm - localized symmetry of BS poly} yields (b).
    Note that the symmetry of \eqref{eqn - thm, homogeneous, BS poly description, 1} about $\lfloor (3d - 6 - t)/2 \rfloor$ induces a symmetry of $\Upsilon_{f}$ about $-3/2$ and of $\sigma (\Upsilon_{f})$ about $-1/2$. 

    Now we prove \eqref{eqn - thm - homogeneous, BS poly description, statement 2}. The $(-3,-2]$ assertion is just Corollary \ref{cor - small roots BS poly}. 

    Regarding the $[-1,0)$ claim, select $\beta \in Z(b_{f}(s)) \cap [-1,0)$. We may assume $\beta \in \Upsilon_{f} \cup \sigma(\Upsilon_{f})$ lest (b) applies. If $\beta \in \sigma(\Upsilon_{f})$, then $\sigma(\beta) \in \Upsilon_{f}$ and (a) applies. If $\beta \in \Upsilon_{f}$, then because $\deg(f) \Upsilon_{f}$ is symmetric about $-3/2$ and unimodal (cf. \eqref{eqn - thm, homogeneous, BS poly description, 1}) and because $\beta \in (-1,0)$, it must be that 
    \[
    \Upsilon_{f} \supseteq \frac{1}{\deg(f)} \cdot (\mathbb{Z} \cap [-2,-1]).
    \]
    In particular, $\sigma(\beta) \in (1/\deg(f)) \cdot (\mathbb{Z} \cap [-2,-1])$ and (a) applies.

    Regarding the $(-2,1)$ claim, we first prove ``$\implies$''. Assume that $\beta \in Z(b_{f}(s)) \cap (-2,-1).$ If $\sigma(\beta) \in Z(b_{f}(s)) \cap (-1,0)$ we are done. If not, then (b) implies $\sigma(\beta) \in \Upsilon_{f} \cup \sigma(\Upsilon_{f})$. If $\sigma(\beta) \in \Upsilon_{f}$ then $\sigma(\beta) \in Z(b_{f}(s)) \cap (-1,0)$ by (a); if $\sigma(\beta) \in \sigma(\Upsilon_{f})$, then $\beta \in \Upsilon_{f}.$ As for the``$\impliedby$'' direction, this immediately follows by (a) and the $[-1,0)$ claim we already justified.

    Only the last sentence remains. Let $\beta \in Z(b_{f}(s))$. If $\beta \in [-1,0)$ it is either in $\Upsilon_{f} \subseteq Z(b_{f}(s))$ or $Z(b_{f}(s)) \setminus \Upsilon_{f}$ by (a). If $\beta \in (-2,-1)$, then \eqref{eqn - thm - homogeneous, BS poly description, statement 2} tells us $\beta \in \Upsilon_{f}$ or $\sigma(\beta) \in Z(b_{f}(s)) \cap (-1,0)$. In the former case we are done; in the latter case the membership ``$\sigma(\beta) \in Z(b_{f}(s)) \cap (-1,0)$'' is detected by the ``determined by'' data and if membership holds we detect $\beta \in Z(b_{f}(s))$ using the last part of \eqref{eqn - thm - homogeneous, BS poly description, statement 2}. If $\beta \in (-3,2)$, then \eqref{eqn - thm - homogeneous, BS poly description, statement 2} says this is equivalent to $\beta \in \Upsilon_{f}$.
\end{proof}

\subsection{Hyperplane Arrangements}

Now we restrict our class of divisors further to hyperplane arrangements $\mathscr{A} \subseteq \mathbb{C}^{3}$. First, we work in $\mathbb{C}^{n}$. A \emph{hyperplane arrangement} $\mathscr{A} \subseteq \mathbb{C}^{n}$ is a collection of hyperplanes $H_{k} \subseteq \mathbb{C}^{n}$ so that $\mathscr{A} = \cup_{k} H_{k}$. The \emph{degree} of $\mathscr{A}$, denoted $\deg(\mathscr{A})$, is the total number of hyperplanes $H_{k}$. A \emph{flat} (or \emph{edge}) of $\mathscr{A}$ is the intersection of some subset of the $d$ hyperplanes; the \emph{intersection lattice} $\mathscr{L}(\mathscr{A})$ is the poset with elements these flats, ordered by reverse inclusion. A property of $\mathscr{A}$ is \emph{combinatorial} when it is a function of the intersection lattice: if $\mathscr{A}$ has ``property'' and $\mathscr{B}$ is another arrangement with $\mathscr{L}(\mathscr{A}) \simeq \mathscr{L}(\mathscr{B})$, then $\mathscr{B}$ has ``property.'' Note that the number of hyperplanes containing a given flat is combinatorial.

Let $f \in R = \mathbb{C}[x_{1}, \dots , x_{n}]$ define $\mathscr{A}$. We say $\mathscr{A}$ is \emph{central} when each $H_{k}$ contains the origin (i.e. $f$ is homogeneous); we say $\mathscr{A}$ is \emph{essential} when $\cap_{k} H_{k} = \{0\}$; we say $\mathscr{A}$ is \emph{indecomposable} when there exists no coordinate system so that $\mathscr{A}$ admits a defining equation $f = gh$ where $g$ and $h$ use disjoint variables. 

From the point of view of Bernstein--Sato polynomials it is harmless to assume that $\mathscr{A}$ is central, essential, and indecomposable by using the fact the Bernstein--Sato polynomial is the least common multiple of all the local Bernstein--Sato polynomials. Note also that an arrangement is locally quasi-homogeneous, since locally the arrangement is just a subarrangement. 

Now restrict again to $\mathbb{C}^{3}$. It is natural to ask if the zeroes of the Bernstein--Sato polynomial of $f \in R = \mathbb{C}[x_{1}, x_{2}, x_{3}]$ defining $\mathscr{A} \subseteq \mathbb{C}^{3}$ are combinatorial. Walther \cite[Example 5.10]{uli} showed this is not the case: $-2 + (2 / \deg(f))$ is a not necessarily combinatorially detected root. (See Example \ref{ex - Ziegler's pair}). Using our results we will prove this is the \emph{only} non-combinatorial zero. Theorem \ref{thm - BS poly arrangement}: gives a complete formula for $Z(b_{f}(s))$; explicitly describes the (easy) combinatorial roots; shows the only possibly non-combinatorial root is $-2 + (2 / \deg(f))$; gives several equivalent characterizations of when $-2 + (2 / \deg(f)) \in Z(b_{f}(s))$. 

We conduct some preliminaries for the sake of stating the aforementioned equivalent conditions. First, recall the \emph{Castelnuovo--Mumford} regularity of a graded $R$-module $M$ is
\[
\reg(M) = \max_{0 \leq t \leq n} \{ (\maxdeg H_{\mathfrak{m}}^{t} M ) + t \}.
\]
In the case of a reduced, central, essential, and indecomposable hyperplane arrangement in $\mathbb{C}^{3}$, Schenck \cite[Corollary 3.5]{SchenckElementaryModificationsLineConfigs} produced a sharp bound for the Castelnuovo--Mumford regularity of $\Der_{R}(-\log \mathscr{A})$. (Using techniques from \cite{DerksenSidman-CMRegularyByApproximation}, this bound has been generalized \cite[Proposition 1.3]{SaitoDegenerationPoleOrderArrangement} to logarithmic derivations of $\mathscr{A} \subseteq \mathbb{C}^{n}$ and then \cite[Theorem 1.7]{BathTLCT} to logarithmic $p$-forms of $\mathscr{A} \subseteq \mathbb{C}^{n}$.). As we have done throughout the paper, since for us $R = \mathbb{C}[x_{1}, x_{2}, x_{3}]$, we reinterpret this bound in terms of the Milnor algebra:

\begin{lemma} \label{lemma - local cohomology milnor alg for arrangements}
    Let $f \in R = \mathbb{C}[x_{1}, x_{2}, x_{3}]$ define a reduced, central, essential, and indecomposable hyperplane arrangement. Then there is a graded isomorphism of $R$-modules $H_{\mathfrak{m}}^{0}(R / (\partial f)) \simeq [H_{\mathfrak{m}}^{2}(\Der_{R}(-\log_{0} f))](-\deg(f))$ as well as graded $\mathbb{C}$-vector space isomorphisms $[H_{\mathfrak{m}}^{1}(R / (\partial f))]_{k} \simeq [H_{\mathfrak{m}}^{3}(\Der_{R}(-\log_{0} f))]_{k - d}$ for all $k \geq d - 3.$ Moreover,
    \[
    \reg R / (\partial f) \leq 2 \deg(f) - 5.
    \]
\end{lemma}

\begin{proof}
    Similar to Proposition \ref{prop - weighted data dictionary, ext log 1 forms, local cohomology milnor algebra}, we have graded s.e.s.
    \begin{equation*}
        0 \to \Der_{R}(-\log_{0} f)(-\deg(f)) \to \Der_{R}(-\deg(f)) \to (\partial f) \to 0
    \end{equation*}
    and
    \begin{equation*}
        0 \to (\partial f) \to R \to R / (\partial f) \to 0.
    \end{equation*}
    By standard manipulations of the l.e.s. of local cohomology we get the graded isomorphism promised. We also get the graded exact sequence $H_{\mathfrak{m}}^{1}(R / (\partial f)) \xhookrightarrow{} [H_{\mathfrak{m}}^{3}(\Der_{R}(-\log_{0})f)](-\deg(f)) \to H_{\mathfrak{m}}^{3}(\Der_{R})$. The graded $\mathbb{C}$-vector space isomorphisms follow since $H_{\mathfrak{m}}^{3}(\Der_{R})_{\geq -3} = 0$.

    The regularity bound is now a corollary of the Castelnuovo--Mumford regularity bound $\reg \Der_{R}(-\log f) \leq \deg(f) - 3$ \cite[Corollary 3.5]{SchenckElementaryModificationsLineConfigs}. (Pay attention to our different grading conventions.) This implies the same regularity bound for $\Der_{R}(-\log_{0} f)$.
\end{proof}

Second we recall the property of $\emph{formality}$, well studied amongst the arrangement community, cf. \cite{FalkRandellHomotopyTheoryArrangements}, \cite{BrandtTeraoFreeArrangementsRelationSpaces}, \cite{TohvaneanuTopologicalFormalArrangements}, \cite{DipasqualeSidmanTravesGeometricAspectsJacobian}. 

\begin{define} \label{def - formality}
    Let $\mathscr{A} = \cup \{H_{k} \} \subseteq \mathbb{C}^{n}$ by a central, reduced arrangement and select a linear $f_{k} \in R = \mathbb{C}[x_{1}, \dots, x_{n}]$ such that $\Var(f_{k}) = H_{k}$. Writing $f_{k} = a_{1,k} x_{1} + \dots a_{n,k} x_{n}$, form a $n \times \deg(\mathscr{A})$ matrix whose $k^{\text{th}}$ column is the normal vector of $f_{k}$:
    \[
    M(\{f_{k}\}) = \begin{bmatrix}
        a_{i, k}
    \end{bmatrix}
    \]
    The \emph{relation space} $\mathscr{R}(\mathscr{A})$ is the kernel of $M(\{f_{k}\})$:
    \[
    \mathscr{R}(\mathscr{A}) = \{ \mathbf{v} \in \mathbb{C}^{\deg(\mathscr{A})} \mid M(\{f_{k}\}) \mathbf{v} = 0\}.
    \]
    The \emph{length} of $\mathbf{v} \in \mathscr{R}(\mathscr{A})$ is the number of non-zero entries of $\mathbf{v}$. We say $\mathscr{A}$ is \emph{formal} if $\mathscr{R}(\mathscr{A})$ is generated by relations of length $3$. 
\end{define}

Now we are in position to state our formula for the Bernstein--Sato polynomial of \emph{any} reduced, central, essential, and indecomposable arrangement $\mathscr{A} \subseteq \mathbb{C}^{3}$. This relies on an earlier result \cite[Theorem 1.3]{BathCombinatoriallyDetermined} of ours: if $f$ defines such an arrangement, then $Z(b_{f}(s)) \cap [-1, 0)$ \emph{is combinatorially determined} by an explicit formula. (Actually this result is more general, cf. loc. cit.) 

\begin{theorem} \label{thm - BS poly arrangement}
    Let $f \in R = \mathbb{C}[x_{1}, x_{2}, x_{3}]$ define a reduced, central, essential, and indecomposable hyperplane arrangement; let $Z \subseteq \mathbb{P}^{2}$ be its projectivization. For every point $z \in \Sing(Z)$, set $m_{z}$ to be the number of lines in $Z$ containing $z$. Define
    \[
    \CombRoots = \left[ \bigcup_{3 \leq k \leq 2 \deg(f) - 3}  \frac{-k}{\deg(f)}  \right] \cup \left[ \bigcup_{z \in \Sing(Z)} \quad \bigcup_{2 \leq i \leq 2m_{z} - 2}  \frac{-i}{m_{z}}  \right].
    \]
    Then
    \begin{equation} \label{eqn - thm - BS poly arrangement, statement 1}
        Z(b_{f}(s)) = \CombRoots \quad \text{OR} \quad Z(b_{f}(s)) = \CombRoots \enspace \cup \enspace \frac{-2\deg(f) + 2}{\deg(f)}.
    \end{equation}
    That is, $Z(b_{f}(s))$ always contains the combinatorially determined roots $\CombRoots$ and at most one non-combinatorial root $\frac{-2\deg(f) + 2}{\deg(f)}$. 
    
    The presence of the non-combinatorial root is characterized by the following equivalent properties, where $\widetilde{\frac{R}{(\partial f)}}$ is the graded sheafification of $R / (\partial f)$):
    \begin{enumerate}[label=(\alph*)]
        \item $Z(b_{f}(s)) \ni \frac{-2 \deg(f) + 2}{\deg(f)}$;
        \item $[H_{\mathfrak{m}}^{0}(R / (\partial f)]_{\deg(f) - 1} \neq 0$;
        \item $[H_{\mathfrak{m}}^{0}(R / (\partial f))]_{2 \deg(f) - 5} \neq 0$;
        \item $\reg R / (\partial f) = 2 \deg(f) - 5$;
        \item there is the bound on the number of global sections twisted by $2 \deg(f) - 5$:
        \[
        \dim_{\mathbb{C}} \bigg( \Gamma \big( \mathbb{P}^{2}, \widetilde{ \frac{R}{ (\partial f)}} (2 \deg(f) - 5) \big) \bigg) < \dim_{\mathbb{C}} \bigg( [ R / (\partial f)]_{2d - 5} \bigg);
        \]
        \item there is the bound on the number of global sections twisted by $\deg(f) - 1$:
        \[
        \dim_{\mathbb{C}} \bigg( \Gamma \big( \mathbb{P}^{2}, \widetilde{ \frac{R}{ (\partial f)}} (\deg(f) - 1) \big) \bigg) < \dim_{\mathbb{C}} \bigg( [\Der_{R}(-\log_{0} f)]_{\deg(f) -2} \bigg) + \bigg(\binom{\deg(f) + 1}{2} - 3 \bigg) 
        \]
        \item the hyperplane arrangement $\Var(f)$ is not formal. 
    \end{enumerate}
\end{theorem}

\begin{proof}
    By \cite[Theorem 1.3]{BathCombinatoriallyDetermined}, $Z(b_{f}(s)) \cap [-1,0) = \CombRoots \cap [-1,0).$ By Theorem \ref{thm - homogeneous, BS poly description} we know that $\sigma(\CombRoots \cap [-1,0)) \subseteq Z(b_{f}(s))$, where $\sigma: \mathbb{R} \to \mathbb{R}$ is the involution $\alpha \mapsto -2 - \alpha$. Since $\CombRoots = (\CombRoots \cap [-1,0)) \sqcup \sigma(\CombRoots \cap [-1,0))$, we deduce $\CombRoots \subseteq Z(b_{f}(s))$. Using Theorem \ref{thm - homogeneous, BS poly description}, in particular the first and second ``$\iff$'' of \eqref{eqn - thm - homogeneous, BS poly description, statement 2}, we deduce that
    \begin{align} \label{eqn - thm - BS poly arrangement 1}
    Z(b_{f}(s)) \setminus \CombRoots 
    = \bigg\{ \frac{-(t + 3)}{\deg(f)} \mid t \in \deg H_{\mathfrak{m}}^{0} (R / (\partial f)) \bigg\} \setminus \CombRoots . 
    \end{align}
    (Recall that in Theorem \ref{thm - homogeneous, BS poly description}, $\Upsilon_f = \{\frac{-(t + 3)}{\deg(f)} \mid t \in \deg H_{\mathfrak{m}}^{0} (R / (\partial f))\}$.) By the regularity bound of Lemma \ref{lemma - local cohomology milnor alg for arrangements} along with the symmetry \cite[Theorem 4.7]{StratenDucoWarmtGorensteinDuality}, \cite[Theorem 3.4]{LocalCohomologyJacobianRing} of $\deg H_{\mathfrak{m}}^{0} (R / (\partial f))$,
    \[
    \bigg\{ \frac{-(t + 3)}{\deg(f)} \mid t \in \deg H_{\mathfrak{m}}^{3} (R / (\partial f)) \bigg\} \subseteq \frac{1}{\deg(f)} \cdot \mathbb{Z} \cap [-2 \deg(f) +2, - \deg(f) - 2]
    \]
    We then determine \eqref{eqn - thm - BS poly arrangement 1} as:
    \begin{equation} \label{eqn - thm BS poly arrangement 2}
        Z(b_{f}(s)) \setminus \CombRoots = \begin{cases}
            \frac{-2 \deg(f) + 2}{\deg(f)} \iff [H_{\mathfrak{m}}^{0}(R / (\partial f))]_{2\deg(f) - 5} \neq 0 \\
            \emptyset \quad \quad \quad \text{ otherwise }.
        \end{cases}
    \end{equation}

    We have proven \eqref{eqn - thm - BS poly arrangement, statement 1}. The subsequent sentence is immediate by the description of $\CombRoots$ and Walther's \cite[Example 5.10]{uli}. We have also proven the equivalence of (a) and (c). The equivalence of (b) and (c) is the oft-used symmetry property \cite[Theorem 4.7]{StratenDucoWarmtGorensteinDuality}, \cite[Theorem 3.4]{LocalCohomologyJacobianRing}. As for (e) and (f), recall the canonical exact sequence of $\mathbb{C}$-vector spaces
    \begin{align} \label{eqn - thm - BS poly arrangement, 3}
    0 \to [H_{\mathfrak{m}}^{0}(R / (\partial f))]_{q} 
    \to [R / (\partial f)]_{q} \to \Gamma \big( \mathbb{P}^{2},  \widetilde{\frac{R}{\partial f}} (q) \big) \to [H_{\mathfrak{m}}^{1}(R / (\partial f))]_{q} \to 0
    \end{align}
    for all $q \in \mathbb{Z}$.
    Then \eqref{eqn - thm - BS poly arrangement, 3} plus the regularity bound within Lemma \ref{lemma - local cohomology milnor alg for arrangements}, makes it clear that (e) is equivalent to (c). As for (f), we have $\mathbb{C}$-isomorphisms 
    \[
    [H_{\mathfrak{m}}^{1}(R / (\partial f))]_{\deg(f)-1} \simeq [H_{\mathfrak{m}}^{3}(\Der_{R}(-\log_{0} f))]_{-1} \simeq [H_{\mathfrak{m}}^{3}(\Der_{R}(-\log f))]_{-1}.
    \]
    (The first isomorphism is Lemma \ref{lemma - local cohomology milnor alg for arrangements}; the second the direct sum structure of logarithmic derivations.)  Using graded local duality and the canonical graded perfect pairing between $\Der_{R}(-\log f)$ and $\Omega_{R}^{1}(\log f)$ we obtain a $\mathbb{C}$-vector space isomorphism $[H_{\mathfrak{m}}^{3}(\Der_{R}(-\log f))]_{-1} \simeq [\Omega_{R}^{1}(\log f)]_{-2}.$ Moreover, $[\Omega_{R}^{1}(\log f)]_{-2} = [\iota_{E}(\Omega_{R}^{2}(\log_{0} f))]_{-2} \simeq [\Der_{R}(-\log_{0} f)(\deg(f)]_{-2}$, cf. Remark \ref{rmk - basics alg log forms}, \eqref{eqn - graded iso log 2 forms, log der}). Applying these findings to \eqref{eqn - thm - BS poly arrangement, 3} reveals that (b) is equivalent to 
    \begin{align*}
   \dim_{\mathbb{C}} \bigg( \Gamma \big( \mathbb{P}^{2}, \widetilde{ \frac{R}{ (\partial f)}} (\deg(f) - 1) \big) \bigg) 
   &< \dim_{\mathbb{C}} \bigg( [R / (\partial f)]_{\deg(f)-1} \bigg) + \dim_{\mathbb{C}} \bigg( \Der_{R}(-\log_{0})_{\deg(f) - 2} \bigg) \\
   &=  \dim_{\mathbb{C}} \bigg( R_{\deg(f)-1} \bigg) - 3 + \dim_{\mathbb{C}} \bigg( \Der_{R}(-\log_{0})_{\deg(f) - 2} \bigg) \\
    \end{align*}
    (The ``$=$'' is because $\partial_{1} \bullet f, \partial_{2} \bullet f, \partial_{3} \bullet f$ must be linearly independent over $\mathbb{C}$, lest indecomposability fail.) That is, (b) is equivalent to (f). 

    Only (d) and (g) are extant. Lemma \ref{lemma - local cohomology milnor alg for arrangements} demonstrates that $\reg R / (\partial f) = 2 \deg(f) - 5$ if and only if $(\reg \Der_{R}(-\log f)) = \deg(f) - 3$. In \cite[Corollary 7.2, Corollary 7.8]{DipasqualeSidmanTravesGeometricAspectsJacobian}, (once you account for our different grading conventions) it is shown that $\reg \Der_{R}(-\log f) = \reg \Der_{R}(-\log_{0} f) = \deg(f) - 3$ if and only if $[H_{\mathfrak{m}}^{2}(\Der_{R}(-\log_{0} f))]_{\deg(f) - 5} \neq 0$. Loc. cit. also proves this nonvanishing is equivalent to the non-formality of $\Var(f)$. This certifies the pairwise equivalences of (c), (d), and (g).
\end{proof}

\begin{remark}
\noindent
    \begin{enumerate}[label=(\alph*)]
        \item Theorem \ref{thm - BS poly arrangement} answers \cite[Question 8.8]{DipasqualeSidmanTravesGeometricAspectsJacobian} in the case of $\mathbb{C}^{3}$ arrangements.
        \item Let $f$ define a generic arrangement in $\mathbb{C}^{3}$. Theorem \ref{thm - BS poly arrangement} can be used to recover Walther's formula \cite{WaltherBSPolyGenericArrangement} for $Z(b_{f}(s))$ because the graded free resolution of $\Der_{R}(-\log f)$ is known for logarithmic derivations, cf. \cite{RoseTeraoAFreeResolution}. Of course, Walther's formula holds for generic arrangements in $\mathbb{C}^{n}$.
    \end{enumerate}
\end{remark}

\begin{example} \label{ex - Ziegler's pair}
    We consider the original example of Ziegler's pair \cite{ZieglerCombinatorialConstructionLogForms}: Walther \cite[Example 5.10]{uli} showed these arrangements have different Bernstein--Sato polynomial roots, but the same combinatorics. Let $f = xyz(x + 3z)(x + y + z)(x + 2y + 3z)(2x + y + z)(2x + 3y + z)(2x + 3y + 4z).$ and $g = xyz(x + 5z)(x + y + z)(x + 3y + 5z)(2x + y + z)(2x + 3y + z)(2x + 3y + 4z).$ Here $\Var(f)$ and $\Var(g)$ have the same intersection lattice, but (working in $\mathbb{P}^{2}$) the six double points of $\Var(f)$ lie on a quadric whereas those of $\Var(g)$ do not. M. Saito \cite[Remark 4.14(iv)]{SaitoBSpolysArrangements} later computed their Bernstein--Sato polynomials, showing they are the same except $-2 + (2/9) \in Z(b_{f})(s)$ whereas $- 2 + (2/9) \notin Z(b_{g}(s))$. This demonstrates Theorem \ref{thm - BS poly arrangement}. 
    
    Invoking Macaulay2 we compute: $H_{\mathfrak{m}}^{0}(R / (\partial f))_{13} \simeq H_{\mathfrak{m}}^{0}(R / (\partial f))_{8} \simeq \mathbb{C}$  whereas  $H_{\mathfrak{m}}^{0}(R / (\partial g))_{13} = H_{\mathfrak{m}}^{0}(R / (\partial g))_{8} = 0$. So we see (b) and (c) in Theorem \ref{thm - BS poly arrangement}. We can also see (e): using \eqref{eqn - thm - BS poly arrangement, 3} and the regularity bound in Lemma \ref{lemma - local cohomology milnor alg for arrangements},   
    \begin{equation*}
        42 = \dim_{\mathbb{C}} \bigg( \Gamma \big(\mathbb{P}^{2}, \widetilde{\frac{R}{(\partial f)}}(13) \big) \bigg) < \dim_{\mathbb{C}} \bigg( [R / (\partial f)]_{13} \bigg) = 43
    \end{equation*}
    \begin{equation*}
        42 = \dim_{\mathbb{C}} \bigg( \Gamma \big(\mathbb{P}^{2}, \widetilde{\frac{R}{(\partial g)}}(13) \big) \bigg) = \dim_{\mathbb{C}} \bigg( [R / (\partial g)]_{13} \bigg) = 42
    \end{equation*}
    By \eqref{eqn - thm - BS poly arrangement, 3} and the $\mathbb{C}$-isomorphism $H_{\mathfrak{m}}^{1}(R / (\partial f))_{\deg(f) - 1} \simeq \Der_{R}(-\log_{0} f)_{\deg(f) - 2}$ justified nearby, we see (f) in Theorem \ref{thm - BS poly arrangement}:
    \begin{equation*}
        65 = \dim_{\mathbb{C}} \bigg( \Gamma \big( \mathbb{P}^{2}, \widetilde{\frac{R}{(\partial f)}}(8) \big) \bigg) < \dim_{\mathbb{C}} \bigg( \Der_{R}(-\log_{0} f)_{7} \bigg) + \bigg( \binom{10}{2} - 3 \bigg) = 24 + 42 = 66
    \end{equation*}
    \begin{equation*}
        66 = \dim_{\mathbb{C}} \bigg( \Gamma \big( \mathbb{P}^{2}, \widetilde{\frac{R}{(\partial g)}}(8) \big) \bigg) = \dim_{\mathbb{C}} \bigg( \Der_{R}(-\log_{0} g)_{7} \bigg) +  \bigg( \binom{10}{2} - 3 \bigg) = 24 + 42 = 66.
    \end{equation*}
\end{example}

\bibliographystyle{alpha}
\bibliography{refs}

\end{document}